%% file: ms.tex
\documentclass[onefignum,onetabnum,pdftex]{siamart171218}

\usepackage[geometry]{ifsym}
\usepackage{anysize}
\usepackage{float}
\usepackage{caption}
\usepackage{psfrag}
\usepackage{color}
\usepackage{amsfonts}
\usepackage{graphicx,subfigure}
\usepackage{amssymb}
\usepackage{multirow}
\usepackage{subfigure}
\usepackage[normalem]{ulem} 
\usepackage{cancel} 
\usepackage{enumitem}

\usepackage{epstopdf}
\usepackage{tikz}
\usepackage{pgfplots}
\usepgfplotslibrary{external}
\graphicspath{{FIGURES/}{figures/}}

\def\qed{\hbox{${\vcenter{\vbox{                        
   \hrule height 0.4pt\hbox{\vrule width 0.4pt height 6pt
   \kern5pt\vrule width 0.4pt}\hrule height 0.4pt}}}$}}

\newtheorem{remark}{Remark}[section]

\def\RR{{\mathop{{\rm I}\kern-.2em{\rm R}}\nolimits}}
\def\PP{{\mathop{{\rm I}\kern-.2em{\rm P}}\nolimits}}

\def\NN{{\mathop{{\rm I}\kern-.2em{\rm N}}\nolimits}}

\newcommand{\be}{\begin{equation}}
\newcommand{\ee}{\end{equation}}
\newcommand{\ba}{\begin{eqnarray}}
\newcommand{\ea}{\end{eqnarray}}
\newcommand{\bi}{\begin{itemize}}
\newcommand{\ei}{\end{itemize}}

\newcommand{\supp}{\mathop{\mathrm{supp}}}
\newcommand{\esupp}{\mathop{\mathrm{esupp}}}

\newcommand{\trunc}{\mathop{\mathrm{trunc}}}
\newcommand{\Trunc}{\mathop{\mathrm{Trunc}}}

\newcommand\bp{{\bf p}}
\newcommand\bzeta{\boldsymbol \zeta}

\def\lec{\lesssim}

\definecolor{blue}{rgb}{0,0,0.99}

\definecolor{dred}{rgb}{0.92,0,0}

\definecolor{dsky}{rgb}{0.1,1,1}

\newcommand\Aone{\ref{assumption_decomposition}}
\newcommand\Atwo{\ref{assumption_scs}}

\let\tilde\widetilde

\title{BPX preconditioners for isogeometric analysis \\ using (truncated) hierarchical B-splines \thanks{Submitted to the editors DATE.}
\funding{The work of D.C. was partially supported by Basic Science Research Program through the National Research Foundation of Korea (NRF) funded by the Ministry of Education (2018R1D1A1B07048773). The work of R.V. was partially supported by the ERC Advanced Grant “CHANGE” (694515, 2016-2020). The work of C.B. and C.G. have been partially supported by INdAM through GNCS and Finanziamenti Premiali SUNRISE. C.B., C.G. and R.V. are members of the INdAM Research group GNCS.
}}

\author{Cesare Bracco\thanks{Department of Mathematics and Computer Science, University of Florence, Viale Morgagni 67/A, 50134, Florence, Italy (\email{cesare.bracco@unifi.it}, \email{carlotta.giannelli@unifi.it})}
\and Durkbin Cho\thanks{Department of Mathematics, Dongguk University, Pil-dong 3-ga, Jung-gu, Seoul, 04620, South Korea, (\email{durkbin@dongguk.edu})}
\and Carlotta Giannelli\footnotemark[1]
\and Rafael V\'azquez\thanks{Chair of Modelling and Numerical Simulation, Institute of Mathematics, \'Ecole Polytechnique F\'ed\'erale de Lausanne, Station 8, 1015, Lausanne, Switzerland, and Istituto di Matematica Applicata e Tecnologie Informatiche "Enrico Magenes" del CNR, Pavia, Italy. (\email{rafael.vazquez@epfl.ch})}
}

\begin{document}


\maketitle

\begin{abstract}
We present the construction of additive multilevel preconditioners, also known as
BPX preconditioners, for the solution of the linear system arising in isogeometric
adaptive schemes with (truncated) hierarchical B-splines. We show that the locality
of hierarchical spline functions, naturally defined on a multilevel structure, can be
suitably exploited to design and analyze efficient multilevel decompositions. By
obtaining smaller subspaces with respect to standard tensor-product B-splines,
the computational effort on each level is reduced. We prove that, for suitably graded
hierarchical meshes, the condition number of the preconditioned system is bounded
independently of the number of levels. A selection of numerical examples validates
the theoretical results and the performance of the preconditioner.
\end{abstract}



\section{Introduction}
\input{introduction.tex}

\section{Splines} \label{sec:splines}
\input{splines_3_shorter.tex}

\section{The additive multilevel preconditioner} \label{sec:preliminaries}
\input{preliminaries_3.tex}

\section{Analysis of the BPX preconditioner for THB-splines} \label{sec:decompositions}
\input{decomposition_3.tex}

\section{Numerical tests} \label{sec:numerical}
\input{numerical.tex}


\bibliographystyle{siam}
\bibliography{biblio_3}

\end{document}

%% file: introduction.tex
The use of splines for the solution of partial differential equations has gained popularity with the introduction of isogeometric analysis (IGA) in \cite{Hughes_Cottrell_Bazilevs}. One of the most active topics in recent years in IGA has been the development and analysis of adaptive methods. Hierarchical B-splines (HB-splines) \cite{Vuong_giannelli_juttler_simeon} and truncated hierarchical B-splines (THB-splines) \cite{Giannelli2012485}, which are defined from a multilevel structure, are among the most promising constructions of splines with local refinement capabilities. Indeed, adaptive methods with (T)HB-splines have appeared in the engineering literature (see for instance \cite{Kuru2013,hennig2018} and references therein), while their  mathematical theory has been developed in \cite{BC2016,BC2017} and \cite{GHP17}. This theory is based on the concept of \emph{admissible meshes}, hierarchical multilevel meshes with a suitable grading.

The efficient implementation of adaptive methods requires to apply suitable preconditioners for the solution of the linear system arising from the discretization. The development of preconditioners in IGA has drawn the attention of several researchers, and in particular multilevel methods for tensor-product B-splines have been first analyzed in \cite{Gahalaut_MG,Gahalaut_AMG,BHKS13}, while robustness in terms of the degree was then explored in \cite{Hofreither2016,Hofreither_SINUM2017,Manni_MG}, always limited to the non-adaptive case. Due to the multilevel structure of (T)HB-splines, multilevel preconditioners seem the most natural choice, and indeed they have been used for the first time in \cite{Hofreither2016b}, and very recently also in \cite{dePrenter2019}. An additive multilevel preconditioner was also analyzed in the case of T-splines in \cite{CV19}. More recently, an overlapping Schwarz preconditioner for adaptive IGA-boundary elements was introduced in \cite{FUHRER2019}, although we remark that in this case local refinement is achieved by univariate B-splines.

In this paper we analyze multilevel preconditioners for (T)HB-splines, by focusing on additive multilevel preconditioners (also known as BPX) \cite{BPX1}, although most of the theoretical results can be applied in the analysis of multiplicative multilevel methods. Analogously to the analysis of BPX preconditioners in the finite element context \cite{JXu_SIAM_Review,CNX}, our analysis requires the proof of two properties that respectively bound the minimum and the maximum eigenvalue: the stability of the decomposition, and the so-called strengthened Cauchy-Schwarz inequality. We show that, under admissibility of the mesh, it is possible to define suitable decompositions such that both properties hold, and the condition number is bounded independently of the number of levels.

An important issue is the choice of suitable subspaces to decompose the discrete space into levels. In the adaptive finite element setting, the subspaces of multilevel preconditioning must be defined with certain locality, as it was first analyzed in \cite{WuChen06,XCH10}, see also \cite{CNX,XCN}. A similar decomposition was used for T-splines in \cite{CV19}. Due to the high continuity of splines, we can generalize the local construction in \cite{WuChen06} in two different ways: one based on the support of the functions, and the other one in the result of truncation. For both choices, we prove that the condition number is bounded independently of the number of levels. This local construction was not respected in \cite{Hofreither2016b} and \cite{dePrenter2019} for (T)HB-splines, and therefore our decomposition provides smaller subspaces on each level, reducing the computational cost. Even more important, the upper bound of the maximum eigenvalue of the preconditioned linear system requires certain locality, as we will see both in the proofs and in the numerical results. 


Finally, it is worth to remark the important role of the smoother for multilevel preconditioners with splines. In fact, standard smoothers such as (one iteration of) Jacobi or symmetric Gauss-Seidel are not stable with respect to the degree. Multilevel methods for B-splines with stable smoothers based on the mass matrix have been proposed in \cite{Hofreither2016,Hofreither_SINUM2017}, however their efficient implementation is based on the tensor-product structure of B-splines, and their extension to the adaptive setting is not straightforward. Very recently, Schwarz smoothers were introduced in \cite{RRG19,dePrenter2019}. In this paper we limit ourselves to the study of the stability under $h$-refinement, leaving aside the important issue of finding a stable smoother.

The outline of the paper is as follows. In Section~\ref{sec:splines} we recall the definition of (T)HB-splines and the notion of admissible hierarchical meshes, along with some important theoretical results regarding quasi-interpolation in hierarchical spline spaces. In Section~\ref{sec:preliminaries} we present the model problem, and some preliminaries about the BPX preconditioner and the main results to prove that the condition number is bounded. Section~\ref{sec:decompositions} is the core of the paper: we start presenting the different choices of the subspaces that we will analyze, and then we prove the theoretical results that show that, under admissibility of the mesh, our local decompositions satisfy the requirements to provide bounded condition numbers for the preconditioned system. We finish presenting several numerical tests in Section~\ref{sec:numerical}, that confirm our theoretical results.


In the rest of the paper, we will adopt the following compact notation. Given
two real numbers $a,b$ we write $a \lec b$,
when $a\leq C b$ for a generic constant $C$ independent of the mesh size and the number of levels, but that may depend on the degree and the admissibility class, to be defined below. We write $a\simeq b$ when $a \lec b$ and $b \lec a$.

%% file: splines_3_shorter.tex
In this section we recall all the main definitions, notations, and properties related to hierarchical spline spaces.

\subsection{Tensor-product B-splines}

Multivariate B-splines can be constructed by means of tensor products.
For each direction $k=1,\ldots,d$,
assume that $n_k \in \mathbb{N}$, the degree $p_k \in \mathbb{N}$, and the $p_k$-open and ordered knot vector
$\Xi_k=\{\xi_{k,0},\ldots,\xi_{k,n_k+p_k}\}$ are given, where by $p_k$-open we mean that the first and last knots are repeated $p_k+1$ times. In the following we will assume that $\xi_{k,0} = 0$ and $\xi_{k,n_k+p_k} = 1$.
We set the polynomial degree vector
${\mathbf p}:=(p_1,\ldots,p_d)$ and ${\bf \Xi}:=\{\Xi_1, \ldots, \Xi_d\}$. We introduce a set of multi-indices ${\bf I}:=\{{\bf i}=(i_1,\ldots,i_d):0\le i_k\le n_k-1\}$ and for each multi-index ${\bf i}=(i_1,\ldots,i_d)$, we define the local knot vector
\[
{\bf \Xi}_{\bf i, p} := \{ \Xi_{i_1,p_1}, \ldots, \Xi_{i_d,p_d} \},
\]
with the local knot vector in each direction given by $\Xi_{i_k,p_k}:=\{\xi_{i_k},\ldots,\xi_{i_k+p_k+1}\}$. Let ${B}[\Xi_{i_k,p_k}]$ for $i_k=0, \ldots, n_k-1$ be the univariate B-splines of degree $p_k$ in the $k$th direction defined with the Cox-De Boor formula \cite{DeBoor}. We denote by
\begin{equation*}
{\mathcal B}:=\left\{ {B}_{\bf i,p}({\bzeta}) = {B}[\Xi_{i_1,p_1}](\zeta_1) \cdot \ldots \cdot {B}[\Xi_{i_d,p_d}](\zeta_d) , \quad \mbox{for\ all\ } {\bf i} \in {\bf I}\right\}.
\end{equation*}
the set of multivariate B-splines obtained with the tensor product approach.
The spline space in the parametric domain ${\Omega}=[0,1]^d$ is then
\[
S_{\bf p}({\bf \Xi}):=\mbox{span}\{{B}_{\bf i,p}({\bzeta}),\; {\bf i} \in {\bf I}\}.
\]

We also introduce the set of non-repeated interface knots, or breakpoints, $\{\xi_{k, i_0},\ldots,\xi_{k,i_{M_k}}\}$, for each $k = 1,\ldots,d$, which determine the intervals $I_{k,j_k}=(\xi_{k,i_{j_k}}, \xi_{k,i_{j_k+1}})$, for $0 \le j_k \le M_{k}-1$. These intervals lead to the rectangular grid $G$ in the unit domain
\begin{equation*}
G:=\{Q_{\bf j}=I_{1,j_1} \times \ldots \times I_{d,j_d}, \ \mbox{for\ } 0 \le j_k \le M_k-1, \, k = 1, \ldots, d \}.
\end{equation*}
For a generic element $Q_{\bf j}$, we also define its support extension as the union of the supports of functions that do not vanish on $Q_{\bf j}$, namely
\begin{equation} \label{eq:supp-ext-nvariate}
\widetilde{Q}_{\bf j} := \bigcup \{\supp(B) : B \in \mathcal{B} \wedge Q \subset \supp(B) \}.
\end{equation}
Note that the support extension is the Cartesian product of the analogous definition for univariate B-splines, and it contains $2p_k+1$ knot spans in each direction, see \cite[Section~2.2]{IGA-acta}.

The following assumption of local quasi-uniformity guarantees that the size of an element is comparable to the size of its support extension.
\begin{assumption}\label{assumpt_quasiuniform2}
For each $k=1,\ldots,d$, the partition given by the breakpoints $\{\xi_{k,i_0},\xi_{k,i_1},\ldots,\xi_{k,i_{M_k}}\}$ is locally
quasi-uniform, that is, there exists a constant $\theta \ge 1$ such that the
mesh sizes $h_{j_k} = \xi_{i_{j_k+1}}-\xi_{i_{j_k}}$ satisfy the relation $\theta^{-1} \le h_{j_k} /h_{j_k+1} \le \theta$, for $ j_k = 0, \ldots , M_k-2$.
\end{assumption}

Finally, we introduce the quasi-interpolant
\begin{equation}\label{multi_quasiint}
{\bf \Pi}_{\bf p,\Xi}: L^2(\Omega) \rightarrow S_{\bf p}({\boldsymbol \Xi}), \qquad
{\bf \Pi}_{\bf p,\Xi}(f) := \sum_{{\bf i}\in{\bf I}} \lambda_{\bf i,p}(f){B}_{\bf i, p},
\end{equation}
where the dual functionals $\lambda_{\bf i,p} : L^2(\Omega) \rightarrow \mathbb{R}$ are defined by tensor-product of a dual basis of univariate B-splines, and in fact they form a dual basis, see \cite[Theorem~4.41 and Theorem~12.6]{Schumi}. As a consequence, ${\bf \Pi}_{\bf p,\Xi}$ is a projector. Moreover, both the dual functionals and the quasi-interpolant are stable with respect to the $L^2$-norm and, under Assumption~\ref{assumpt_quasiuniform2}, also with respect to the $H^1$-norm, see \cite[Section~2.2.2]{IGA-acta} for more details.


\subsection{Hierarchical B-splines} \label{sec:HB-splines}
Let us consider a sequence $S_{\bf p}({\bf \Xi}^0)\subset S_{\bf p}({\bf \Xi}^1) \subset \cdots \subset S_{\bf p}({\bf \Xi}^L)$ of $L+1$ spaces of tensor-product splines of degree $\bp=(p_1,\ldots,p_d)$ defined on the closed domain $[0,1]^d$, and an associated sequence of closed domains $\Omega^0 \supseteq \Omega^1 \supseteq \cdots \supseteq \Omega^{L+1}$, with $\Omega^0 = [0,1]^d$ and $\Omega^{L+1}=\emptyset$. For $\ell=0,1,\ldots, L$, we denote by ${\mathcal B}^\ell$  and by ${G}^\ell$ the tensor product B-spline basis and the tensor-product mesh corresponding to $S_{\bf p}({\bf \Xi}^\ell)$, respectively. Each ${G}^\ell$ is obtained by uniform dyadic refinement of ${G}^{\ell-1}$, $\ell=1,\ldots,L$, and therefore we can associate to each level a {\it mesh size} $h_\ell \simeq 2^{-\ell}$. Let ${\mathcal Q}$ be the {\it hierarchical mesh} defined by
\[
{\mathcal Q} : = \{Q\in {\mathcal G}^\ell,\ 0\le \ell \le L\} \quad \mbox{with} \quad {\mathcal G}^\ell := \{Q\in G^\ell:~Q\subset \Omega^\ell \wedge Q\nsubseteq \Omega^{\ell+1}\},
\]
and we assume that the subdomain $\Omega^{\ell+1}$ is built as the union of cells of the previous level, namely $\Omega^{\ell+1} = \bigcup_{Q \subset {\cal R}} \overline{Q}$, for some ${\cal R} \subseteq G^\ell$.
\begin{definition}\label{levelQ}
For each element $Q\in {\cal Q} \cap {\cal G}^k$ we define its level as $\ell(Q):=k$.
\end{definition}

\begin{definition} The hierarchical B-spline (HB-spline) basis ${\mathcal H}_{\bf p}({\mathcal Q})$ with respect to the mesh ${\mathcal Q}$ is defined as
\[
{\mathcal H}_{\bf p}({\mathcal Q}):=\bigcup_{\ell=0}^L  A_{\bf p}^{\ell}({\mathcal Q})\,, \qquad
A_{\bf p}^{\ell}({\mathcal Q}) \,:=\, \{{B}\in {\mathcal B}^\ell:\,
\supp (B)  \subseteq \Omega^\ell\,\,\wedge \,\,
\supp (B) \not\subseteq \Omega^{\ell+1} \},
\]
and
$
S_{\bf p}({\mathcal Q}):={\rm span}\ {\mathcal H}_{\bf p}({\mathcal Q})
$
is the  hierarchical spline space.
\end{definition}

Note that the mesh ${\mathcal Q}$ and the basis ${\cal H}_{\bf p}({\mathcal Q})$ can be constructed through an iterative procedure. Let us first introduce, for $0 \le \ell \le \ell' \le L$ the sets
\begin{align*}
& {\cal Q}^{\ell,\ell'} := \{ Q \in G^\ell : Q \subseteq \Omega^{\ell'} \}, \qquad {\cal B}^{\ell,\ell'} := \{ B \in {\cal B}^\ell : \supp B \subseteq \Omega^{\ell'} \},
\end{align*}
made of elements (respectively functions) of level $\ell$ contained (with support contained) in $\Omega^{\ell'}$. Then, the hierarchical mesh ${\mathcal Q} = \mathcal{Q}^L$ and the basis ${\cal H}_{\bf p}({\mathcal Q}) = {\cal H}_{\bf p}({\mathcal Q}^L)$ can be iteratively constructed as follows:
\begin{enumerate}
\item ${\mathcal Q}^0 := G^0$ and
${\cal H}_{\bf p}({\mathcal Q}^0) := {\cal B}^0$;
\item for $\ell=0,\ldots,L-1$
\begin{equation}\label{eq:intmeshesspaces}
{\cal Q}^{\ell+1} := ({\cal Q}^\ell \setminus {\cal Q}^{\ell,\ell+1}) \cup {\cal Q}^{\ell+1,\ell+1}, \qquad {\cal H}_{\bf p}({\mathcal Q}^{\ell+1}) := ({\cal H}_{\bf p}({\cal Q}^{\ell}) \setminus {\cal B}^{\ell,\ell+1}) \cup {\cal B}^{\ell+1,\ell+1},
\end{equation}
\end{enumerate}
where at each step we remove from the basis the B-splines in ${\cal B}^{\ell,\ell+1}$ (coarse functions whose support is completely contained in $\Omega^{\ell+1}$), and add the B-splines in ${\cal B}^{\ell+1,\ell+1}$
(fine functions whose support is completely contained in $\Omega^{\ell+1}$).
For the intermediate spaces we will also use the notation $S_{\bf p}({\mathcal Q^\ell}):={\rm span}\ {\mathcal H}_{\bf p}({\mathcal Q^\ell})$, for $\ell=0,\ldots,L$.

\smallskip
We introduce the truncation operator ${\trunc}^{\ell+1} : S_{\bf p}({\bf \Xi}^\ell) \rightarrow S_{\bf p}({\bf \Xi}^{\ell+1})$ as follows.
For any $s \in S_{\bf p}({\bf \Xi}^\ell)$ with representation in the B--spline basis of $S_{\bf p}({\bf \Xi}^{\ell+1})$ given by
\begin{equation*}
s  = \sum_{B \in {\cal B}^{\ell+1}} \sigma_B^{\ell+1} {B},
\end{equation*}
the truncation of $s$ with respect to level $\ell+1$ is defined as
\begin{equation*}
{\trunc}^{\ell+1} (s)   \,:=\,
\displaystyle{ \sum_{B \in {\cal B}^{\ell+1} \,:\,
\supp {B}\, \not\subseteq\, \Omega^{\ell+1}  } } \sigma_B^{\ell+1} {B} \,,
\quad \ell=0,\ldots,L.
\end{equation*}
The (cumulative) truncation operator ${\Trunc}^{\ell+1} : S_{\bf p}({\bf \Xi}^\ell) \rightarrow S_{\bf p}({\mathcal Q}) \subseteq S_{\bf p}({\bf \Xi}^L)$ with respect to all finer levels in the hierarchy is then defined as
\begin{equation*}
{\Trunc}^{\ell+1}(s) \,:=\,   {\trunc}^{L}( {\trunc}^{L-1}(\cdots( {\trunc}^{\ell+1}(s)) \cdots))\,,
\quad \quad \ell=0,\ldots,L-1,
\end{equation*}
and for convenience we also define ${\Trunc}^{L+1} (s) := s$,  for $s\in S_{\bf p}({\bf \Xi}^L)$. The truncated hierarchical B-spline basis of $S_{\bf p}({\mathcal Q})$ is obtained by applying the cumulative truncation operator to the elements of ${\mathcal H}_{\bf p}({\mathcal Q})$.
\begin{definition}\label{dfn:thb}
The truncated hierarchical B-spline (THB-spline) basis ${\cal T}_{\bf p} ({\mathcal Q})$ of degree ${\bf p}$  with respect to the mesh ${\mathcal Q}$ is defined as
\begin{equation*}  \label{truncbasis2}
{\cal T}_{\bf p} ({\mathcal Q}) \,:=\,
\bigcup_{\ell=0}^L A_{\bf p}^{\ell,T}({\mathcal Q})\,, \qquad A_{\bf p}^{\ell,T}({\mathcal Q}):=\{{\Trunc}^{\ell+1} (B)\,:\, B \in A_{\bf p}^{\ell}({\mathcal Q})\}.
\end{equation*}
\end{definition}
For each THB-spline basis function $T \in A_{\bf p}^{\ell,T}({\mathcal Q})$, we define its mother function as the corresponding function without truncation, namely
\[
{\rm mot}(T) := B \iff T = {\Trunc}^{\ell+1}(B).
\]
Analogously to ${\cal H}_{\bf p}({\mathcal Q})$, the truncated basis ${\cal T}_{\bf p}({\mathcal Q}) = {\cal T}_{\bf p}({\mathcal Q}^L)$ can be also constructed iteratively:
\begin{enumerate}
\item ${\cal T}_{\bf p}({\mathcal Q}^0) := {\cal B}^0$;
\item for $\ell=0,\ldots,L-1$
\[
{\cal T}_{\bf p}({\mathcal Q}^{\ell+1}) := \{{\trunc}^{\ell+1}(T):\,T \in {\cal T}_{\bf p}({\cal Q}^\ell) \setminus {\cal B}^{\ell,\ell+1}\} \cup {\cal B}^{\ell+1,\ell+1}.
\]
\end{enumerate}



THB-splines form a \emph{partition of unity} and satisfy the {\sl preservation of coefficients} property. They are also a strongly stable basis for $S_{\bf p}({\cal Q})$ with respect to the supremum norm unlike the classical hierarchical basis, which is only weakly stable, see \cite{GJS14} for details.

\subsubsection{Admissible meshes} In order to be able to construct suitable decompositions, we will need to consider particular classes of hierarchical meshes \cite{bracco2018b,BC2016}. We briefly recall some of the main definitions and properties related to them.

\begin{definition}\label{dfn:hse}
The multilevel support extension $S({Q},k) $ of an element ${Q}\in{G}^\ell$ with respect to level $k$, with $0\le k\le \ell$,
is defined as
\[
S(Q,k) := \tilde{Q'} , \text{ with } Q' \in G^k, \, Q \subseteq Q',
\]
where $Q'$ is an ancestor of $Q$ of level $k$, and $\tilde{Q'}$ is the support extension defined in \eqref{eq:supp-ext-nvariate}.
\end{definition}

\begin{definition}\label{extsupp}
For any $T={\Trunc}^{\ell+1} (B)\in {\cal T}_{\bf p}({\cal Q})$, $B \in A_{\bf p}^{\ell}({\mathcal Q})$, its extended support is defined as
\begin{equation*}
\esupp(T): =\supp\left({\trunc}^{\ell+1}(B)\right).
\end{equation*}
\end{definition}
Note that it obviously holds that $\supp(T) \subseteq \esupp(T)$.

\begin{definition}\label{admesh}
A hierarchical mesh ${\cal Q}$ is ${\cal H}$-admissible (respectively, ${\cal T}$-admissible) of class $m$, $2\le m<L+1$, if the HB-splines (respectively, THB-splines) taking non-zero values on any cell $Q\in{\cal Q}$ belong to at most $m$ successive levels.
\end{definition}

In order to define a further type of admissibility, we need to introduce, for $\ell = 0, \ldots, L$, the auxiliary subdomains
\begin{equation*}
\begin{array}{l}
{{\omega}}^{\ell}_{\cal H} :=\bigcup\left\{ \overline{{Q}} \,:\, {Q} \in {G}^{\ell} \,\wedge\, S({Q},{\ell}-1)\subseteq {\Omega}^{\ell} \right\}, \\
{{\omega}}^{\ell}_{\cal T} :=\bigcup\left\{ \overline{{Q}} \,:\, {Q} \in {G}^{\ell} \,\wedge\, S({Q},{\ell})\subseteq {\Omega}^{\ell} \right\},
\end{array}
\end{equation*}
where clearly $\omega^\ell_{\cal H} \subseteq \omega^\ell_{\cal T}$.

\begin{definition}\label{admeshstr}
A hierarchical mesh ${\cal Q}$ is strictly ${\cal H}$-admissible (respectively, strictly ${\cal T}$-admissible) of class $m$ if
\begin{equation*}
{\Omega}^\ell\subseteq {\omega}^{\ell-m+1}_{\cal H}, \qquad (\text{resp. } {\Omega}^\ell\subseteq {\omega}^{\ell-m+1}_{\cal T}),
\end{equation*}
for $\ell=m,m+1,\ldots,L$.
\end{definition}
Note that the subdomains $\omega^\ell_{\cal T}$ are analogous to the ones defined in \cite[Section~4]{Kraft}.
They represent the biggest subset of $\Omega^\ell$ such that ${\cal H}_{\bf p}({\cal Q}^\ell)$ spans the restriction of the B-spline space $S_{\bf p}({\bf \Xi}^\ell)$ to $\omega^\ell_{\cal T}$.

The following result is proved in \cite[Proposition~1]{bracco2018b}.
\begin{proposition}\label{admHT}
For a hierarchical mesh ${\cal Q}$ the following properties hold:
\begin{itemize}
\item[(a)] if ${\cal Q}$ is strictly ${\cal H}$-admissible (resp. strictly ${\cal T}$-admissible) of class $m$, then it is ${\cal H}$-admissible (resp. ${\cal T}$-admissible) of class $m$;
\item[(b)] if ${\cal Q}$ is (strictly) ${\cal H}$-admissible of class $m$, then it is (strictly) ${\cal T}$-admissible of class $m$.
\end{itemize}
\end{proposition}

\begin{proposition}\label{admHT2}
If ${\cal Q}$ is strictly ${\cal H}$-admissible (resp. ${\cal T}$-admissible) of class $m$, then all the intermediate meshes ${\cal Q}^k$, $k=0,\ldots,L-1$ are strictly ${\cal H}$-admissible (resp. ${\cal T}$-admissible) of class $m$.
\end{proposition}

\begin{proof}
The proof is an obvious consequence of the definition, after noting that, for the intermediate meshes ${\cal Q}^k$, the subdomains $\Omega^\ell$ up to level $k$, and the auxiliary domains $\omega_{\cal H}^\ell, \omega_{\cal T}^\ell$ up to level $k-m+1$, coincide with the ones of ${\cal Q}$.
\end{proof}


\begin{definition}\label{extsuppext}
For any element $Q \in {\cal Q}$, we define
\begin{equation*}
{\bar S^*}(Q) := \bigcup\{\esupp(T): T \in {\cal T}_{\bf p}({\cal Q}) \wedge \esupp(T) \cap Q \not = \emptyset \}
\quad \text{and} \quad
S^*(Q):=\mathrm{int} ({\bar S}^*(Q)).
\end{equation*}
\end{definition}
We note that if ${\cal Q}$ is a strictly ${\cal T}$-admissible mesh of class $m$, and $Q$ is an element of level $\ell(Q)$, the functions appearing in the previous definition must belong to levels $\ell(Q) -m +1$ to $\ell(Q)$. Indeed, in \cite{BC2019} the set ${\bar S}^*(Q)$ is denoted by ${\bar S}^*(Q,\ell(Q)-m+1)$. We also generalize the previous definition as follows.

\begin{definition}\label{extsuppext2}
Given a subset $\sigma\subseteq \Omega$ and a strictly ${\cal T}$-admissible hierarchical mesh of class $m$, we define
\begin{equation*}
S^*(\sigma):=\mathrm{int} \bigcup_{Q\in {\cal Q}:\, Q\subseteq \sigma} \bar{S}^*(Q).
\end{equation*}
\end{definition}

The following results are from Theorem~4 and Corollary~5 in \cite{BC2019}.

\begin{proposition}\label{prop:4incorr}
Let ${\cal Q}$ be a strictly ${\cal T}$-admissible mesh of class $m$, and $Q \in {\cal Q}$. The set $S^*(Q)$ is connected, it contains a bounded number of elements in ${\cal Q}$, which is independent of $Q$, and it holds that $h_{S^*(Q)} \simeq h_Q$. The constants depend on $m$ and the degree $p$, but are independent of $Q$ and the number of levels.
\end{proposition}
\begin{corollary}\label{corol:5incorr}
Let ${\cal Q}$ be a strictly ${\cal T}$-admissible mesh of class $m$. There exists a constant $C_{\cal R}$ such that, for all $Q \in {\cal Q}$, the number of elements $Q' \in {\cal Q}$ such that $Q \subset S^*(Q')$ is bounded by $C_{\cal R}$. The constant $C_{\cal R}$ depends on $m$ and $p$, but not on the number of levels of ${\cal Q}$.
\end{corollary}

\subsubsection{Hierarchical quasi-interpolant} \label{sec:quasi-interpolant}
With a slight abuse of notation, for each level $\ell=0,\ldots,L$ let ${\mathcal I}^\ell$ be a quasi-interpolant in $S_{\bf p}({\bf \Xi}^\ell)$ of the form
\begin{equation}\label{tpqi1}
{\mathcal I}^\ell(v):=\sum_{B\in {\cal B}^\ell} \lambda_{B}(v) B, \qquad v\in L^2(\Omega),
\end{equation}
where each dual functional $\lambda_B$ is defined through the local projection onto an element $Q_B$ in the support of $B$, see \cite{buffa2016b} for details.
Given the hierarchical space $S_{\bf p}({\mathcal Q})$ and its truncated basis ${\cal T}_{\bf p} ({\mathcal Q})$, we can define the hierarchical quasi-interpolant as in \cite{MS15}
\begin{equation}\label{hqi1}
{\mathcal I}_{{\mathcal Q}}(v):=\sum_{\ell=0}^{L}\sum_{T\in A_{\bf p}^{\ell,T}({\mathcal Q})} \lambda_{T}(v) T, \qquad v\in L^2(\Omega),
\end{equation}
where each $\lambda_{T}(v):=\lambda_{B}(v)$, with $B = {\rm mot}(T)$, is the functional of the quasi-interpolation scheme of level $\ell$ in \eqref{tpqi1}, corresponding to the mother B-spline of $T$. Note that with our choice of the dual functionals, the quasi-interpolant is the same as in \cite{BC2017}.


For $0\le k\le L$, we denote by ${\mathcal I}_{{\mathcal Q}^k}$ the hierarchical quasi-interpolant of type \eqref{hqi1} in the intermediate hierarchical space $S_{\bf p}({\mathcal Q}^k)$.
In the construction of the quasi-interpolant ${\mathcal I}_{{\mathcal Q}^k}$, for each basis function $T\in A_{\bf p}^{\ell,T}({\mathcal Q}^k)$, $\ell = 0, \ldots, k$, we choose the element $Q_T$ of the local projection such that $Q_T \subset \Omega^\ell \setminus \Omega^{\ell+1}$ whenever possible. Note that such an element always exists for $\ell =0, \ldots, k-1$; for $\ell=k$ it exists for functions with support not contained in $\Omega^{k+1}$, and that will not be removed from the basis at the next step, while it does not exist for functions with support contained in $\Omega^{k+1}$ and that will be removed from the basis, but in this case we can choose any element in the support. With this choice, ${\mathcal I}_{{\mathcal Q}^k}$ is a projector on $S_{\bf p}({\mathcal Q}^k)$, see \cite{MS15}. Moreover, we also have that the dual functional associated to $T$ is not modified after truncation. More precisely, let $0\le \ell \le k_1,k_2\le L$, then we get
\begin{equation}\label{hqi1b}
T_1\in A_{\bf p}^{\ell,T}({\mathcal Q}^{k_1}),\, T_2\in A_{\bf p}^{\ell,T}({\mathcal Q}^{k_2}), \text{ and } {\rm mot}(T_1) = {\rm mot}(T_2) \implies \lambda_{T_1}=\lambda_{T_2}.
\end{equation}
Let us denote by $\|\cdot \|_0$ the norm of $L^2(\Omega)$, and by $\|\cdot\|_{0,\omega}$ the norm in $L^2(\omega)$, with $\omega \subseteq \Omega$.
Since $\lambda_{T}(v) =\lambda_{B}(v)$, with $B = {\rm mot}(T)$, the stability of the dual functionals
\begin{equation} \label{eq:stability_dual}
|\lambda_T(v)| \lesssim |Q_T|^{-1/2} \| v \|_{0,Q_T} \simeq h_{Q_T}^{-d/2} \| v \|_{0,Q_T},
\end{equation}
where $h_{Q_T}$ is the size of $Q_T$, follows from Theorem~1 in \cite{buffa2016b}. In addition, the quasi-interpolants ${\mathcal I}_{{\mathcal Q}^k}$ satisfy the following stability property.
\begin{lemma}\label{stabaux1}
Let ${\cal Q}$ be a strictly ${\cal T}$-admissible hierarchical mesh of class $m$. There exists a constant $C$ depending on ${\bf p}$ and $m$ such that for any element $Q$ of ${\cal Q}$
\begin{equation*}
\Vert {\mathcal I}_{{\mathcal Q}}(v)\Vert_{0,Q} \le C\Vert v\Vert_{0,S^*(Q)}.
\end{equation*}
\end{lemma}

\begin{proof}
The result is proved in Proposition 10 in \cite{BC2019}, following Proposition~5 in \cite{BC2017}.
\end{proof}

\begin{remark}
For simplicity, we have restricted ourselves to the parametric domain $\Omega = [0,1]^d$. All the results of this section can be simply extended to the isogeometric setting, in which the domain is defined as $\Omega = {\bf F}([0,1]^d)$, under the standard assumptions that the parametrization ${\bf F}$ is defined from the coarsest space of the hierarchy, $S_{\bf p}(\boldsymbol{\Xi}^0)$, and it does not contain singularities. For instance, the hierarchical mesh in the physical domain is defined as the set of elements $\{ {\bf F}(Q) : Q \in {\cal Q} \}$, while the hierarchical basis in the physical domain is given by $\{B \circ {\bf F}^{-1} : B \in {\cal H}_{\bf p}({\cal Q}) \}$. Other definitions, such as the subdomains $\Omega^\ell$, the support extension, or the truncated basis, are extended to the physical domain in a completely analogous way. As a consequence, the theoretical results in Propositions~\ref{admHT}, \ref{admHT2} and \ref{prop:4incorr}, and in Lemma~\ref{stabaux1} are also valid in the physical domain, with constants that also depend on the parametrization ${\bf F}$.
\end{remark}

%% file: preliminaries_3.tex
In this section we present the construction of the additive multilevel preconditioner, and the theoretical results needed to prove that the condition number is bounded. The preconditioner, and also the theoretical results, rely on the choice of suitable subspaces ${\cal V}_i$, that will be introduced in detail in Section~\ref{sec:decompositions}.

\subsection{Problem setting}
We consider as the model problem the Poisson equation with homogeneous Dirichlet boundary conditions defined in $\Omega \subset \mathbb{R}^d$. For $f \in L^2(\Omega)$ the solution is a discrete function $u \in {\cal V}$ such that
\begin{equation}\label{weak_form}
a(u,v) := \int_\Omega \nabla u \cdot \nabla v \, {\rm d}x = \int_\Omega f v \, {\rm d}x \quad \forall v\in \mathcal{V},
\end{equation}
where ${\cal V}$ is a suitable discrete space, that in our case will be the space of hierarchical B-splines that vanish on the boundary.
Defining a linear operator $A:\mathcal{V}\rightarrow \mathcal{V}$ by
\[
(Au,v)=a(u,v), \ \forall u, v \in \mathcal{V}
\]
and also $b:=Pf \in\mathcal{V}$, the $L^2$-projection of $f$ into ${\cal V}$, the discrete problem is equivalent to find $u\in \mathcal{V}$ that solves the linear operator equation
\begin{equation}\label{linear_op}
Au=b.
\end{equation}
We denote with $\Vert \cdot \Vert_A$ the energy norm defined by $\Vert u\Vert_A^2=a(u,u)$, $\forall u\in {\cal V}$, and analogously to the $L^2$ norm, its restriction to a subdomain $\omega \subseteq \Omega$ will be denoted by $\| \cdot \|_{A,\omega}$. Similarly, we will denote by $a(u,v)_\omega := \int_\omega \nabla u \cdot \nabla v \, {\rm d}x$.


\subsection{The method of parallel subspace corrections}
Let $B$ be a preconditioner for the linear operator equation above, and let $u^k, k = 0, 1, 2, \ldots$ the solution sequence of the preconditioned conjugate gradient (PCG) algorithm. Then the following error estimate is well-known:
\[
\|u-u^k\|_A\le 2\left(\frac{\sqrt{\kappa(BA)}-1}{\sqrt{\kappa(BA)}+1}\right)^k
\|u-u^0\|_A,
\]
which implies that the PCG method converges faster with a smaller condition number $\kappa(BA)$.
The method of parallel subspace corrections (PSC) provides a particular construction of the operator $B$. The starting point is a suitable decomposition of $\mathcal{V}$
\begin{equation} \label{eq:decomposition-general}
\mathcal{V} = \sum_{i=0}^J \mathcal{V}_i,
\end{equation}
where $\mathcal{V}_i$ are subspaces of $\mathcal{V}$, and we assume that each subspace ${\cal V}_i$ is associated to a certain mesh size $h_i$.
The discrete problem \eqref{weak_form} can be split into sub-problems in
each $\mathcal{V}_i$ with smaller size. Throughout this paper, we use the following operators, for $i=0,1,\ldots,J$:
\begin{itemize}
\item $I_i : \mathcal{V}_i \rightarrow \mathcal{V}$ the natural inclusion, also called the prolongation operator;
\item $A_i : \mathcal{V}_i \rightarrow \mathcal{V}_i$ the restriction of $A$ to the subspace $\mathcal{V}_i$;
\item $R_i : \mathcal{V}_i \rightarrow \mathcal{V}_i$ an approximation of $A_i^{-1}$, usually called the \emph{smoother}.
\end{itemize}
With this notation the operator $B$ for the method PSC is given by
\begin{equation}\label{def_B}
B:=\sum_{i=0}^{J} I_i R_i I^t_i, 
\end{equation}
where $I_i^t$ denotes the adjoint operator of $I_i$. In the multilevel space decomposition setting, the operator $B$ in \eqref{def_B} is the well-known additive multilevel preconditioner, also known as BPX preconditioner \cite{BPX1}.

The convergence analysis of PSC and the analysis of the condition number of the BPX preconditioned system are based on the following important properties \cite{CNX}, which are closely related to properties (4.2) and (4.3) in \cite{JXu_SIAM_Review}

\begin{enumerate}[label=(A\arabic*)]
\item {\bf Smoothing property.} For $0 \le i \le J$, it holds that \label{assumption_smoother}
\begin{equation*}
(R_i^{-1} u_i, u_i) \simeq h_i^{-2} \| u_i \|_0^2, \quad \forall u_i \in {\cal V}_i
\end{equation*}
\item {\bf Stable Decomposition.} For any $v\in\mathcal{V}$, there exists a decomposition $v=\sum_{i=0}^J
v_i,\ v_i\in\mathcal{V}_i,\ i=0,\ldots,J$ such that \label{assumption_decomposition}
\begin{equation*}
\sum_{i=0}^{J} h_i^{-2} \|v_i\|_{0}^2 \lesssim \|v\|_A^2.
\end{equation*}
\item {\bf Strengthened Cauchy-Schwarz (SCS) inequality.} For any $u_i,v_i\in\mathcal{V}_i,
i=0,\ldots,J$ \label{assumption_scs}
\begin{equation*}
\left| \sum_{i=0}^{J} \sum_{j=i+1}^{J} a(u_i,v_j) \right| \lesssim
\left(\sum_{i=0}^{J} \|u_i\|_A^2 \right)^{1/2}
\left(\sum_{j=0}^{J} h_j^{-2} \|v_j\|_0^2 \right)^{1/2}.
\end{equation*}
\item {\bf Inverse inequality.} For any $u_i \in {\cal V}_i$, it holds that \label{assumption_inverse}
\begin{equation*}
\| u_i \|_A^2 \lesssim h_i^{-2} \| u_i \|^2_0.
\end{equation*}
\end{enumerate}

\begin{theorem}\label{thm:cond}
Let ${\cal V} = \sum_{i=0}^J {\cal V}_i$ be a space decomposition satisfying assumptions \ref{assumption_decomposition}--\ref{assumption_inverse}, and the $R_i$ smoothers satisfying \ref{assumption_smoother}, for $i = 0, \ldots, J$. Then, $B$ defined by \eqref{def_B} satisfies
\begin{equation*} 
\kappa(BA) \lec 1.
\end{equation*}
\end{theorem}
\begin{proof}
Let $u = \sum_{i=0}^J u_i$, with $u_i \in {\cal V}_i$. We first notice that, from \ref{assumption_scs} and \ref{assumption_inverse}, and then applying \ref{assumption_smoother}, it holds that
\[
\| u \|_A^2  = \| \sum_{i=0}^J u_i \|_A^2 \lesssim \sum_{i=0}^J h_i^{-2} \|u_i\|_0^2 \lesssim \sum_{i=0}^J (R_i^{-1} u_i, u_i).
\]
Taking the infimum, we obtain
\[
\| u \|_A^2 \lesssim \inf_{\sum_{i=0}^J u_i = u} \sum_{i=0}^J (R_i^{-1} u_i, u_i) = (B^{-1} u, u),
\]
see \cite[Lemma 2.4]{Xu_Zikatanov_2002} for a concise proof of the identity. This implies that the maximum eigenvalue is bounded, $\lambda_{\max}(BA) \lesssim 1$, see \cite[Lemma~2.1]{JXu_SIAM_Review}.

Similarly, to bound the minimum eigenvalue we first apply the same identity as above, then \ref{assumption_smoother} and finally \ref{assumption_decomposition} to get
\[
(B^{-1}u,u) \lesssim \sum_{i=0}^J (R_i^{-1} u_i, u_i) \lesssim \sum_{i=0}^J h_i^{-2} \|u_i\|_0^2 \lesssim \| u \|_A^2,
\]
and as a consequence the minimum eigenvalue satisfies $\lambda_{\min}(BA) \gtrsim 1$.
\end{proof}

%% file: decomposition_3.tex
In this section we first introduce suitable decompositions as in \eqref{eq:decomposition-general} for the construction of the BPX preconditioner for hierarchical B-splines, along with some other possible decompositions that will be used in the numerical examples of Section~\ref{sec:numerical}. We then prove that the chosen decompositions satisfy the smoother property \ref{assumption_smoother}, the stability property \Aone\ and the SCS inequality \Atwo, under the condition that the hierarchical mesh is strictly ${\cal T}$-admissible.

From now on, we will consider discrete spaces of \emph{functions that vanish on the boundary}. For simplicity, we will maintain the same notation for these spaces. We will also assume that the hierarchical B-splines have $C^0$ continuity or higher, in such a way that the discrete spaces are contained in $H^1_0(\Omega)$.

\subsection{Decompositions of hierarchical spaces} 
Let ${\cal Q}$ be a hierarchical mesh, and $\mathcal V:=S_{\bf p}({\mathcal Q})$ the corresponding hierarchical space, constructed as in Section~\ref{sec:HB-splines}.
By starting from the initial mesh ${\cal Q}^0$, we make use of the intermediate hierarchical meshes ${\cal Q}^\ell$ defined in \eqref{eq:intmeshesspaces} for $\ell=1,\ldots,L$, and start introducing the following auxiliary sets of basis functions, using the convention that ${\cal T}_{\bf p}({\mathcal Q}^{-1}) = \emptyset$,
\begin{align*}
\Phi^\ell:= {\cal T}_{\bf p}({\mathcal Q}^\ell)\setminus{\cal T}_{\bf p}({\mathcal Q}^{\ell-1}), \; \text{ for } 0 \le \ell \le L, \\
\Psi^{\ell-1}:= {\cal T}_{\bf p}({\mathcal Q}^{\ell-1})\setminus{\cal T}_{\bf p}({\mathcal Q}^{\ell}), \; \text{ for } 1 \le \ell \le L.
\end{align*}
The set $\Phi^\ell$ is formed by functions that, at step $\ell$ of the algorithm, are either added to the basis or further truncated, while $\Psi^{\ell-1}$ is formed by functions that at the same step are either removed from the basis, or have been truncated further, see Figure~\ref{fig:psi-phi}. Note that, since ${\cal S}_{\bf p}({\cal Q}^{\ell-1}) \subseteq {\cal S}_{\bf p}({\cal Q}^{\ell})$, it trivially holds that
\begin{equation}\label{eq:psiCphi}
{\rm span} \, \Psi^{\ell-1} \subseteq{ \rm span} \, \Phi^\ell \; \text{ for } 1 \le \ell \le L.
\end{equation}
\begin{figure}
\subfigure[Functions in $\Psi^{\ell-1}$, solid (red) lines.]{
\includegraphics[width=0.48\textwidth, trim=2cm 0cm 2cm 1cm, clip]{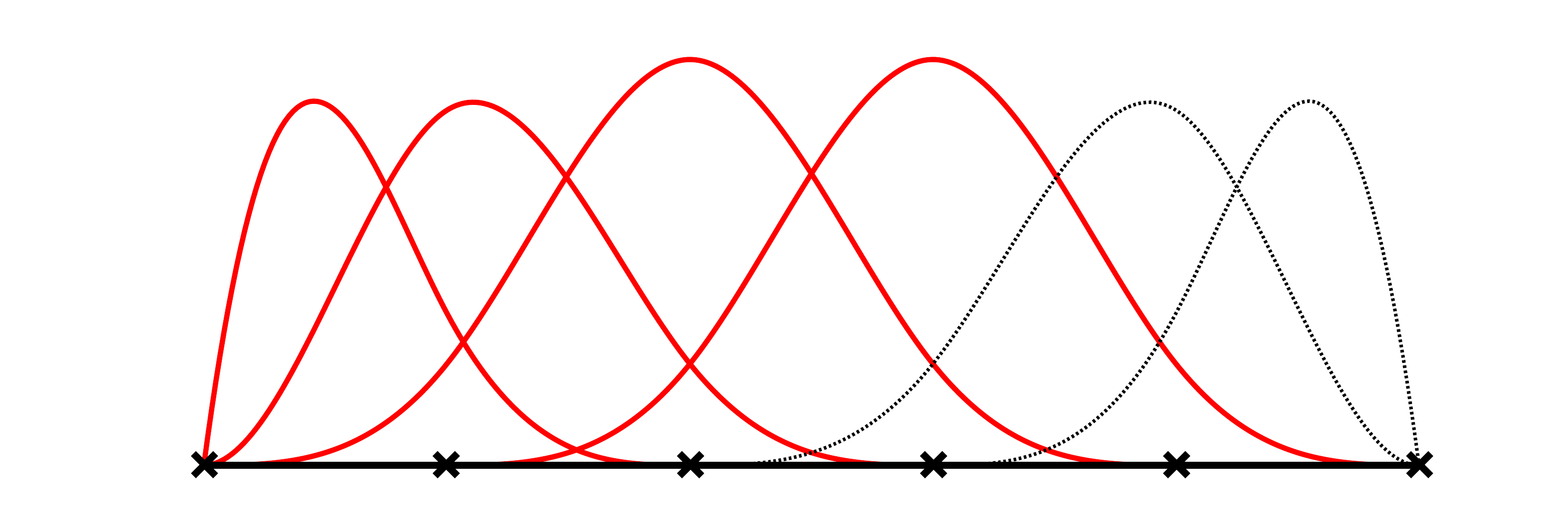}
}
\subfigure[Functions in $\Phi^\ell$, solid (blue) lines.]{
\includegraphics[width=0.48\textwidth, trim=2cm 0cm 2cm 1cm, clip]{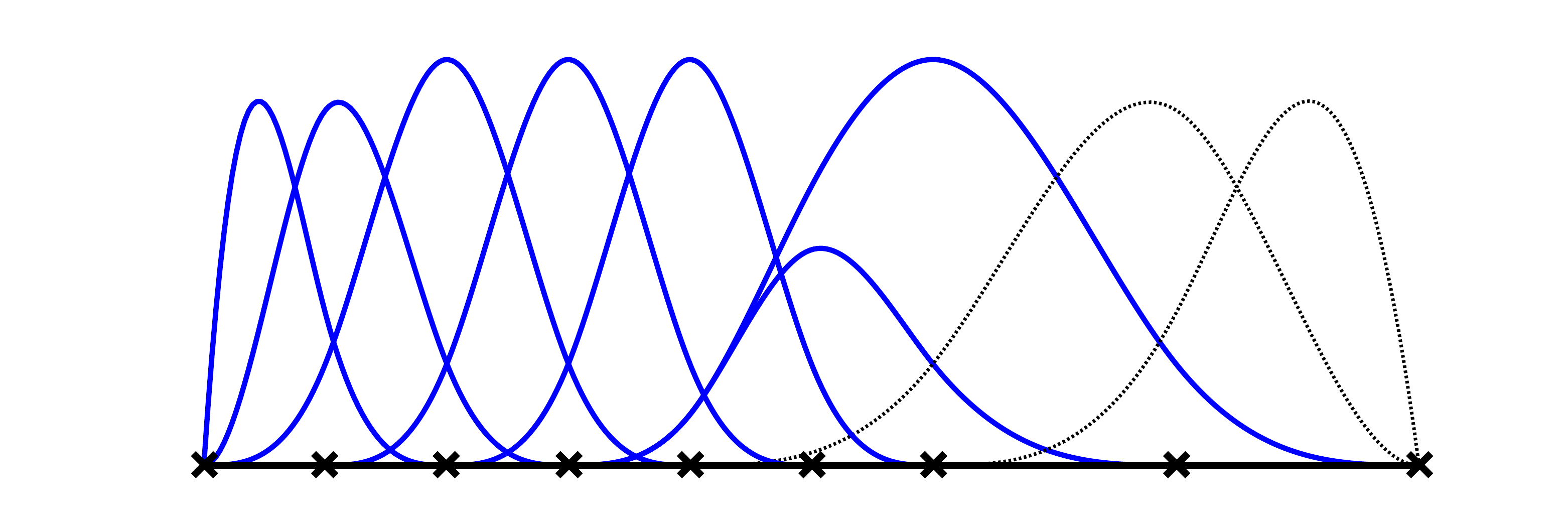}
}
\caption{Visualization of functions in $\Psi^{\ell-1}$ and $\Phi^\ell$ in a simple example. After refining the first three elements, the functions in $\Psi^{\ell-1}$ are either removed from the basis or truncated, while functions in $\Phi^{\ell}$ have been added to the basis or truncated with respect to those in $\Psi^{\ell-1}$.} \label{fig:psi-phi}
\end{figure}
We can now start defining different choices of the decompositions. We first define, for each $0 \le \ell \le L$, the subspaces
\[
{\cal V}^\ell_{\text{mod}} := {\rm span}\, {\Phi}^{\ell} = {\rm span} ({\cal T}_{\bf p}({\mathcal Q}^\ell)\setminus{\cal T}_{\bf p}({\mathcal Q}^{\ell-1})),
\]
that, as we explained above, are the functions newly added to the basis, or further truncated (i.e., ``modified'' functions). The functions are supported in $\Omega^\ell$ plus a certain neighborhood, so they are local as required by multilevel methods with adaptive refinement \cite[Chapter~9]{multigrid_tutorial}, and are somehow analogous to those defined in \cite{WuChen06,XCH10} for finite elements. As we will see, this is the minimal decomposition for which we are able to prove the robustness of the preconditioner with respect to the number of levels. However, the computation of the subspaces, and in particular which functions have been truncated, may not be simple. For this reason we introduce a different decomposition, still based on local subspaces, on which at level $\ell$ we choose THB-splines up to level $\ell$ whose support intersects $\Omega^\ell$, namely
\begin{equation}\label{eq:V-suppT}
{\cal V}^\ell_{{\cal T}\text{-{\rm supp}}} := {\rm span}\, {\cal T}^\ell_{\text{{\rm supp}}}, \quad \text{ with } {\cal T}^\ell_{\text{{\rm supp}}}:= \{{T}\in {\cal T}_{\bf p}({\mathcal Q}^{\ell}):\, {\rm supp}\ {T}\cap {\rm int}(\Omega^\ell)\neq \emptyset\},\; \text{ for } 0\le \ell\le L.
\end{equation}
We also use an analogous decomposition for HB-splines, which reads
\[
{\cal V}^\ell_{{\cal H}\text{-{\rm supp}}} := {\rm span}\, {\cal H}^\ell_{\text{{\rm supp}}}, \quad \text{ with } {\cal H}^\ell_{\text{{\rm supp}}}:= \{{B}\in {\cal H}_{\bf p}({\mathcal Q}^{\ell}):\, {\rm supp}\ {B}\cap {\rm int}(\Omega^\ell)\neq \emptyset\},\; \text{ for } 0\le \ell\le L.
\]
We will show that, under suitable conditions of the mesh, the BPX preconditioner based on these subspaces provides a bounded condition number.


For comparison, we introduce two more decompositions, one more local and one completely global, that will be used in the numerical tests. The first one considers at each level only basis functions completely contained in $\Omega^\ell$, and recalling the notation introduced in Section~\ref{sec:HB-splines}, it is defined as
\[
{\cal V}^\ell_{\text{new}} := {\rm span}\, {\cal B}^{\ell,\ell},\; \text{ for } 0\le \ell\le L,
\]
which is the space generated by functions added to the basis at level $\ell$. The second one considers at each level all basis functions in the construction up to level $\ell$, and is given by
\[
{\cal V}^\ell_{\text{all}} := {\rm span}\, {\cal T}_{\bf p}({\mathcal Q}^{\ell}) = S_{\bf p}({\cal Q}^\ell),\; \text{ for } 0\le \ell\le L.
\]
Note that in this case the locality is lost, and as we will see this causes a dependence of the condition number with respect to the number of levels. In Figure~\ref{fig:decompositions} we show a simple example in one dimension to understand the different choices of the basis functions associated to these subspaces.
\begin{figure}
\subfigure[Basis of ${\cal V}^\ell_{{\cal T}\text{-supp}}$, level 0.]{
\includegraphics[width=0.3\textwidth, trim=3cm 1cm 3cm 0cm, clip]{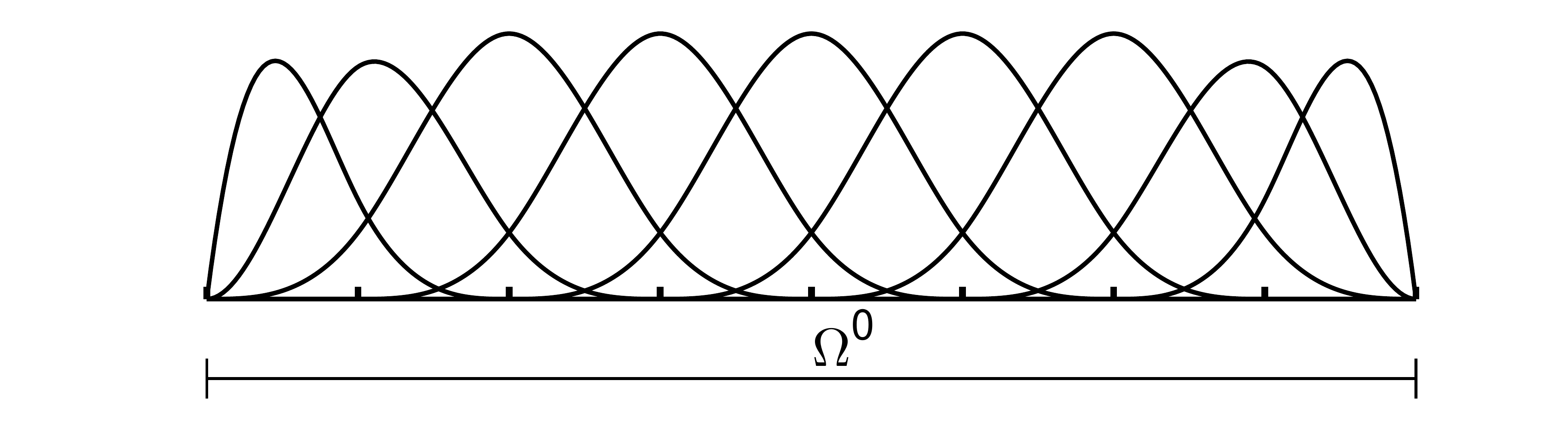}
}
\subfigure[Basis of ${\cal V}^\ell_{\rm mod}$, level 0.]{
\includegraphics[width=0.3\textwidth, trim=3cm 1cm 3cm 0cm, clip]{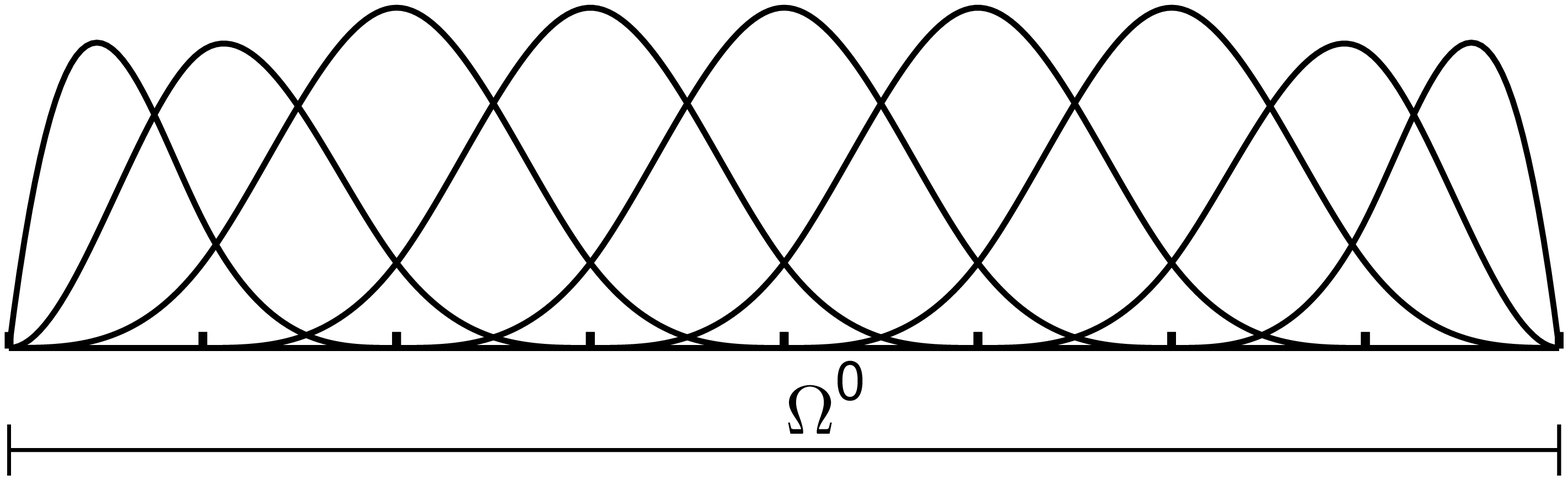}
}
\subfigure[Basis of ${\cal V}^\ell_{\rm new}$, level 0.]{
\includegraphics[width=0.3\textwidth, trim=3cm 1cm 3cm 0cm, clip]{functions_level0_domain}
}

\subfigure[Basis of ${\cal V}^\ell_{{\cal T}\text{-supp}}$, level 1.]{
\includegraphics[width=0.3\textwidth, trim=3cm 1cm 3cm 0cm, clip]{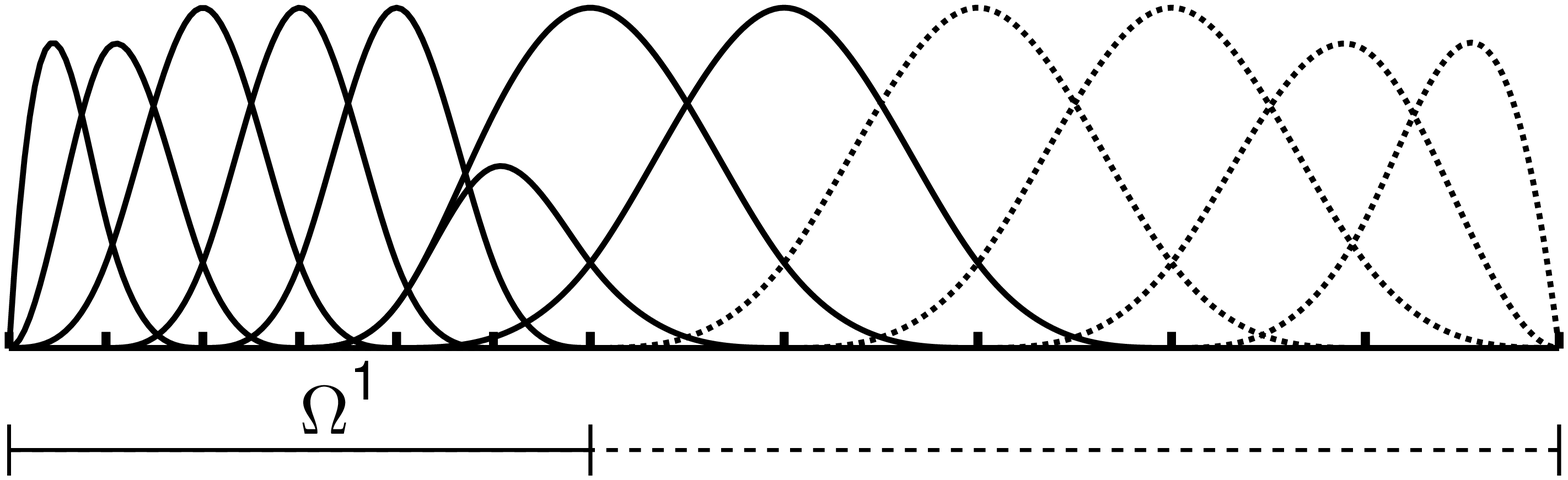}
}
\subfigure[Basis of ${\cal V}^\ell_{\rm mod}$, level 1.]{
\includegraphics[width=0.3\textwidth, trim=3cm 1cm 3cm 0cm, clip]{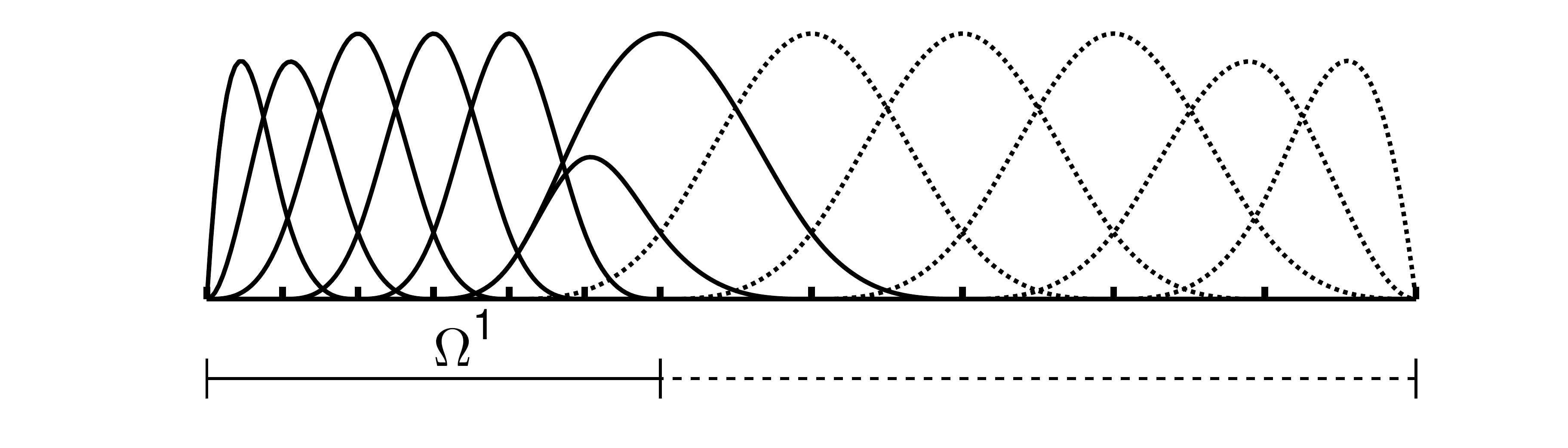}
}
\subfigure[Basis of ${\cal V}^\ell_{\rm new}$, level 1.]{
\includegraphics[width=0.3\textwidth, trim=3cm 1cm 3cm 0cm, clip]{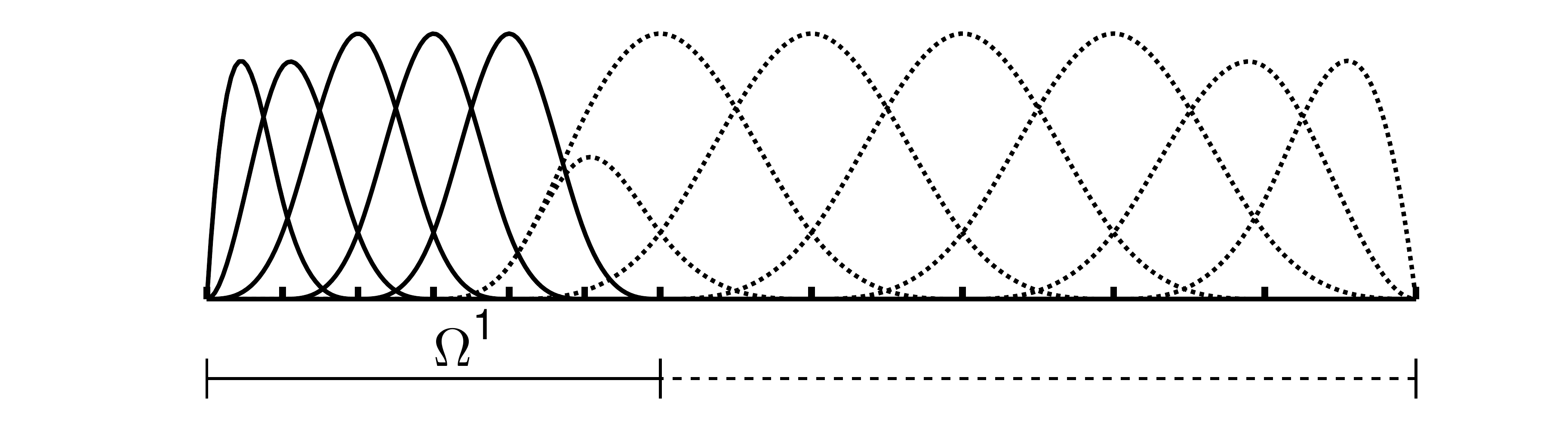}
}

\subfigure[Basis of ${\cal V}^\ell_{{\cal T}\text{-supp}}$, level 2.]{
\includegraphics[width=0.3\textwidth, trim=3cm 1cm 3cm 0cm, clip]{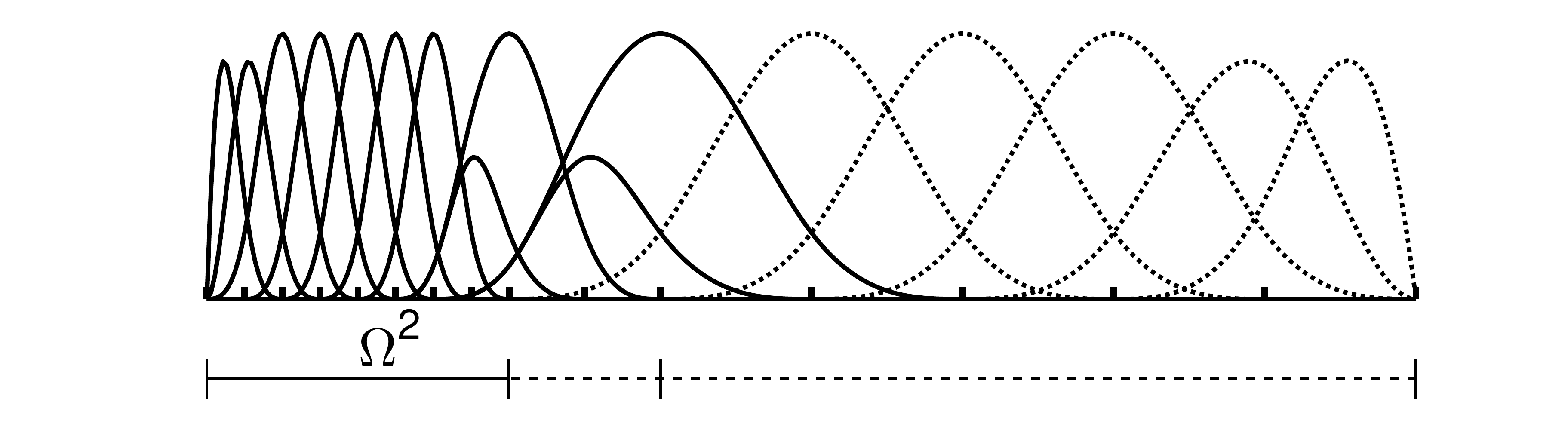}
}
\subfigure[Basis of ${\cal V}^\ell_{\rm mod}$, level 2.]{
\includegraphics[width=0.3\textwidth, trim=3cm 1cm 3cm 0cm, clip]{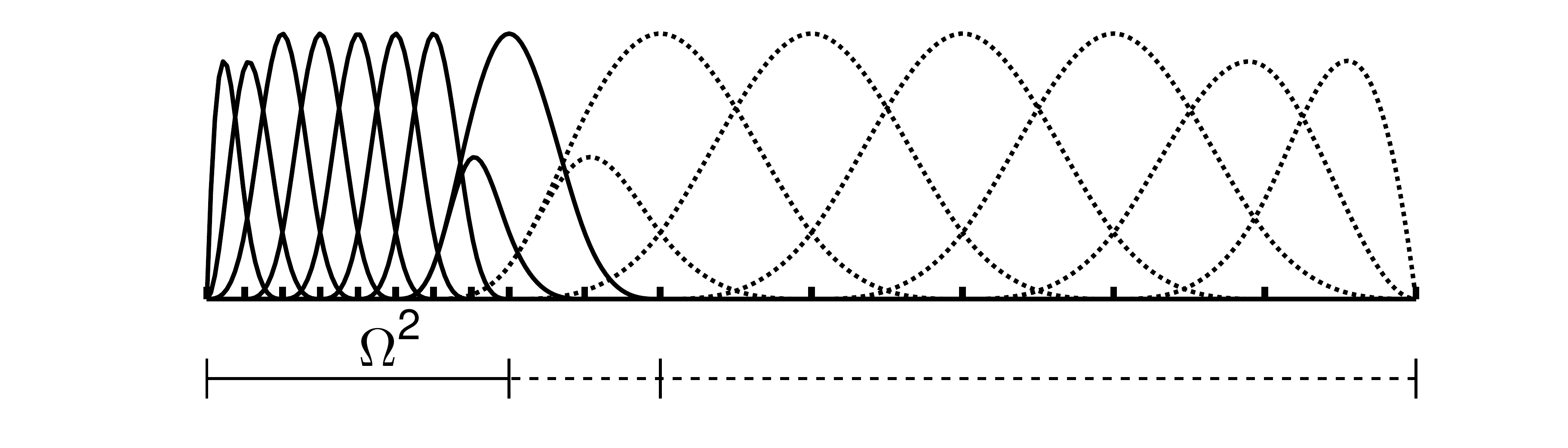}
}
\subfigure[Basis of ${\cal V}^\ell_{\rm new}$, level 2.]{
\includegraphics[width=0.3\textwidth, trim=3cm 1cm 3cm 0cm, clip]{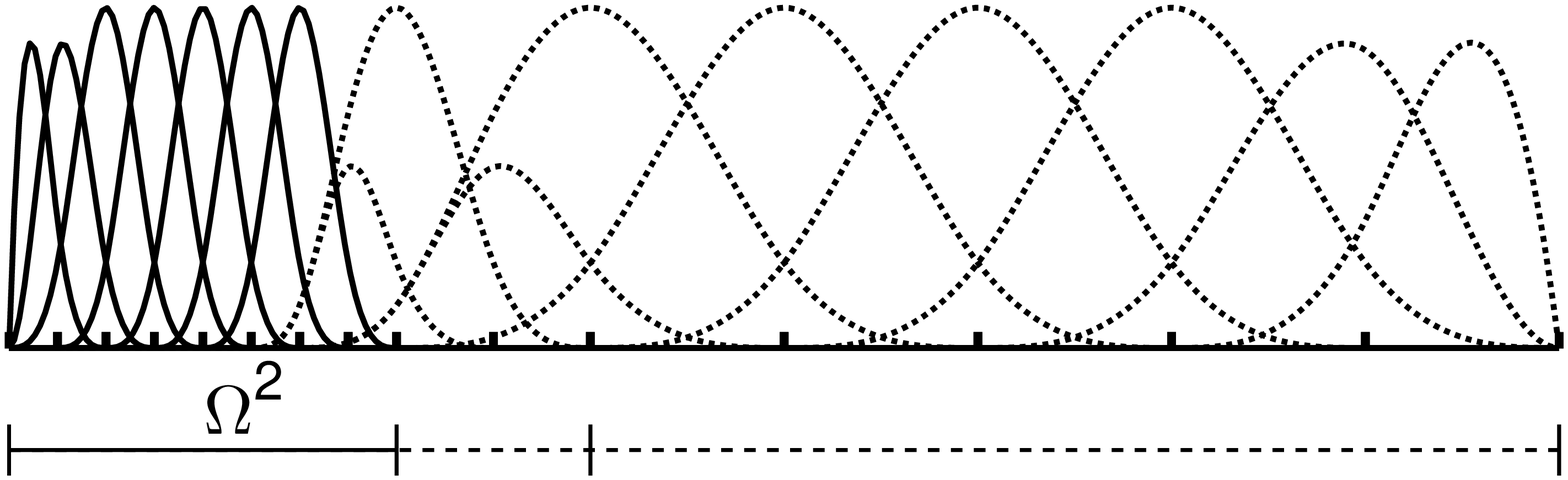}
}
\caption{Bases for the different subspaces of THB-splines. The solid lines represent functions in the basis of each subspace, while the dotted lines represent the remaining THB-spline basis functions. The mesh is represented by the ticks in the axis, and it is determined by $\Omega^0 = [0,1]$, $\Omega^1 = [0,3/8]$, $\Omega^2 = [0,2/8]$.} \label{fig:decompositions}
\end{figure}


The locality of the subspaces can be also understood thanks to the following proposition.
\begin{proposition} \label{prop:nestedness}
For any $0 \le \ell \le L$, it holds that
\[
{\cal V}^\ell_{\text{\rm new}} \subseteq {\cal V}^\ell_{\text{\rm mod}} \subseteq {\cal V}^\ell_{{\cal T}\text{-{\rm supp}}} \subseteq {\cal V}^\ell_{{\cal H}\text{-{\rm supp}}} \subseteq {\cal V}^\ell_{\text{\rm all}}.
\]
\end{proposition}
\begin{proof}
The inclusions are trivial from the definitions, except ${\cal V}^\ell_{\rm mod} \subseteq {\cal V}^\ell_{{\cal T}\text{-supp}}$, for which we observe that $\Phi^\ell$ collects the set of new and modified THB-splines in the transition from $\ell-1$ to $\ell$. Since the support of new THB-splines is fully contained in $\Omega^\ell$, we only need to prove it for the modified ones. Let $\phi \in \Phi^\ell$, with $\phi = \trunc^\ell \psi$ and $\psi \in \Psi^{\ell-1}$ the function before the last truncation. It is easy to prove, using the nestedness of the subdomains $\Omega^\ell \subseteq \Omega^{\ell-1}$ and basic aspects of the truncation mechanism, that writing $\psi$ as linear combination of B-splines of level $\ell-1$,
$\psi = \sum_{B \in  {\cal B}^{\ell-1}} \alpha_B B$,
there exists at least one function $B$, whose support intersects the interior of $\Omega^\ell$ but is not fully contained in $\Omega^\ell$, such that $\alpha_B \not = 0$. Since we assume that the continuity is at least $C^0$, the support of $\trunc^\ell (B)$ intersects $\Omega^\ell$, and the result follows.
\end{proof}

In order to prove the good behavior of the BPX preconditioner, we associate to all the subspaces of level $\ell$ in Proposition~\ref{prop:nestedness} the mesh size $h_\ell$, which is the same as for $S_{\bf p}({\bf \Xi}^\ell)$, the B-spline space of level $\ell$.

In the remaining part of this section we will analyze for which of the spaces above the properties \ref{assumption_smoother}-\ref{assumption_inverse} hold true. In fact, the inverse estimate \ref{assumption_inverse} holds for all the subspaces, and it is a trivial consequence of the inverse estimate for B-splines, see \cite[Theorem~4.2]{Bazilevs_Beirao_Cottrell_Hughes_Sangalli}. For the other three properties we will show the result on the biggest (or smallest) subspace for which we could prove it, and the others will follow by nestedness.

\subsection{Smoothing property for Jacobi and symmetric Gauss-seidel} \label{sec:smoother}

\input{smoother_property.tex}

\subsection{Stability of the decompositions} \label{sec:stability}
\input{decomp_stability_3B.tex}

\subsection{Strengthened Cauchy-Schwarz (SCS) inequality} \label{sec:scs}
\input{decomp_CS_3.tex}

\subsection{Final result and discussion} \label{sec:discussion}

We can now finally introduce the main result of the paper.
\begin{theorem}
Let the mesh ${\cal Q}$ be strictly ${\cal T}$-admissible. Then, the BPX preconditioner associated to subspaces ${\cal V}^\ell_{\text{mod}}$ and ${\cal V}^\ell_{{\cal T}\text{-supp}}$, with Jacobi or Gauss-Seidel smoother, satisfies $\kappa (BA) \lesssim 1$. The hidden constant depends on the degree $p$ and the admissibility class $m$.
\end{theorem}
\begin{proof}
As we mentioned above, the inverse inequality \ref{assumption_inverse} is a consequence of the one for B-splines, see \cite[Theorem~4.2]{Bazilevs_Beirao_Cottrell_Hughes_Sangalli}. The smoothing property \ref{assumption_smoother} comes from Proposition~\ref{prop:smoother} and Corollary~\ref{corol:smoother}, the stable decomposition \ref{assumption_decomposition} is proved in Theorem~\ref{stabledecomp_HBmeshes} and Corollary~\ref{stability_others}, while the SCS inequality \ref{assumption_scs} is proved in Theorem~\ref{SCS_HBmeshes} and Corollary~\ref{corol:SCS}. Then, the result follows from Theorem~\ref{thm:cond}.
\end{proof}
The result states that the preconditioner is optimal for the subspaces ${\cal V}^\ell_{\text{mod}}$ and ${\cal V}^\ell_{{\cal T}\text{-supp}}$. Actually, both can be seen as a generalization to the hierarchical spline setting of the preconditioner for finite elements in \cite{WuChen06,XCH10}. Although the former gives more local subspaces, and therefore smaller subproblems, the latter is simpler to implement, and in particular the indices of basis functions are obtained using the connectivity information restricted to elements in $\Omega^\ell$, as for assembling the matrices.

For the smallest subspace ${\cal V}^\ell_{\text{new}}$, we have not been able to prove the stability of the decomposition \ref{assumption_decomposition}, so the proof for the lower bound for the minimum eigenvalue is missing in this case.

Regarding the biggest subspace ${\cal V}^\ell_{\text{all}}$, as we have mentioned above both the smoothing property \ref{assumption_smoother} and the SCS inequality \ref{assumption_scs} are not valid, and the maximum eigenvalue is not bounded. This is confirmed by the numerical results in the next section.

Finally, for the subspace ${\cal V}^\ell_{{\cal H}\text{-supp}}$ we have not proved the smoothing property \ref{assumption_smoother} and the SCS inequality \ref{assumption_scs}. In fact, by requiring the mesh to be strictly ${\cal H}$-admissible, the results of Sections~\ref{sec:smoother} and \ref{sec:scs} can be proved analogously to the ${\cal V}^\ell_{{\cal T}\text{-supp}}$ case. We will see in the numerical tests that ${\cal T}$-admissibility of the mesh is not sufficient to bound the condition number independently of the number of levels in this case.

\begin{remark}
A decomposition similar to the one given by ${\cal V}^\ell_{\text{all}}$ was used in \cite{Hofreither2016b} for multiplicative multigrid methods, the difference being that the subspaces are chosen following the refinement of the adaptive procedure instead of the algorithm of Section~\ref{sec:HB-splines}. However, they also lack the locality property, which prevents the SCS inequality to hold in general. Note that both decompositions coincide if only elements of the finest level are allowed to be refined at each step. 
\end{remark}
\begin{remark}
Very recently, de Prenter et al. \cite{dePrenter2019} proposed a different decomposition of the space, with a construction from the finest to the coarsest level. Starting from the finest level ${\cal V}^L = {\cal V}$, the hierarchical mesh of each coarser level $\ell$ is defined as the coarsest mesh that can be obtained with a single level of derefinement, i.e., reactivating elements such that all their children of the next level are active, and the subspace ${\cal V}^\ell$ is the hierarchical space defined on this mesh. This decomposition is different from all the ones we define, and also from the one in \cite{Hofreither2016b}. There are two important issues for this decomposition that avoid to prove the robustness with respect to the number of levels: first, it lacks locality, since coarse functions can appear in principle in all levels, and second, the derefinement procedure does not respect admissibility. We presume that the decomposition can be changed without much difficulty to take into account the local refinement, while admissibility can be recovered using the coarsening algorithms in \cite{CGRV19}. However, we do not see any clear advantage with respect to the decompositions we have described above.
\end{remark}


%% file: smoother_property.tex

In this section, we show the smoothing property \ref{assumption_smoother} is satisfied, for the subspaces ${\cal V}^\ell_{{\cal T}\text{-supp}}$, when the smoother is one iteration of Jacobi or symmetric Gauss-Seidel. Since the Jacobi and the symmetric Gauss-Seidel smoothers are spectrally equivalent for any symmetric positive definite matrix (see e.g. \cite[Proposition 6.12]{Panayot} for details), we only need to prove the result for the Jacobi smoother. We remark that, as already mentioned in the introduction, none of them is an optimal smoother for B-splines, since the hidden constant in Theorem~\ref{thm:cond} depends on the degree, but the extension of optimal smoothers to the adaptive setting is beyond the scope of this work. We start with an auxiliary result.

\begin{lemma} \label{lemma:mass_matrix}
Let ${\cal Q}$ be a strictly ${\cal T}$-admissible hierarchical mesh of class $m$. For any $0 \le \ell \le L$ and for any $u \in {\cal V}^\ell_{{\cal T}\text{-{\rm supp}}}$, written as $\displaystyle u = \sum_{T \in {\cal T}^\ell_{\text{\rm supp}}} \eta_T T$, it holds that
\[
h_\ell^d \sum_{T \in {\cal T}^\ell_{\text{\rm supp}}} \eta_T^2 \lesssim \|u\|_0^2 \lesssim h_\ell^d \sum_{T \in {\cal T}^\ell_{\text{\rm supp}}} \eta_T^2,
\]
where the hidden constants depend on the degree and the admissibility class $m$, but are independent of the level $\ell$, and the number of levels. Moreover, as an immediate consequence we have
\[
\|T\|_0^2 \simeq h_\ell^d \, \text{ for all } T \in {\cal T}^\ell_{\rm supp}.
\]
\end{lemma}
\begin{proof}
We first prove the left inequality. Let $u = \sum_{T \in {\cal T}^\ell_{\text{\rm supp}}} \eta_T T$, and we recall from \eqref{eq:V-suppT} that the subspace ${\cal V}^\ell_{{\cal T}\text{-{\rm supp}}}$ is built from the hierarchical mesh ${\cal Q}^\ell$. For every element of level $\ell$, $Q \in {\cal Q}^\ell \cap G^\ell$, we define the set of basis functions ${\cal I}_Q := \{T \in {\cal T}^\ell_{\text{supp}} : \esupp(T) \cap Q \not = \emptyset \}$, and by admissibility the number of basis functions in ${\cal I}_Q$ is bounded. Then, recalling the definitions of the dual functionals in Section~\ref{sec:quasi-interpolant}, to each basis function $T \in {\cal I}_Q$ we associate an element $Q_T \subset \supp(T)$ on which we compute the local $L^2$-projection. By admissibility, we know that $Q_T \subset S^*(Q)$, and its size satisfies $h_{Q_T} \simeq h_\ell$, where the hidden constant depends on the admissibility class. Using these results, and the stability of the dual functionals \eqref{eq:stability_dual}, we get
\[
\sum_{T \in {\cal I}_Q} \eta_T^2 = \sum_{T \in {\cal I}_Q} |\lambda_T(u)|^2 \lesssim \sum_{T \in {\cal I}_Q} h_{Q_T}^{-d} \| u \|_{0,Q_T} \simeq h_\ell^{-d} \sum_{T \in {\cal I}_Q} \| u \|^2_{0,Q_T} \lesssim h_\ell^{-d} \|u\|^2_{0,S^*(Q)}.
\]
Noting that each function is supported in a bounded number of elements, a sum on the elements $Q \in {\cal Q}^\ell \cap G^\ell$, together with Proposition~\ref{prop:4incorr} and Corollary~\ref{corol:5incorr}, proves the first inequality.

For the right inequality, let us first define the set of (active and non-active) elements of level $\ell$ belonging to the supports of functions in ${\cal T}^\ell_{\text{supp}}$ as
\[
\tau^\ell:=\{Q\in G^\ell:\, Q\subseteq \supp(T),\, T\in{\cal V}^\ell_{{\cal T}\text{-{\rm supp}}}\},
\]
and for every element $Q \in \tau^\ell$ we define ${\cal I}_Q$ as above. Since the THB-splines are positive and a partition of unity, and since $\#{\cal I}_Q$ is bounded from the admissibility of the mesh, we get
\[
\int_Q u^2=\int_Q\Big(\sum_{T\in {\cal I}_Q} \eta_T~T\Big)^2\le \int_Q\Big(\sum_{T\in {\cal I}_Q} \eta_T\Big)^2\simeq h_\ell^d \Big(\sum_{T\in {\cal I}_Q} \eta_T\Big)^2\lesssim h_\ell^d\sum_{T\in {\cal I}_Q} \eta_T^2.
\]
Again, the result follows taking the sum for $Q \in \tau^\ell$ and using that, by admissibility, each function is supported in a bounded number of elements.
\end{proof}

\begin{corollary} \label{corol:mass}
If the mesh is strictly ${\cal T}$-admissible, the mass matrices associated to the subspaces ${\cal V}^\ell_{{\cal T}\text{-{\rm supp}}}$ are spectrally equivalent to the identity, with a constant independent of the number of levels, but which depends on the degree $p$ and the admissibility class $m$.
\end{corollary}
\begin{proof}
The result follows from $\|u\|_0^2 = \boldsymbol{\eta}_\ell^\top {\cal M}_\ell \boldsymbol{\eta}_\ell$, where ${\cal M}_\ell$ is the mass matrix associated to the basis ${\cal T}^\ell_{\rm supp}$, and $\boldsymbol{\eta}_\ell = [\eta_T]_{T \in {\cal T}^\ell_{\text{supp}}}$.
\end{proof}

We can now prove that \ref{assumption_smoother} holds for Jacobi smoother.

\begin{proposition}\label{prop:smoother}
Let ${\cal Q}$ be a strictly ${\cal T}$-admissible hierarchical mesh of class $m$, and let $R_\ell$ be the Jacobi smoother associated to the subspace ${\cal V}^\ell_{{\cal T}\text{-{\rm supp}}}$, for $0 \le \ell \le L$. Then it holds
\[
(R_\ell^{-1} u, u) \simeq h_\ell^{-2} \|u\|_0^2 \quad \forall u \in {\cal V}^\ell_{{\cal T}\text{-{\rm supp}}}.
\]
\end{proposition}
\begin{proof}
First we introduce some notation for the matrix representation of the smoothers $R_\ell$. By means of the basis of ${\cal V}^\ell_{{\cal T}\text{-{\rm supp}}}$, the analogues on each level of \eqref{linear_op} can be reduced to the following linear algebraic equation
\begin{equation}
{\mathcal A}_\ell \boldsymbol{\eta}_\ell =\boldsymbol{b}_\ell
\end{equation}
where ${\mathcal A}_\ell$ is the stiffness matrix and $\boldsymbol{\eta}_\ell$ is defined as in Corollary~\ref{corol:mass}. In what follows, we shall denote ${\mathcal A}_\ell={\mathcal D}_\ell-{\mathcal L}_\ell-{\mathcal U}_\ell$ where ${\mathcal D}_\ell$, ${\mathcal L}_\ell$ and ${\mathcal U}_\ell$ are the diagonal, lower triangle, and upper triangle part of ${\mathcal A}_\ell$, respectively. 

Let $u \in {\cal V}^\ell_{{\cal T}\text{-{\rm supp}}}$. On the one hand, from the definition of the Jacobi smoother and Corollary~\ref{corol:mass}, we have
\[
(R_\ell^{-1} u, u) = \boldsymbol{\eta}^\top {\cal M}_\ell^\top{\mathcal D}_\ell {\cal M}_\ell \boldsymbol{\eta} \simeq \boldsymbol{\eta}^\top{\mathcal D}_\ell \boldsymbol{\eta} = \sum_{T\in {\cal T}^\ell_{\text{\rm supp}}}
\|\eta_T T\|_A^2.
\]
On the other hand, by Lemma~\ref{lemma:mass_matrix} we obtain
\[
h_\ell^{-2} \|u\|_0^2
\simeq 
h_\ell^{d-2} \sum_{T\in {\cal T}^\ell_{\text{\rm supp}}} \eta_T^2 
\simeq 
h_\ell^{-2} \sum_{T\in {\cal T}^\ell_{\text{\rm supp}}} \|\eta_T T\|_0^2.
\]
Therefore, in order to obtain \ref{assumption_smoother}  it suffices to prove that, for any $T \in {\cal T}^\ell_{\rm supp}$, it holds
\[
\| T\|_A^2 \lesssim h_\ell^{-2} \| T\|_0^2 \lesssim \| T\|_A^2.
\]
The left inequality is just the inverse inequality \ref{assumption_inverse}, while the right inequality is a scaled Poincar\'e inequality, that has been already proved in \cite[Theorem~8]{BC2019}, noting that $\supp(T) \subseteq S^*(Q)$ for any element $Q$ contained in its support, and that $T$ vanishes on the boundary of $S^*(Q)$.
\end{proof}

By the nestedness result of Proposition~\ref{prop:nestedness}, the smoothing property is also valid for the smaller subspaces.
\begin{corollary}\label{corol:smoother}
Let the mesh be strictly ${\cal T}$-admissible. Then the Jacobi smoother associated to the subspaces ${\cal V}^\ell_{\text{{\rm new}}}$ and ${\cal V}^\ell_{\text{{\rm mod}}}$ also satisfies property \ref{assumption_smoother}.
\end{corollary}

\begin{remark}
The previous results are not valid for the subspaces ${\cal V}^\ell_{\rm all}$. See for instance the numerical results in \cite[Table~2]{Giannelli2016337} and in \cite[Figure~6]{bracco2018b} regarding the condition number of the mass matrix.
\end{remark}

%% file: decomp_stability_3B.tex
In this section we first prove the stability of the decomposition based on the subspaces ${\cal V}^\ell_{\text{mod}}$. The stability of other decompositions easily follows from it. We start by considering an auxiliary decomposition for the tensor-product space $S_\bp({\bf \Xi}^L$) associated to $\mathcal{V}$
\begin{equation*}
S_\bp({\bf \Xi}^L) =\sum_{\ell=0}^L S_\bp({\bf \Xi}^\ell),
\end{equation*}
which is well known to be stable, as stated in the following Lemma (see \cite{BHKS13} for details).

\begin{lemma}\label{stabledecom_quasiuniform}
For any $\bar{v}\in S_\bp({\bf \Xi}^L)$, let $\bar{v}_\ell=({\bf \Pi}_{\bp, {\bf \Xi}^\ell}-{\bf \Pi}_{\bp, {\bf \Xi}^{\ell-1}})\bar{v}$
for $\ell=0,\ldots,L$, setting ${\bf \Pi}_{\bp, {\bf \Xi}^{-1}}\bar{v}:= 0$. Then $\bar{v}=\sum_{\ell=0}^L \bar{v}_\ell$ is a stable decomposition in the sense that
\begin{equation*}
\sum_{\ell=0}^L h_\ell^{-2} \| \bar{v}_\ell \|^2_{0} \lesssim \|\bar{v}\|_A^2.
\end{equation*}
\end{lemma}

The second auxiliary result is a discrete Hardy inequality similar to Lemma~4.3 in \cite{CNX}, with the only difference being the shifting index $m$.
\begin{lemma}\label{modhardy}
If the non-negative sequences $\{a_k\}_{k=0}^L$ and $\{b_k\}_{k=0}^L$ satisfy for a certain $m\in \NN$
$$
b_k\le \sum_{\ell={\ell_{min}}}^L a_\ell, \qquad 0\le k\le L, \quad \ell_{min} := \max\{0,k-m+1\},
$$
then for any $s\in(0,1)$ we have
\begin{equation*}
\sum_{k=0}^L s^{-k}b_k\le \frac{1}{1-s}\sum_{k=0}^L s^{-k}a_k.
\end{equation*}
\end{lemma}

\begin{proof}Analogously to the original inequality in \cite{CNX}, denoting by $k_{\max} = \min\{L,m+\ell-1\}$, we note that
\begin{equation*}
\sum_{k=0}^L s^{-k}b_k\le \sum_{k=0}^L \sum_{\ell={\ell_{min}}}^L s^{-k}a_\ell=\sum_{\ell=0}^L\sum_{k=0}^{k_{\max}} s^{-k}a_\ell=\sum_{\ell=0}^L s^{-\ell}a_\ell \sum_{k=0}^{k_{\max}} s^{\ell-k}.
\end{equation*}
Since $s<1$, $\sum_{k=0}^{k_{\max}} s^{\ell-k}$ is bounded by $1/(1-s)$, which proves the result.
\end{proof}

After introducing these auxiliary results, we can prove the stability of the decomposition.
\begin{theorem}\label{stabledecomp_HBmeshes}
Let ${\cal Q}$ be a strictly ${\mathcal T}$-admissible hierarchical mesh of class $m$. For any $v\in\mathcal{V}$, there exist $v_\ell\in {\cal V}^\ell_{\text{mod}},\ \ell=0,\ldots,L$, such that $v= \sum_{\ell=0}^L v_\ell$ and
\begin{equation*}
\sum_{\ell=0}^L h_\ell^{-2} \|v_\ell\|^2_0 \lesssim \|v\|_A^2.
\end{equation*}
\end{theorem}

\begin{proof}
Note that in Lemma \ref{stabledecom_quasiuniform} the decomposition of an element $\bar{v}\in S_\bp({\bf \Xi}^L)$ is constructed by using the projectors ${\bf \Pi}_{\bp, {\bf \Xi}^\ell}$. Analogously, we use the sequence of quasi-interpolants ${\mathcal I}_{{\mathcal Q}^k}:{\mathcal V}\rightarrow S_{\bf p}({\mathcal Q}^k),\ 0\le k\le L$, as defined in Section~\ref{sec:quasi-interpolant},  to decompose $v\in {\cal V}$.

We define $v_k:=({\mathcal I}_{{\mathcal Q}^k}-{\mathcal I}_{{\mathcal Q}^{k-1}})v$, $\ 0\le k \le L$, with the convention that ${\mathcal I}_{{\mathcal Q}^{-1}}v:=0$, and we prove that $v_k\in {\cal V}^k_{\text{mod}}$. In fact, {since the dual functionals associated to functions in ${\cal T}_{\bf p}({\cal Q}^{k}) \cap {\cal T}_{\bf p}({\cal Q}^{k-1})$ do not change}, we have
\begin{align*}
v_k=&{\mathcal I}_{{\mathcal Q}^k}(v)-{\mathcal I}_{{\mathcal Q}^{k-1}}(v)=\sum_{T\in \Phi^k} \lambda_{T}(v) T -\sum_{T\in \Psi^{k-1}} \lambda_{T}(v) T,
\end{align*}
and thanks to \eqref{eq:psiCphi} we have proved that $v_k \in {\cal V}^k_{\text{mod}}$, for $k = 0, \ldots, L$.


Since $v\in {\cal V}\subseteq S_\bp({\bf \Xi}^L)$, we have $v=\sum_{\ell=0}^L \bar{v}_\ell$ with $\bar{v}_\ell:=({\bf \Pi}_{\bf p,\Xi^\ell} -{\bf \Pi}_{\bf p,\Xi^{\ell-1}})v \in S_\bp({\bf \Xi}^\ell)$. The next step is to prove that $({\mathcal I}_{{\mathcal Q}^k}-{\mathcal I}_{{\mathcal Q}^{k-1}})\bar{v}_\ell = 0$ for $0 \le \ell \le k-m$. Note that, since $\bar{v}_\ell \in S_\bp({\bf \Xi}^\ell)$, and $S_\bp({\bf \Xi}^i)\subseteq S_\bp({\bf \Xi}^j)$ for $i\le j$, we can prove it just for $\ell = k-m$.

For any $u \in S_\bp({\bf \Xi}^{k-m})$, for $k -m \ge 0$, we have that
\begin{align*}
({\mathcal I}_{{\mathcal Q}^k}-{\mathcal I}_{{\mathcal Q}^{k-1}})u\vert_{\Omega\setminus\Omega^{k}}
&= \sum_{T\in \Phi^k} \lambda_{T}(u) T\vert_{\Omega\setminus\Omega^{k}} -\sum_{T\in \Psi^{k-1}} \lambda_{T}(u) T\vert_{\Omega\setminus\Omega^{k}}\notag\\
&=\sum_{T\in \Phi^k: \supp(T)\not\subset \Omega^k} \lambda_{T}(u) T\vert_{\Omega\setminus\Omega^{k}} -\sum_{T\in \Psi^{k-1}: \supp(T)\not\subset \Omega^k} \lambda_{T}(u) T\vert_{\Omega\setminus\Omega^{k}},
\end{align*}
because 
$T\vert_{\Omega\setminus\Omega_k}=0$ for any $T\in \Phi^k\cup\Psi^{k-1}$ with $\supp(T)\subseteq \Omega^k$. By observing that for any of these $T_1\in \Phi^k$ there exists $T_2\in \Psi^{k-1}$ such that $T_1={\trunc}^k(T_2)$, which implies $T_1\vert_{\Omega\setminus\Omega^{k}}=T_2\vert_{\Omega\setminus\Omega^{k}}$, and that \eqref{hqi1b} holds, we get
\begin{equation}\label{18jan19a}
({\mathcal I}_{{\mathcal Q}^k}-{\mathcal I}_{{\mathcal Q}^{k-1}})u\vert_{\Omega\setminus\Omega^{k}}=0.
\end{equation}

We now observe that, from Proposition~\ref{admHT2}, ${\cal Q}^{k-1}$ and ${\cal Q}^{k}$ are strictly ${\cal T}$-admissible meshes of class $m$, {and by definition of admissibility $\Omega^k \subseteq \Omega^{k-1} \subseteq \omega_{\cal T}^{k-m}$, therefore the restrictions of ${{\cal S}_{\bf p}({\cal Q}^{k-1})}$ and ${{\cal S}_{\bf p}({\cal Q}^{k})}$ to $\Omega^k$ contain $S_\bp({\bf \Xi}^{k-m})$, see \cite[Section~4]{Kraft}.}
Since every ${\mathcal I}_{{\mathcal Q}^j}$ is a projector into ${{\cal S}_{\bf p}({\cal Q}^{j})}$, and $u\in S_\bp({\bf \Xi}^{k-m})$, we have
\begin{equation}\label{18jan19c}
({\mathcal I}_{{\mathcal Q}^k}-{\mathcal I}_{{\mathcal Q}^{k-1}})u\vert_{\Omega^{k}}=0.
\end{equation}
As a consequence of \eqref{18jan19a}--\eqref{18jan19c}, we get $({\mathcal I}_{{\mathcal Q}^k}-{\mathcal I}_{{\mathcal Q}^{k-1}})\bar{v}_{\ell}=0$ for $0\le \ell < \ell_{min}$, and therefore
\begin{equation*}
v_k=({\mathcal I}_{{\mathcal Q}^k}-{\mathcal I}_{{\mathcal Q}^{k-1}})v = ({\mathcal I}_{{\mathcal Q}^k}-{\mathcal I}_{{\mathcal Q}^{k-1}}) \sum_{\ell=\ell_{min}}^L \bar{v}_\ell.
\end{equation*}

%

Let us introduce the auxiliary subdomains $\sigma^k_{\rm mod} := \bigcup_{T \in {\Phi^k}} \supp(T)$. We observe that by applying Lemma~\ref{stabaux1} to $v_k$ and by using Corollary~\ref{corol:5incorr}, we get for $0\le k\le L$
\begin{align*}
\Vert v_k\Vert^2_{0}
&=\sum_{Q\subseteq \sigma^k_{\rm mod}} \|v_k\|^2_{0,Q}=\sum_{Q\subseteq \sigma^k_{\rm mod}}\left\Vert({\mathcal I}_{{\mathcal Q}^k}-{\mathcal I}_{{\mathcal Q}^{k-1}}) \sum_{\ell=\ell_{min}}^L \bar{v}_\ell\right\Vert^2_{0,Q} \\
& \lesssim \sum_{Q\subseteq \sigma^k_{\rm mod}} \left\|\sum_{\ell=\ell_{min}}^L \bar{v}_\ell \right\|_{0,S^*(Q)}^2 
\lesssim 
\left\|\sum_{\ell=\ell_{min}}^L \bar{v}_\ell \right\|^2_{0,S^*(\sigma^k_{\rm mod})}
=\left\|\sum_{\ell=\ell_{min}}^L \bar{v}_\ell \right\|_{0}^2\lesssim \sum_{\ell=\ell_{min}}^L \Vert\bar{v}_\ell\Vert^2_{0}.
\end{align*}

%

Note that, since we assume to use dyadic refinement between levels, we have $h_\ell\simeq 2^{-\ell}$. Then, by applying Lemma \ref{modhardy} to the sequences $\{a_k:=\Vert \bar v_k\Vert^2_{0}\}_{k=0}^L$ and $\{b_k:=\Vert v_k\Vert^2_{0}\}_{k=0}^L$ with constant $s=2^{-2}$ and Lemma \ref{stabledecom_quasiuniform}, we get
\begin{equation*}
\sum_{\ell=0}^L h_{\ell}^{-2}\Vert v_\ell\Vert^2_{0}\lesssim \sum_{\ell=0}^L h_{\ell}^{-2}\Vert \bar v_\ell\Vert^2_{0}\lesssim \|{v}\|_{A}^2,
\end{equation*}
which proves the theorem. 
\end{proof}


Once we have proved the stability of the decomposition based on the subspaces ${\cal V}^\ell_{\text{mod}}$, the stability for other subspaces follows immediately by the nestedness property of Proposition~\ref{prop:nestedness}, as stated in the following corollary.
\begin{corollary}\label{stability_others}
Let ${\cal Q}$ be a strictly ${\mathcal T}$-admissible hierarchical mesh of class $m$. Then, there exist stable decompositions also for the subspaces ${\cal V}^\ell_{{\cal T}\text{-supp}}, {\cal V}^\ell_{{\cal H}\text{-supp}}$ and ${\cal V}^\ell_{\text{all}}$.
\end{corollary}

%% file: decomp_CS_3.tex
The proof of the SCS inequality for THB-splines relies on two auxiliary results: the SCS inequality in the tensor-product case, which is proved in \cite[Lemma~5.3]{CV19}, and an auxiliary estimate that was first proved in \cite[Lemma 3.4]{XCN}. We start recalling these two required results.

\begin{lemma} [SCS inequality for B-splines on globally quasi-uniform meshes]\label{SCS_quasiuniform}
Let $\{G^i\}_{0\le i\le L}$ be quasi-uniform over the domain $\Omega$. For $u_i \in S_\bp({\bf \Xi}^i)$, $u_j \in S_\bp({\bf \Xi}^j)$ with $j \ge i$ and each element $Q_i \in G^i$, we have
\[
a(u_i,u_j)_{Q_i} \lesssim \gamma^{(j-i)/2} \|u_i\|_{A,Q_i} h_j^{-1}\|u_j\|_{0,Q_i},
\]
and
\[
a(u_i,u_j) \lesssim \gamma^{(j-i)/2} \|u_i\|_A \, h_j^{-1}\|u_j\|_{0},
\]
where $0 <\gamma < 1$ is a constant such that $h_i\simeq \gamma^{i}$.
\end{lemma}

\begin{lemma}\label{prop4.4}
Given any $(x_i)^n_{i=1}$ and $(y_i)^n_{i=1}$ in ${\mathbb R}^n$, and $0 < \gamma < 1$, we have
\[
\sum_{i,j=1}^n \gamma^{|i-j|}x_i y_j \le \frac{2}{1-\gamma}\left(\sum_{i=1}^n x_i^2\right)^{1/2} \left(\sum_{j=1}^n y_j^2\right)^{1/2}.
\]
\end{lemma}
We are now in a position to prove the main result of this section.

\begin{theorem} [SCS inequality for THB-splines]\label{SCS_HBmeshes}
Let ${\cal Q}$ be a strictly ${\mathcal T}$-admissible hierarchical mesh of class $m$. For any $u_\ell,v_\ell\in {\mathcal V}^\ell_{{\cal T}\text{-supp}}$
, with $0\le \ell \le L$, we have
\[
\left| \sum_{i=0}^L \sum_{j=i+1}^L a(u_i,v_j) \right|
\lesssim \left( \sum_{i=0}^L \|u_i\|_A^2 \right)^{1/2} \left( \sum_{j=0}^L h_j^{-2} \|v_j\|_{0} \right)^{1/2}.
\]
\end{theorem}
\begin{proof}

For $0 \le \ell \le L$ we introduce the auxiliary subdomains $\sigma^\ell_{\rm supp} := \bigcup_{T \in {\cal T}^\ell_{\rm supp}} \supp(T)$.
From the definition of ${\mathcal V}^\ell_{{\cal T}\text{-supp}}$ in \eqref{eq:V-suppT}
for any $u_\ell \in {\mathcal V}^\ell_{{\cal T}\text{-supp}}$ we trivially have $u_\ell\in S_\bp({\bf \Xi}^\ell)$ and $\mbox{supp}(u_\ell)\subset \sigma^\ell_{\rm supp}$. Now, for $j > i$ let us define
\[
G^{i,j}:= \{Q_i\in G^i : Q_i\subset {\sigma}^i_{\rm supp} \, \wedge \, Q_i\cap {\sigma}^j_{\rm supp} \neq \emptyset\}.
\]
Let $u_i \in {\mathcal V}^i_{{\cal T}\text{-supp}}$ and $v_j \in {\mathcal V}^j_{{\cal T}\text{-supp}}$. Applying Lemma~\ref{SCS_quasiuniform} followed by the usual Cauchy-Schwarz inequality, and noting that $\supp(v_j) \subset {\sigma}^j_{\rm supp}$, we get
\begin{equation*}
\begin{array}{rl}
a(u_i,v_j) = & \displaystyle \sum_{Q_i\, \in \, G^{i,j} } a(u_i,v_j)_{Q_i} 
\lesssim  \gamma^{(j-i)/2} \sum_{Q_i\, \in \, G^{i,j}} \left(\|u_i\|_{A,Q_i}~ h_j^{-1} \|v_j\|_{0,Q_i}\right)\\
\lesssim & \displaystyle \gamma^{(j-i)/2} \|u_i\|_{A,{\sigma}^i_{\rm supp}} \, h_j^{-1} \left(\sum_{Q_i\, \in \, G^{i,j}} \|v_j\|_{0,Q_i}^2\right)^{1/2}
= \gamma^{(j-i)/2} \|u_i\|_{A,{\sigma}^i_{\rm supp}} \, h_j^{-1} \|v_j\|_{0,{\sigma}^j_{\rm supp}},
\end{array}
\end{equation*}
for some $0 < \gamma < 1$ as in Lemma~\ref{SCS_quasiuniform}.
Taking the sums on the indices $i$ and $j$, and applying Lemma~\ref{prop4.4}, we obtain
\begin{eqnarray*}
&\left| \sum_{i=0}^L a(u_i,\sum_{j=i+1}^L v_j)  \right| 
\le  \sum_{i=0}^L\sum_{j=i+1}^L \left|a(u_i,v_j)  \right| \lesssim \sum_{i=0}^L \sum_{j=i+1}^L \left( \gamma^{(j-i)/2}\|u_i\|_{A,{\sigma}^i_{\rm supp}} \, h_j^{-1}\|v_j\|_{0,{\sigma}^j_{\rm supp}}\right) \\
&\lesssim  \left(\sum_{i=0}^L \|u_i\|^2_{A,{\sigma}^i_{\rm supp}}\right)^{1/2} \left( \sum_{j=0}^L h_j^{-2}\|v_j\|^2_{0,{\sigma}^j_{\rm supp}} \right)^{1/2} = \left(\sum_{i=0}^L \|u_i\|^2_{A}\right)^{1/2} \left( \sum_{j=0}^L h_j^{-2}\|v_j\|^2_{0} \right)^{1/2},
\end{eqnarray*}
which is the desired result.
\end{proof}

From the nestedness of the spaces in Proposition~\ref{prop:nestedness}, we also have the following result.
\begin{corollary}\label{corol:SCS}
Let ${\cal Q}$ be a strictly ${\cal T}$-admissible mesh of class $m$. Then, the SCS inequality \ref{assumption_scs} holds for the ${\cal V}^\ell_{\text{new}}$ and ${\cal V}^\ell_{\text{mod}}$ subspaces.
\end{corollary}

\begin{remark} \label{rem:scs-alldofs}
It is worth to remark that, in general, the SCS inequality does not hold if we consider the subspaces ${\cal V}^\ell_{\text{all}}$. Indeed, in that case the analogues of the subdomains $\sigma^i_{\rm supp}$ would always be equal to $\Omega$, and as a consequence the hidden constant in Theorem~\ref{SCS_HBmeshes} would increase with the number of levels. A similar result, with a constant depending on the number of levels, is obtained for the subspaces ${\cal V}^\ell_{{\cal H}\text{-supp}}$ of HB-splines on strictly ${\cal T}$-admissible meshes. 
\end{remark}

%% file: numerical.tex
For the numerical tests, we have implemented the BPX preconditioner in the Octave/Matlab software GeoPDEs \cite{GEOPDES-NEW}. We refer to \cite{garau2018} for the details of the implementation of THB-splines, and to \cite{bracco2018b} for the algorithms regarding admissible refinement. In all the tests, we have computed the condition number as the quotient between the maximum and minimum eigenvalues, that are approximated using Lanczos' method \cite[Section~6.6]{Saad_book} for the preconditioned system, and the Matlab command \emph{eigs} for the unpreconditioned one. All the numerical tests are run considering as the smoother one single iteration of symmetric Gauss-Seidel method.

\paragraph{Test 1: the role of the decomposition}
In the first set of numerical tests we consider the parametric domain $(0,1)^d$ for dimensions $d=1,2,3$. We test the behavior of the BPX preconditioner with the degree ranging from one to four. For each degree $p$ we start with a zero level mesh of $(2p+1)^d$ elements, and at each refinement step, passing from level $\ell$ to level $\ell+1$, we refine the hyperrectangular region near the origin formed by $(p+2^\ell)^d$ elements, see Figure~\ref{fig:mesh_test1} for examples in dimension two after three refinement steps. This kind of refinement leaves a ``frame'' of $p$ elements between non-consecutive levels, and guarantees that the mesh is strictly ${\cal T}$-admissible with $m=2$.
\begin{figure}[ht]
\centering
\subfigure[Degree 1.]{
\includegraphics[width=0.34\textwidth,trim=1cm 1cm 1cm 0cm, clip]{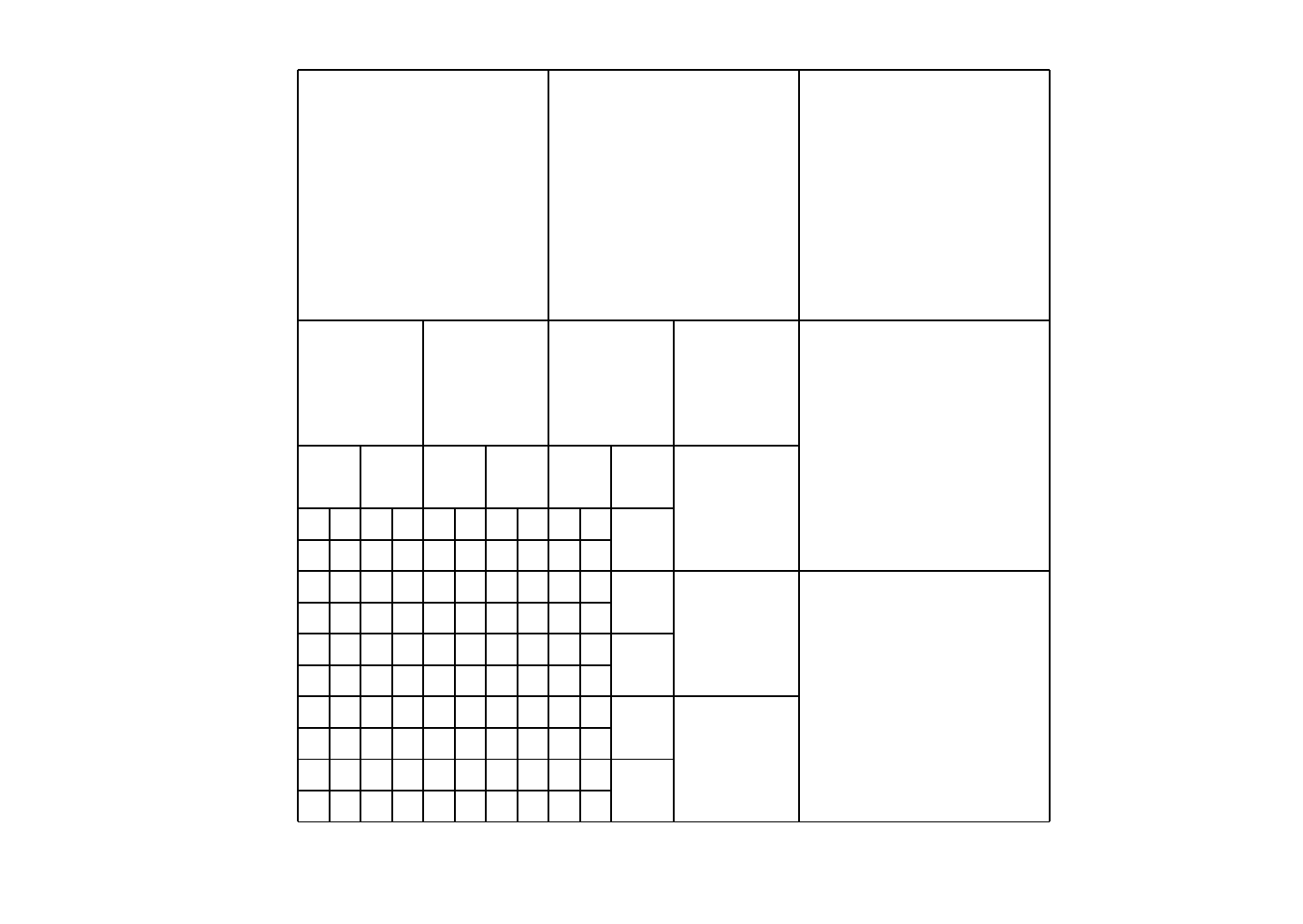}}
\subfigure[Degree 2.]{
\includegraphics[width=0.34\textwidth,trim=1cm 1cm 1cm 0cm, clip]{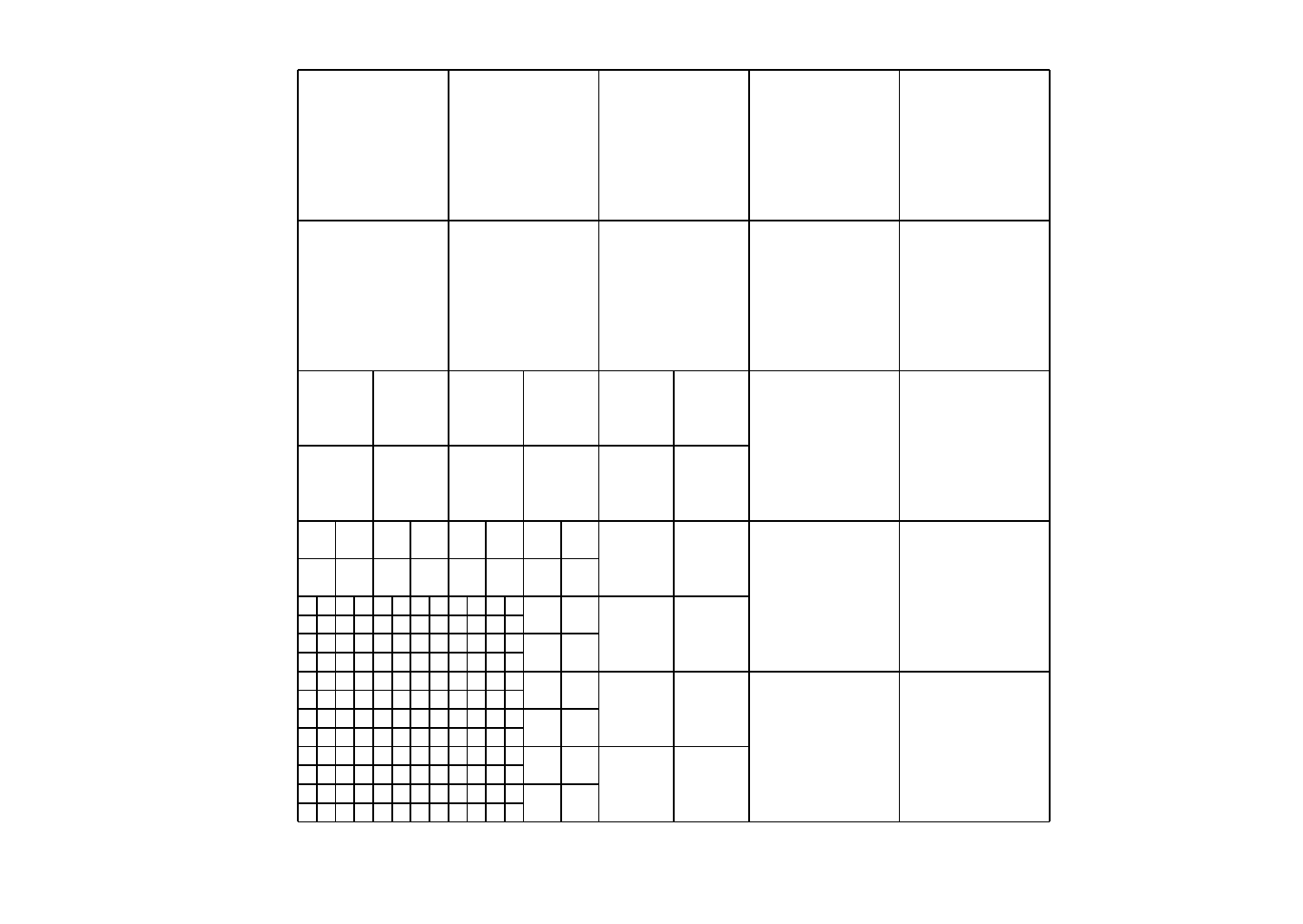}}
\subfigure[Degree 3.]{
\includegraphics[width=0.34\textwidth,trim=1cm 1cm 1cm 0cm, clip]{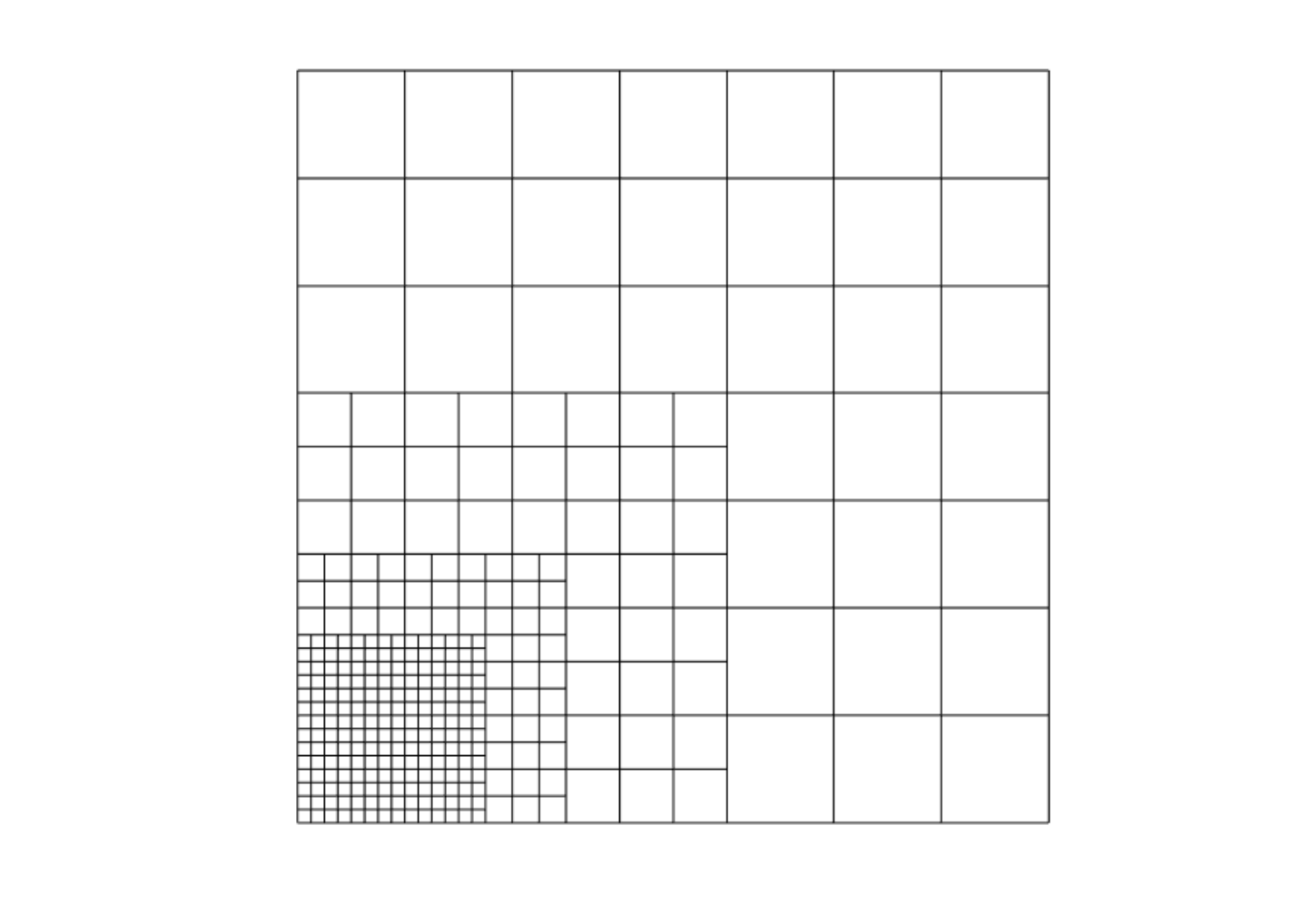}}
\subfigure[Degree 4.]{
\includegraphics[width=0.34\textwidth,trim=1cm 1cm 1cm 0cm, clip]{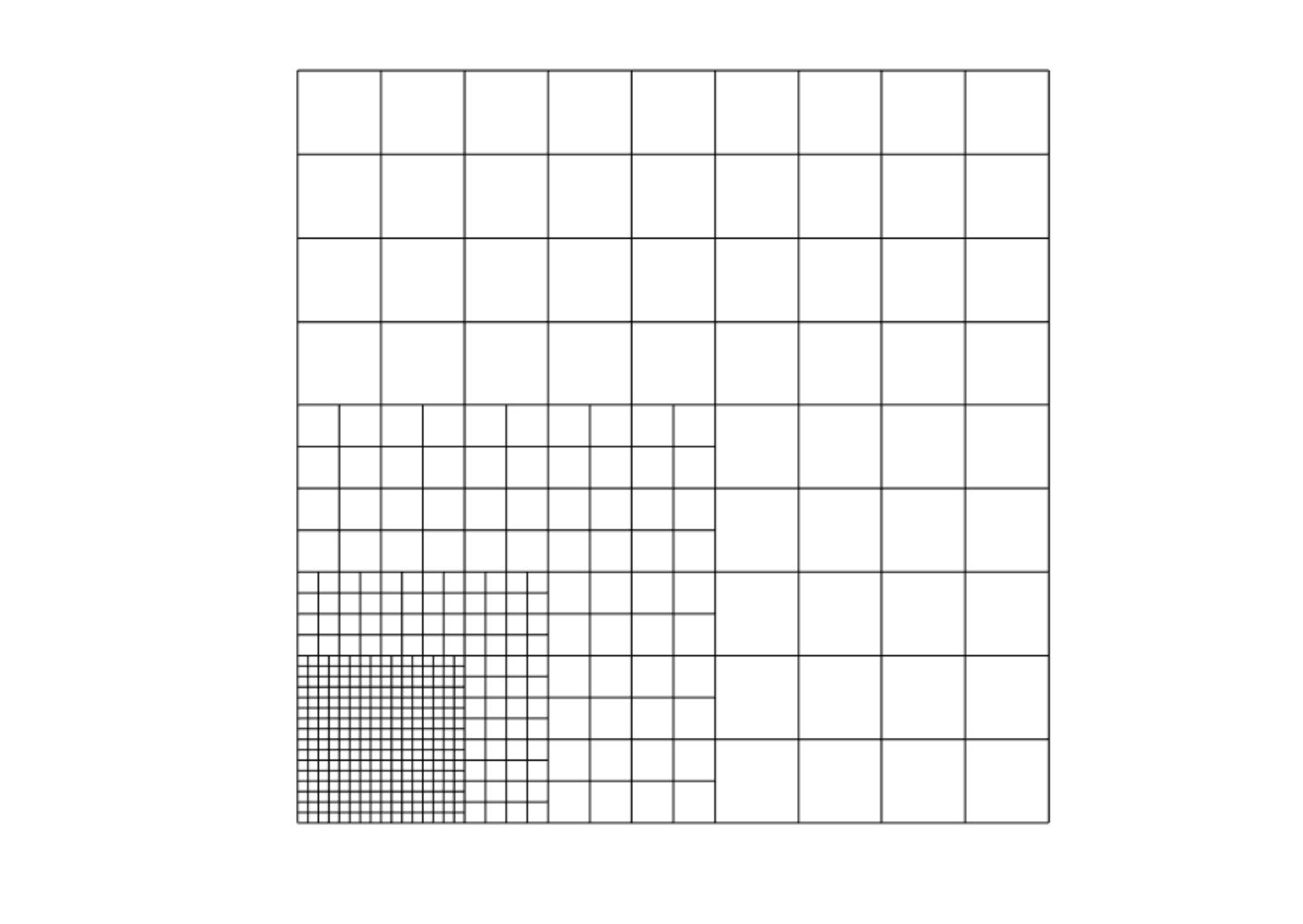}}
\caption{Meshes used in Test 1, for different degrees and $d=2$, after three refinement steps.} \label{fig:mesh_test1}
\end{figure}

We have tested the performance of the BPX preconditioner for THB-splines with the different decompositions introduced in Section~\ref{sec:decompositions}, except for the one based on ${\cal V}^\ell_{{\cal H}\text{-supp}}$, that will be tested later. We start with the case of $d=2$ and compare in Figure~\ref{fig:test1-alldecomps} the value of the condition number of the unpreconditioned linear system with the values obtained for the BPX preconditioner with the different decompositions, for degrees from one to four. We see that in all cases the BPX preconditioner reduces the condition number, however the condition number obtained for the subspaces ${\cal V}^\ell_{\text{new}}$ is greater than for any other, specially for high degree. Moreover, the condition number obtained for the subspaces ${\cal V}^\ell_{\text{all}}$ grows with the number of levels. A detailed analysis of the minimum and maximum eigenvalues, given in Tables~\ref{tab:test1-eigv-all}-\ref{tab:test1-eigv-new} for dimension $d=2$ and for all the degrees, shows that the maximum eigenvalue of the preconditioned system with the subspaces ${\cal V}^\ell_{\text{all}}$ is not bounded, as it was already mentioned in Remark~\ref{rem:scs-alldofs}. In agreement with our theoretical results, the minimum and maximum eigenvalues, and as a consequence the condition number, are bounded for the choices ${\cal V}^\ell_{\text{mod}}$ and ${\cal V}^\ell_{{\cal T}\text{-supp}}$. This seems to be also the case for the choice ${\cal V}^\ell_{\text{new}}$, even if we have not proved the bound for the minimum eigenvalue. Howeover, for this choice the minimum eigenvalue is much lower than in the other cases, providing bigger condition numbers. It is also worth to note that in all cases the magnitude of the minimum eigenvalue decreases with the degree due to the lack of robustness of the Gauss-Seidel smoother. Although both ${\cal V}^\ell_{\text{mod}}$ and ${\cal V}^\ell_{{\cal T}\text{-supp}}$ show good performance, in the following tests we will focus on the latter due to its simplicity, and we will compare it with the performance of ${\cal V}^\ell_{\text{all}}$.

We also show in Figure~\ref{fig:test-square} the results for the other dimensions, and compare the behavior of the choice of local subspaces ${\cal V}^\ell_{{\cal T}\text{-supp}}$ with the one given by global subspaces ${\cal V}^\ell_{\text{all}}$. The results confirm what we have observed in dimension $d=2$: the condition number with the choice of local subspaces is bounded, while the one given by the global ones is not.


\begin{figure}[htbp]
\input{figures-test1-alldecomp.tex}
\end{figure}

\begin{table}[ht]
\centering
\begin{tabular}{|c||c|c||c|c||c|c||c|c|}
\hline
\multirow{2}{*}{Level} & \multicolumn{2}{|c||}{$p=1$} & \multicolumn{2}{|c||}{$p=2$} & \multicolumn{2}{|c||}{$p=3$} & \multicolumn{2}{|c|}{$p=4$}\\
\cline{2-9}
 & $\lambda_{\min}$ & $\lambda_{\max}$ & $\lambda_{\min}$ & $\lambda_{\max}$ & $\lambda_{\min}$ & $\lambda_{\max}$ & $\lambda_{\min}$ & $\lambda_{\max}$ \\
\hline 
2 & 7.9e-01 & 2.0e+00 & 7.7e-01 & 2.0e+00 & 3.2e-01 & 2.0e+00 & 7.5e-02 & 2.0e+00 \\ 
3 & 7.7e-01 & 3.0e+00 & 8.6e-01 & 3.0e+00 & 3.1e-01 & 3.0e+00 & 5.2e-02 & 3.0e+00 \\ 
4 & 7.4e-01 & 4.0e+00 & 8.6e-01 & 4.0e+00 & 2.9e-01 & 4.0e+00 & 4.7e-02 & 4.0e+00 \\ 
5 & 7.3e-01 & 5.0e+00 & 8.7e-01 & 5.0e+00 & 2.8e-01 & 5.0e+00 & 4.3e-02 & 5.0e+00 \\ 
6 & 7.3e-01 & 6.0e+00 & 8.7e-01 & 6.0e+00 & 2.8e-01 & 6.0e+00 & 4.2e-02 & 6.0e+00 \\ 
7 & 7.3e-01 & 7.0e+00 & 8.5e-01 & 7.0e+00 & 2.8e-01 & 7.0e+00 & 4.2e-02 & 7.0e+00 \\ 
8 & 7.2e-01 & 8.0e+00 & 8.7e-01 & 8.0e+00 & 2.8e-01 & 8.0e+00 & 4.2e-02 & 8.0e+00 \\ 
9 & 7.2e-01 & 9.0e+00 & 8.2e-01 & 9.0e+00 & 2.8e-01 & 9.0e+00 & 4.2e-02 & 9.0e+00 \\ 
10 & 7.2e-01 & 1.0e+01 & 8.3e-01 & 1.0e+01 & 2.8e-01 & 1.0e+01 & 4.2e-02 & 1.0e+01 \\ 
\hline 
\end{tabular}
\caption{Minimum and maximum eigenvalues for the preconditioned system for Test 1, $d=2$, for ${\cal V}^\ell_{\text{all}}$.} \label{tab:test1-eigv-all}
\end{table}

\begin{table}[ht]
\centering
\begin{tabular}{|c||c|c||c|c||c|c||c|c|}
\hline
\multirow{2}{*}{Level} & \multicolumn{2}{|c||}{$p=1$} & \multicolumn{2}{|c||}{$p=2$} & \multicolumn{2}{|c||}{$p=3$} & \multicolumn{2}{|c|}{$p=4$}\\
\cline{2-9}
 & $\lambda_{\min}$ & $\lambda_{\max}$ & $\lambda_{\min}$ & $\lambda_{\max}$ & $\lambda_{\min}$ & $\lambda_{\max}$ & $\lambda_{\min}$ & $\lambda_{\max}$ \\
\hline 
2 & 7.9e-01 & 2.0e+00 & 7.7e-01 & 2.0e+00 & 3.2e-01 & 2.0e+00 & 7.5e-02 & 2.0e+00 \\ 
3 & 7.7e-01 & 2.4e+00 & 8.7e-01 & 2.9e+00 & 3.1e-01 & 2.9e+00 & 5.2e-02 & 2.9e+00 \\ 
4 & 7.4e-01 & 2.8e+00 & 8.9e-01 & 3.5e+00 & 2.9e-01 & 3.8e+00 & 4.7e-02 & 3.7e+00 \\ 
5 & 7.2e-01 & 3.1e+00 & 8.9e-01 & 3.6e+00 & 2.9e-01 & 4.3e+00 & 4.4e-02 & 4.5e+00 \\ 
6 & 7.2e-01 & 3.3e+00 & 8.7e-01 & 3.7e+00 & 2.9e-01 & 4.6e+00 & 4.2e-02 & 5.0e+00 \\ 
7 & 7.1e-01 & 3.6e+00 & 8.7e-01 & 3.8e+00 & 2.9e-01 & 4.7e+00 & 4.2e-02 & 5.3e+00 \\ 
8 & 7.1e-01 & 3.8e+00 & 8.6e-01 & 3.8e+00 & 2.9e-01 & 4.8e+00 & 4.2e-02 & 5.5e+00 \\ 
9 & 7.2e-01 & 4.0e+00 & 8.8e-01 & 3.8e+00 & 2.9e-01 & 4.9e+00 & 4.1e-02 & 5.6e+00 \\ 
10 & 7.1e-01 & 4.1e+00 & 8.9e-01 & 3.8e+00 & 2.9e-01 & 4.9e+00 & 4.1e-02 & 5.7e+00 \\ 
\hline 
\end{tabular}
\caption{Minimum and maximum eigenvalues for the preconditioned system for Test 1, $d=2$, for ${\cal V}^\ell_{{\cal T}\text{-supp}}$.} \label{tab:test1-eigv-supp}
\end{table}

\begin{table}[h!]
\centering
\begin{tabular}{|c||c|c||c|c||c|c||c|c|}
\hline
\multirow{2}{*}{Level} & \multicolumn{2}{|c||}{$p=1$} & \multicolumn{2}{|c||}{$p=2$} & \multicolumn{2}{|c||}{$p=3$} & \multicolumn{2}{|c|}{$p=4$}\\
\cline{2-9}
 & $\lambda_{\min}$ & $\lambda_{\max}$ & $\lambda_{\min}$ & $\lambda_{\max}$ & $\lambda_{\min}$ & $\lambda_{\max}$ & $\lambda_{\min}$ & $\lambda_{\max}$ \\
\hline 
2 & 7.9e-01 & 2.0e+00 & 6.5e-01 & 2.0e+00 & 3.2e-01 & 2.0e+00 & 7.4e-02 & 2.0e+00 \\ 
3 & 7.7e-01 & 2.4e+00 & 6.1e-01 & 2.9e+00 & 2.8e-01 & 2.9e+00 & 4.6e-02 & 2.9e+00 \\ 
4 & 7.4e-01 & 2.8e+00 & 6.0e-01 & 3.4e+00 & 2.7e-01 & 3.8e+00 & 4.2e-02 & 3.7e+00 \\ 
5 & 7.2e-01 & 3.1e+00 & 6.0e-01 & 3.6e+00 & 2.7e-01 & 4.3e+00 & 4.1e-02 & 4.5e+00 \\ 
6 & 7.2e-01 & 3.3e+00 & 6.0e-01 & 3.7e+00 & 2.5e-01 & 4.5e+00 & 3.8e-02 & 5.0e+00 \\ 
7 & 7.1e-01 & 3.6e+00 & 6.0e-01 & 3.7e+00 & 2.5e-01 & 4.7e+00 & 3.8e-02 & 5.3e+00 \\ 
8 & 7.1e-01 & 3.8e+00 & 6.0e-01 & 3.8e+00 & 2.5e-01 & 4.8e+00 & 3.8e-02 & 5.5e+00 \\ 
9 & 7.2e-01 & 4.0e+00 & 5.9e-01 & 3.8e+00 & 2.5e-01 & 4.9e+00 & 3.7e-02 & 5.6e+00 \\ 
10 & 7.1e-01 & 4.1e+00 & 5.9e-01 & 3.8e+00 & 2.5e-01 & 4.9e+00 & 3.7e-02 & 5.7e+00 \\ 
\hline 
\end{tabular}
\caption{Minimum and maximum eigenvalues for the preconditioned system for Test 1, $d=2$, for ${\cal V}^\ell_{\text{mod}}$.} \label{tab:test1-eigv-mod}
\end{table}

\begin{table}[h!]
\centering
\begin{tabular}{|c||c|c||c|c||c|c||c|c|}
\hline
\multirow{2}{*}{Level} & \multicolumn{2}{|c||}{$p=1$} & \multicolumn{2}{|c||}{$p=2$} & \multicolumn{2}{|c||}{$p=3$} & \multicolumn{2}{|c|}{$p=4$}\\
\cline{2-9}
 & $\lambda_{\min}$ & $\lambda_{\max}$ & $\lambda_{\min}$ & $\lambda_{\max}$ & $\lambda_{\min}$ & $\lambda_{\max}$ & $\lambda_{\min}$ & $\lambda_{\max}$ \\
\hline 
2 & 7.3e-01 & 1.8e+00 & 2.7e-01 & 2.0e+00 & 5.8e-02 & 2.0e+00 & 8.0e-03 & 2.0e+00 \\ 
3 & 6.5e-01 & 2.3e+00 & 2.6e-01 & 2.8e+00 & 4.5e-02 & 2.9e+00 & 4.4e-03 & 2.9e+00 \\ 
4 & 6.6e-01 & 2.7e+00 & 2.5e-01 & 3.3e+00 & 4.4e-02 & 3.7e+00 & 4.1e-03 & 3.7e+00 \\ 
5 & 6.5e-01 & 3.0e+00 & 2.5e-01 & 3.5e+00 & 4.3e-02 & 4.2e+00 & 3.9e-03 & 4.5e+00 \\ 
6 & 6.4e-01 & 3.3e+00 & 2.5e-01 & 3.7e+00 & 4.3e-02 & 4.5e+00 & 3.9e-03 & 5.0e+00 \\ 
7 & 6.4e-01 & 3.6e+00 & 2.5e-01 & 3.7e+00 & 4.3e-02 & 4.7e+00 & 3.8e-03 & 5.3e+00 \\ 
8 & 6.4e-01 & 3.8e+00 & 2.4e-01 & 3.8e+00 & 4.3e-02 & 4.8e+00 & 3.8e-03 & 5.5e+00 \\ 
9 & 6.3e-01 & 4.0e+00 & 2.4e-01 & 3.8e+00 & 4.3e-02 & 4.9e+00 & 3.8e-03 & 5.6e+00 \\ 
10 & 6.3e-01 & 4.1e+00 & 2.4e-01 & 3.8e+00 & 4.3e-02 & 4.9e+00 & 3.8e-03 & 5.7e+00 \\ 
\hline 
\end{tabular}
\caption{Minimum and maximum eigenvalues for the preconditioned system for Test 1, $d=2$, for ${\cal V}^\ell_{\text{new}}$.} \label{tab:test1-eigv-new}
\end{table}

\begin{figure}[htbp]
\input{figures-test1-decomp.tex}
\end{figure}

\paragraph{Test 2: the role of admissibility}
For the second numerical test we consider again the parametric domain, in this case focusing on the two-dimensional case. We start from an initial mesh of $9\times 9$ elements, independently of the degree, and at each refinement step we refine the region $(0,1/3)^2$, in such a way that $\Omega^{\ell+1} = \Omega^{\ell}$ for any $\ell > 0$, see Figure~\ref{fig:mesh-test2}. In this case the resulting mesh is not admissible, and therefore the theoretical results of Section~\ref{sec:decompositions} are not valid anymore.

\begin{figure}[htb]
\centering
\subfigure[Mesh with three levels, $m=\infty$.]{
\includegraphics[width=0.34\textwidth,trim=1cm 1cm 1cm 0cm, clip]{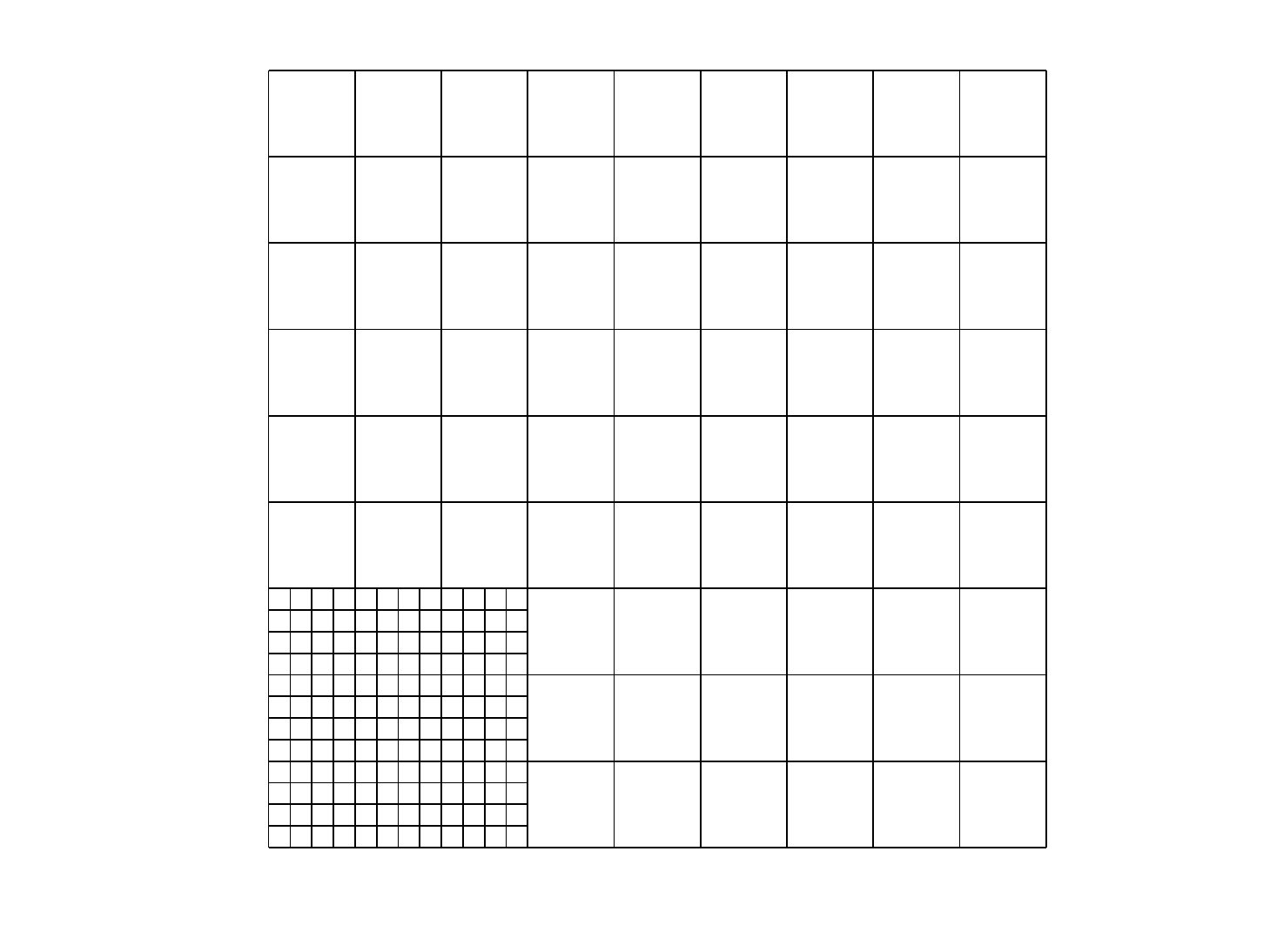}}
\subfigure[Mesh with four levels, $m=\infty$.]{
\includegraphics[width=0.34\textwidth,trim=1cm 1cm 1cm 0cm, clip]{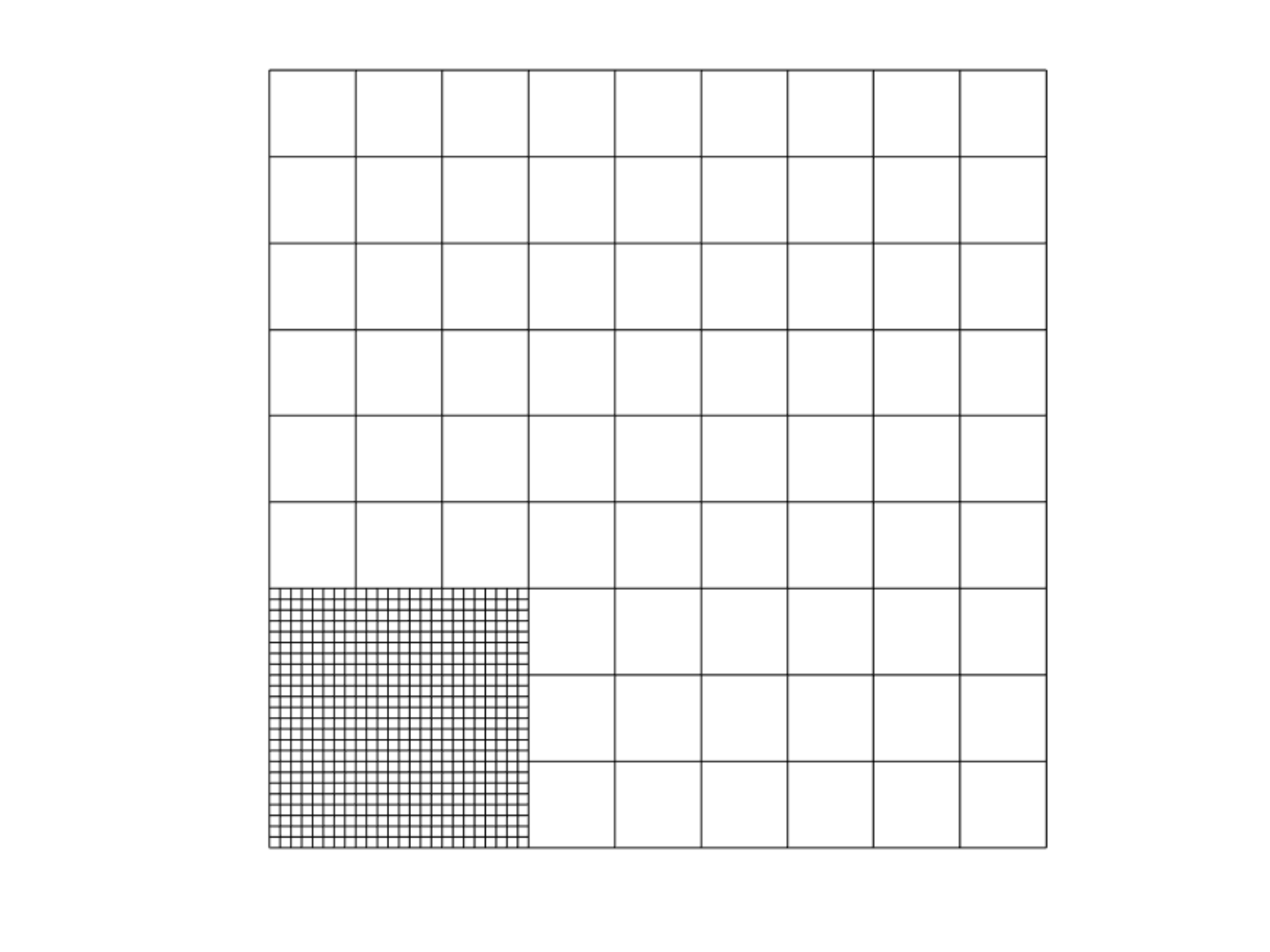}}
\caption{Meshes used in Test 2, and also in Test 3 and Test 4 for $m=\infty$, with three and four levels.} \label{fig:mesh-test2}
\end{figure}

We repeat the same numerical tests that we run for the previous example. We first compare, in Figure~\ref{fig:test-non-admissiblea}, the condition number of the unpreconditioned system with the one obtained with the BPX preconditioner, and the decomposition given by ${\cal V}^\ell_{{\cal T}\text{-supp}}$. Even if the theoretical results do not hold, and the condition number increases with the number of levels, the preconditioner does a good work in reducing it. We show in Figure~\ref{fig:test-non-admissibleb} a comparison of the results between the decompositions based on subspaces ${\cal V}^\ell_{{\cal T}\text{-supp}}$ and ${\cal V}^\ell_{\text{all}}$. In this case both decompositions provide very similar results, because the spaces in the decompositions only differ slighty. We note however that the size of the spaces in ${\cal V}^\ell_{{\cal T}\text{-supp}}$ is always smaller, which leads to solve smaller linear systems when applying the smoother on each level.

\begin{figure}[htb]
\input{figures-test2-admissibility.tex}
\end{figure}

As before, we complete the results with the analysis of the minimum and maximum eigenvalues, that are displayed in Table~\ref{tab:test2-eigv}. The results for the subspaces ${\cal V}^\ell_{\text{all}}$ are very similar to the ones obtained in the previous test, and we observe an increase of the maximum eigenvalue with respect to the number of levels. A similar behavior is now also observed for the decomposition based on subspaces ${\cal V}^\ell_{{\cal T}\text{-supp}}$, while the maximum eigenvalue was bounded in the previous numerical test. From this results we conclude that the lack of admissibility causes the maximum eigenvalue to increase with the number of levels. Instead, we do not observe the minimum eigenvalue to decrease despite the lack of admissibility.
\begin{table}[ht]
\centering
\begin{tabular}{|c|c||c|c||c|c||c|c||c|c|}
\hline
\multirow{2}{*}{Level} & \multirow{2}{*}{Decomp.} & \multicolumn{2}{|c||}{$p=1$} & \multicolumn{2}{|c||}{$p=2$} & \multicolumn{2}{|c||}{$p=3$} & \multicolumn{2}{|c|}{$p=4$}\\
\cline{3-10}
& & $\lambda_{\min}$ & $\lambda_{\max}$ & $\lambda_{\min}$ & $\lambda_{\max}$ & $\lambda_{\min}$ & $\lambda_{\max}$ & $\lambda_{\min}$ & $\lambda_{\max}$ \\
\hline 
\multirow{2}{*}{2} 
&All & 8.1e-01 & 2.0e+00 & 8.0e-01 & 2.0e+00 & 3.5e-01 & 2.0e+00 & 1.1e-01 & 2.0e+00 \\ 
&Support & 7.6e-01 & 2.0e+00 & 8.1e-01 & 2.0e+00 & 3.5e-01 & 2.0e+00 & 1.1e-01 & 2.0e+00 \\ 
\hline 
\multirow{2}{*}{3} 
&All & 7.8e-01 & 3.0e+00 & 8.6e-01 & 3.0e+00 & 3.2e-01 & 3.0e+00 & 8.2e-02 & 3.0e+00 \\ 
&Support & 7.3e-01 & 2.9e+00 & 8.3e-01 & 3.0e+00 & 3.0e-01 & 3.0e+00 & 8.2e-02 & 3.0e+00 \\ 
\hline 
\multirow{2}{*}{4} 
&All & 7.6e-01 & 4.0e+00 & 8.5e-01 & 4.0e+00 & 3.4e-01 & 4.0e+00 & 6.9e-02 & 4.0e+00 \\ 
&Support & 7.3e-01 & 3.6e+00 & 8.3e-01 & 4.0e+00 & 3.2e-01 & 4.0e+00 & 6.9e-02 & 3.9e+00 \\ 
\hline 
\multirow{2}{*}{5} 
&All & 7.4e-01 & 5.0e+00 & 8.0e-01 & 5.0e+00 & 3.3e-01 & 5.0e+00 & 6.4e-02 & 5.0e+00 \\ 
&Support & 7.2e-01 & 4.2e+00 & 8.0e-01 & 4.9e+00 & 3.1e-01 & 4.9e+00 & 6.4e-02 & 4.9e+00 \\ 
\hline 
\multirow{2}{*}{6} 
&All & 7.3e-01 & 6.0e+00 & 7.7e-01 & 6.0e+00 & 3.2e-01 & 6.0e+00 & 6.2e-02 & 6.0e+00 \\ 
&Support & 7.2e-01 & 4.5e+00 & 7.6e-01 & 5.8e+00 & 3.1e-01 & 5.8e+00 & 6.2e-02 & 5.8e+00 \\ 
\hline 
\multirow{2}{*}{7} 
&All & 7.2e-01 & 7.0e+00 & 7.3e-01 & 7.0e+00 & 3.1e-01 & 7.0e+00 & 6.1e-02 & 7.0e+00 \\ 
&Support & 7.2e-01 & 4.8e+00 & 7.3e-01 & 6.5e+00 & 3.1e-01 & 6.7e+00 & 6.1e-02 & 6.8e+00 \\ 
\hline 
\multirow{2}{*}{8} 
&All & 7.2e-01 & 8.0e+00 & 6.9e-01 & 8.0e+00 & 3.2e-01 & 8.0e+00 & 6.0e-02 & 8.0e+00 \\ 
&Support & 7.2e-01 & 5.0e+00 & 6.9e-01 & 7.0e+00 & 3.2e-01 & 7.7e+00 & 6.0e-02 & 7.7e+00 \\ 
\hline 
\end{tabular}
\caption{Results for Test 2: non-admissible meshes. Minimum and maximum eigenvalues for the preconditioned system using the decompositions based on ${\cal V}^\ell_{{\cal T}\text{-supp}}$ and ${\cal V}^\ell_{\text{all}}$.} \label{tab:test2-eigv}
\end{table}

\paragraph{Test 3: the role of the admissibility class}
We complete the previous test by checking the dependence on the admissibility class $m$. Starting from the same mesh of $9\times 9$ elements, and marking at each step the same elements as before, we apply the admissible refinement algorithm \cite{BC2016,bracco2018b} for strictly ${\cal T}$-admissible meshes with different values of $m$, see Figure~\ref{fig:mesh-test3}. Note that the results without applying admissible refinement, that we denote by $m=\infty$, are the same as in Test~2. In this case we only consider the subspaces ${\cal V}^\ell_{{\cal T}\text{-supp}}$. The results for degrees two and three, that we present in Figure~\ref{fig:test2-p2-p3}, show that the admissibility class plays a minor role in the condition number of the preconditioned system, and even when the mesh is not admissible the condition number is significantly reduced.

\begin{figure}[htb]
\centering
\subfigure[Mesh for $p=2, m=3$.]{
\includegraphics[width=0.32\textwidth,trim=1cm 1cm 1cm 0cm, clip]{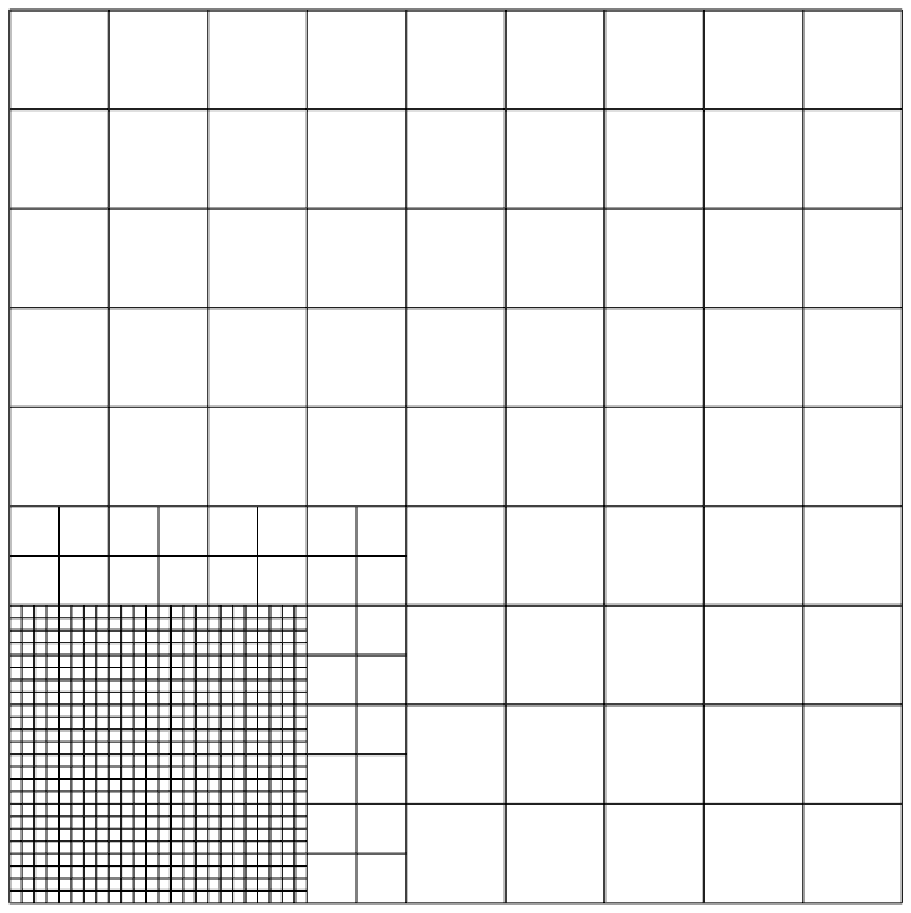}}
\subfigure[Mesh for $p=3, m=3$.]{
\includegraphics[width=0.32\textwidth,trim=1cm 1cm 1cm 0cm, clip]{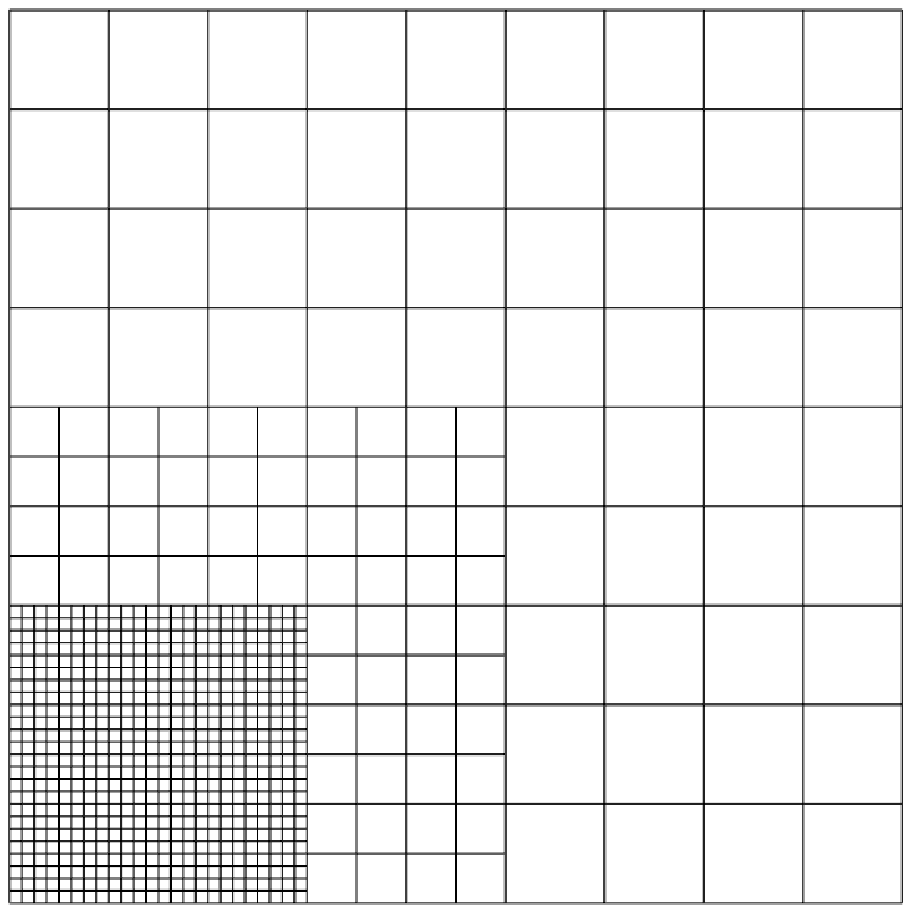}}
\subfigure[Mesh for $p=3, m=2$.]{
\includegraphics[width=0.32\textwidth,trim=1cm 1cm 1cm 0cm, clip]{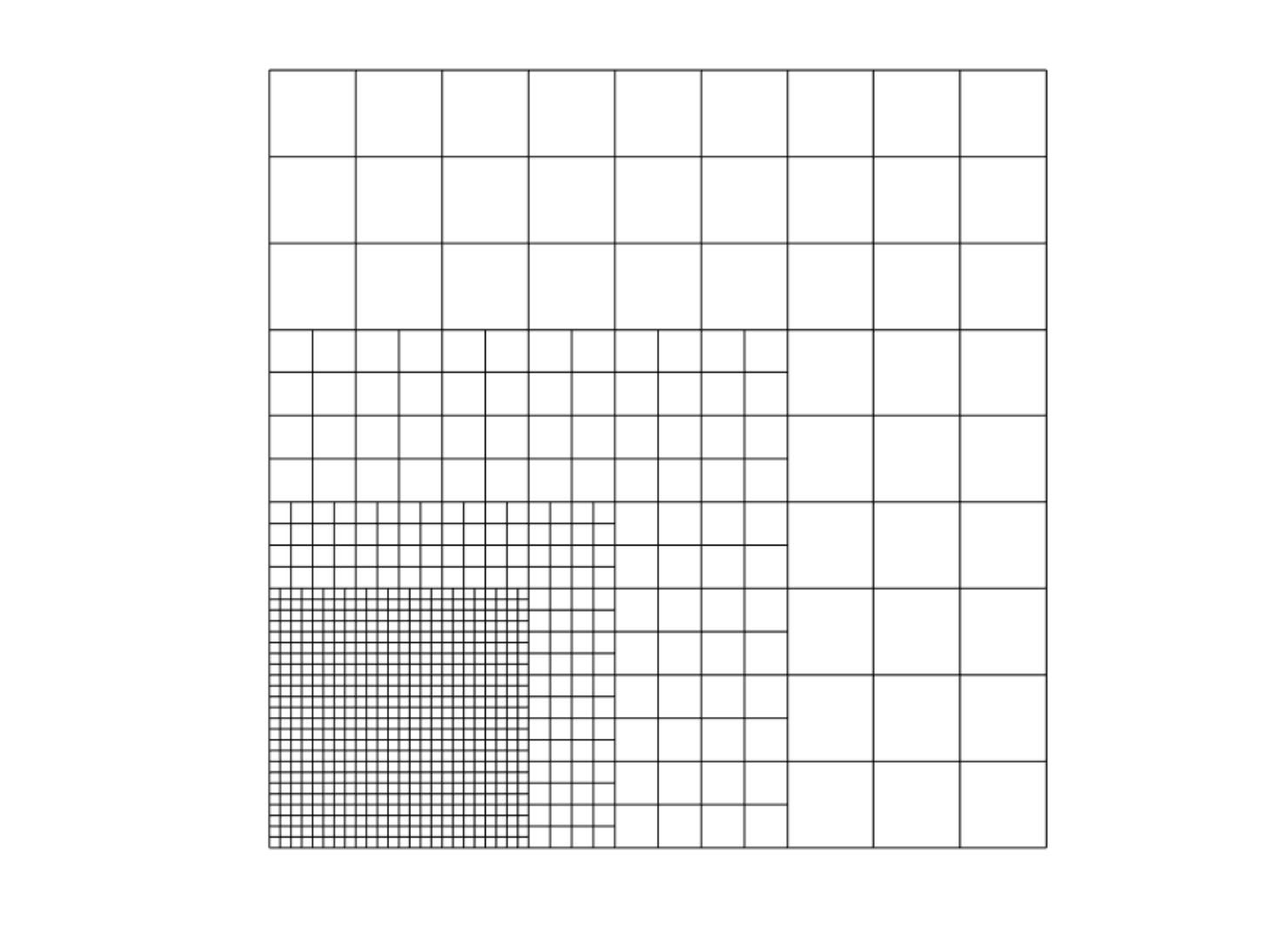}}
\caption{Some strictly ${\cal T}$-admissible meshes used in Test 3 and in Test 4, for different degrees $p$ and admissibility class $m$, with four levels.} \label{fig:mesh-test3}
\end{figure}

\begin{figure}[htb]
\input{figures-test3-admclass.tex}
\end{figure}

\paragraph{Test 4: a comparison of HB-splines and THB-splines}
To compare the behavior of HB-splines and THB-splines, we run a numerical test very similar to the previous one. We consider the same domain and initial mesh, and at each refinement step we mark exactly the same region $(0,1/3)^2$, applying the admissible refinement algorithms from \cite{bracco2018b}, in this case both for strictly ${\cal H}$-admissible and strictly ${\cal T}$-admissible meshes. Some strictly ${\cal H}$-admissible meshes are depicted in Figure~\ref{fig:mesh-test-HB-THB}. We consider the two decompositions based on the support, that is, for THB-splines the decomposition is given by ${\cal V}^\ell_{{\cal T}\text{-supp}}$, and for HB-splines it is given by ${\cal V}^\ell_{{\cal H}\text{-supp}}$.
\begin{figure}[htb]
\centering
\subfigure[Mesh for $p=2, m=3$.]{
\includegraphics[width=0.32\textwidth,trim=1cm 1cm 1cm 0cm, clip]{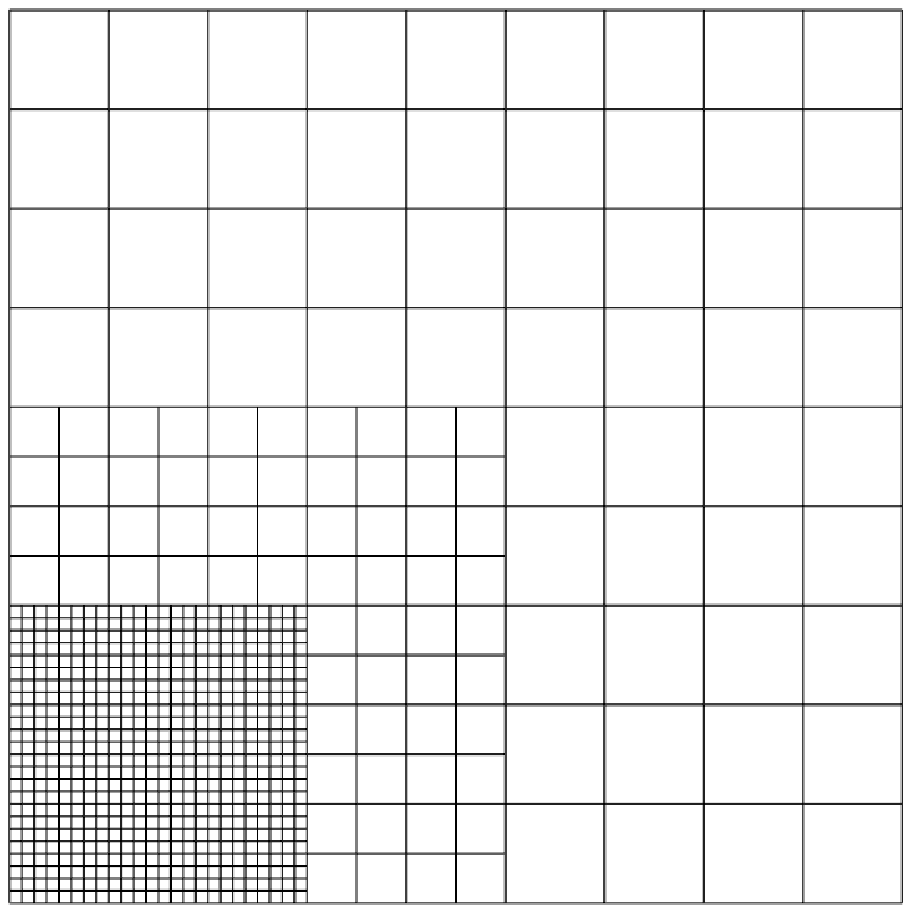}}
\subfigure[Mesh for $p=3, m=3$.]{
\includegraphics[width=0.32\textwidth,trim=1cm 1cm 1cm 0cm, clip]{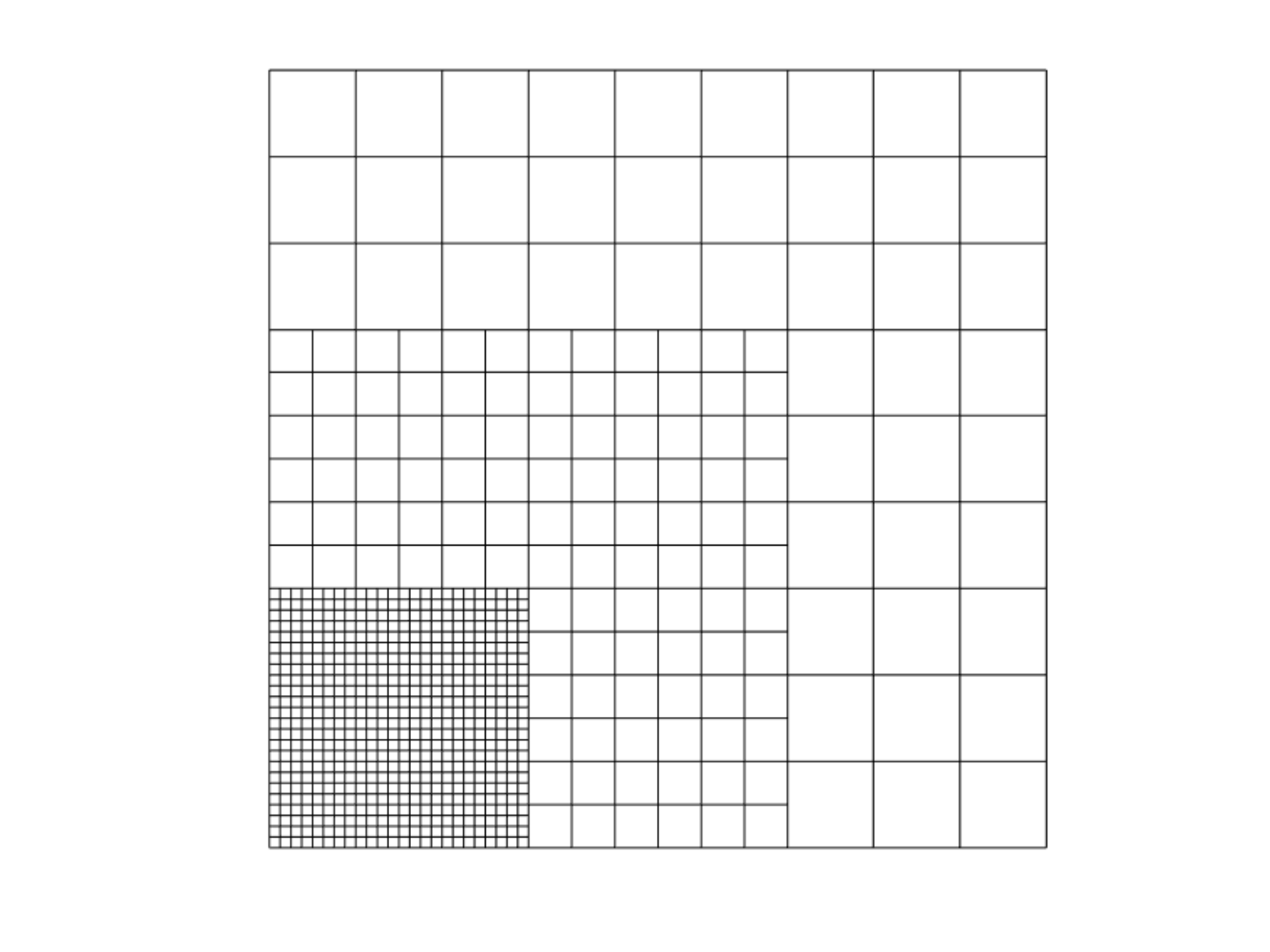}}
\subfigure[Mesh for $p=3, m=2$.]{
\includegraphics[width=0.32\textwidth,trim=1cm 1cm 1cm 0cm, clip]{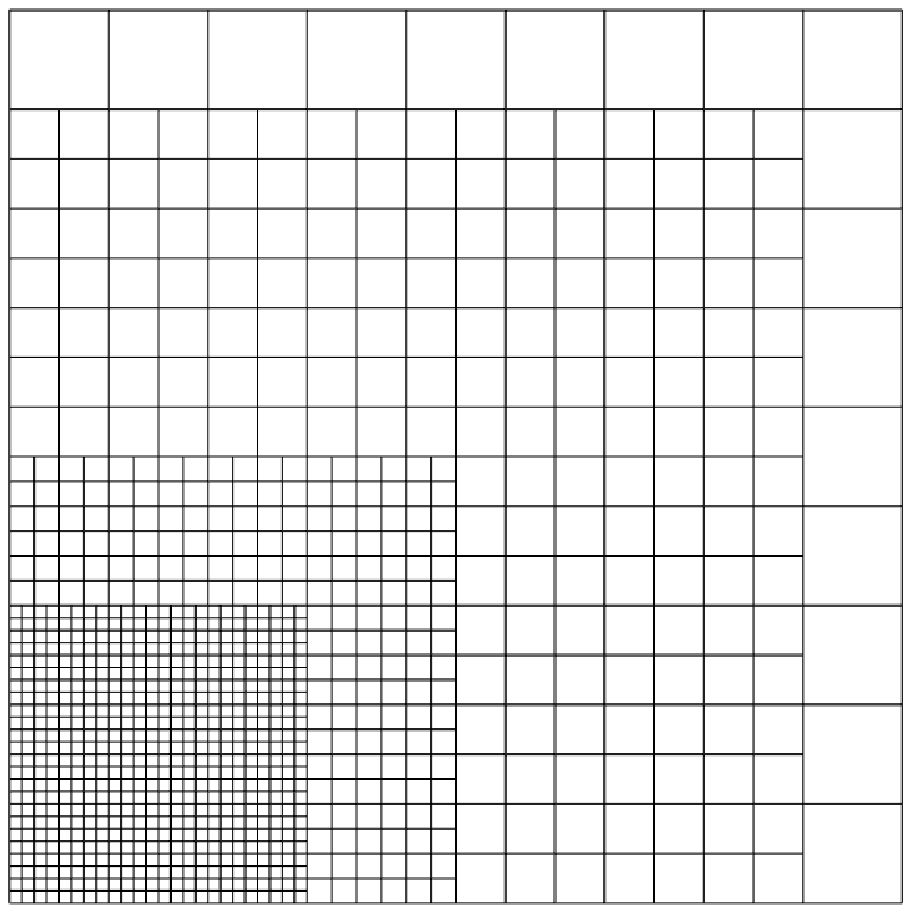}}
\caption{Some strictly ${\cal H}$-admissible meshes used in Test 4, for different degrees $p$ and admissibility class $m$, with four levels.} \label{fig:mesh-test-HB-THB}
\end{figure}

In Figure~\ref{fig:test3a} we present the results on strictly ${\cal H}$-admissible meshes for the admissibility class $m=3$. The condition number for THB-splines is always smaller than for HB-splines, and the difference increases with the degree. Although not reported in the paper, the same behavior is obtained for other values of $m$. Moreover, a better behavior of THB-splines was also observed in the previous work \cite{Hofreither2016b} with a decomposition which is equivalent to ${\cal V}^\ell_{\text{all}}$.

Finally, in Figure~\ref{fig:test3b} we show the condition number of the preconditioned system obtained for HB-splines on strictly ${\cal H}$-admissible and strictly ${\cal T}$-admissible meshes, for different values of $m$ and degree $p=3$, and with the subspaces ${\cal V}^\ell_{{\cal H}\text{-supp}}$. Analogously to what we have seen before, it is clear that HB-splines require that the mesh is strictly ${\cal H}$-admissible in order to obtain a bounded condition number. However, the preconditioner reduces the condition number with respect to the unpreconditioned system when the mesh is not strictly ${\cal H}$-admissible.

\begin{figure}[htb]
\input{figures-test4-HB-THB.tex}
\end{figure}

\paragraph{Test 5: THB-splines on adaptive meshes}
For the last numerical test we consider a real adaptive problem, where the refinement is not decided a priori but following an adaptive algorithm \cite{BC2016}. We consider the curved L-shaped domain of Figure~\ref{fig:curvedL_mesh}, which is defined as a single patch with a line of $C^0$ continuity. The initial mesh has $32 \times 16$ elements. We solve Poisson problem with Dirichlet boundary conditions, and solution given in polar coordinates by
\[
u(\rho,\phi) = \rho^{2/3} \sin(2\phi/3).
\]
For the adaptive refinement we use an a posteriori residual estimator, and marking is done using D\"orfler's strategy with parameter $\theta = 0.85$. The iterative refinement is performed until we reach a fixed number of levels, that we take equal to eleven. Some sample meshes are shown in Figure~\ref{fig:curvedL_mesh}.
\begin{figure}[ht]
\centering
\subfigure[Initial mesh.]{
\includegraphics[width=0.31\textwidth,trim=5mm 0mm 10mm 5mm, clip]{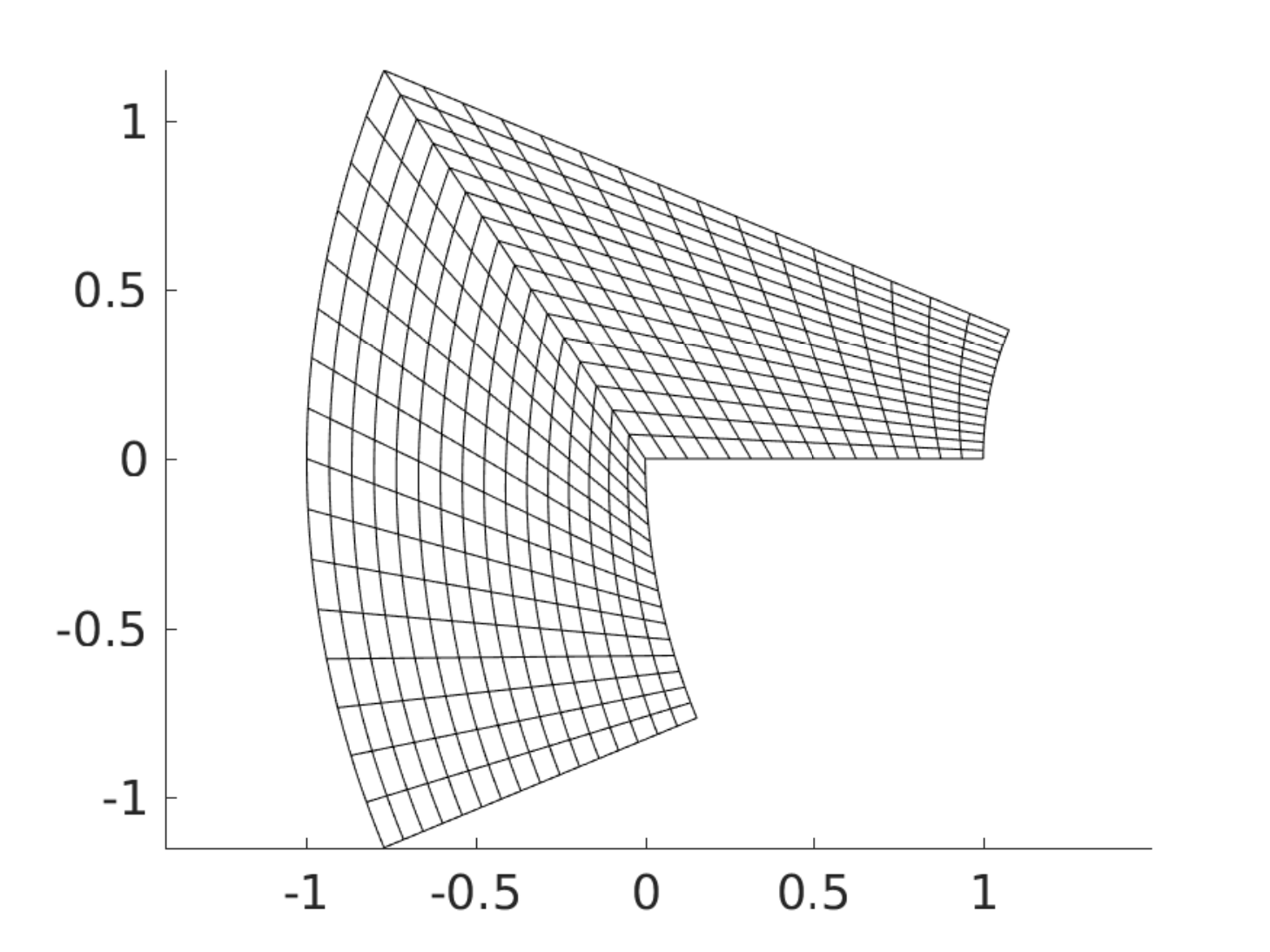}
}
\subfigure[Mesh after 7 steps, $p=2$, $m=2$.]{
\includegraphics[width=0.31\textwidth,trim=5mm 0mm 10mm 5mm, clip]{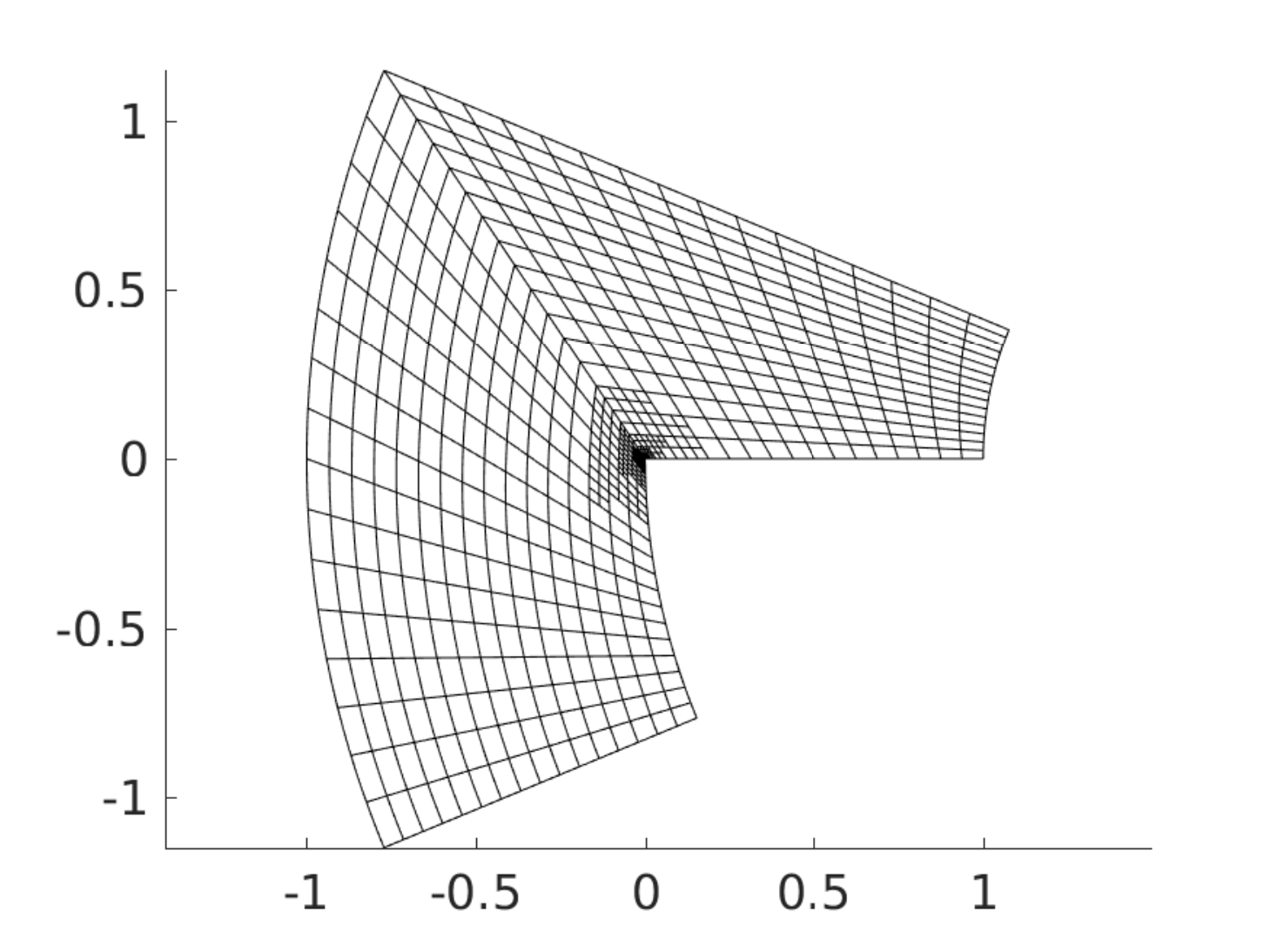}
}
\subfigure[Mesh after 7 steps, $p=4$, $m=2$.]{
\includegraphics[width=0.31\textwidth,trim=5mm 0mm 10mm 5mm, clip]{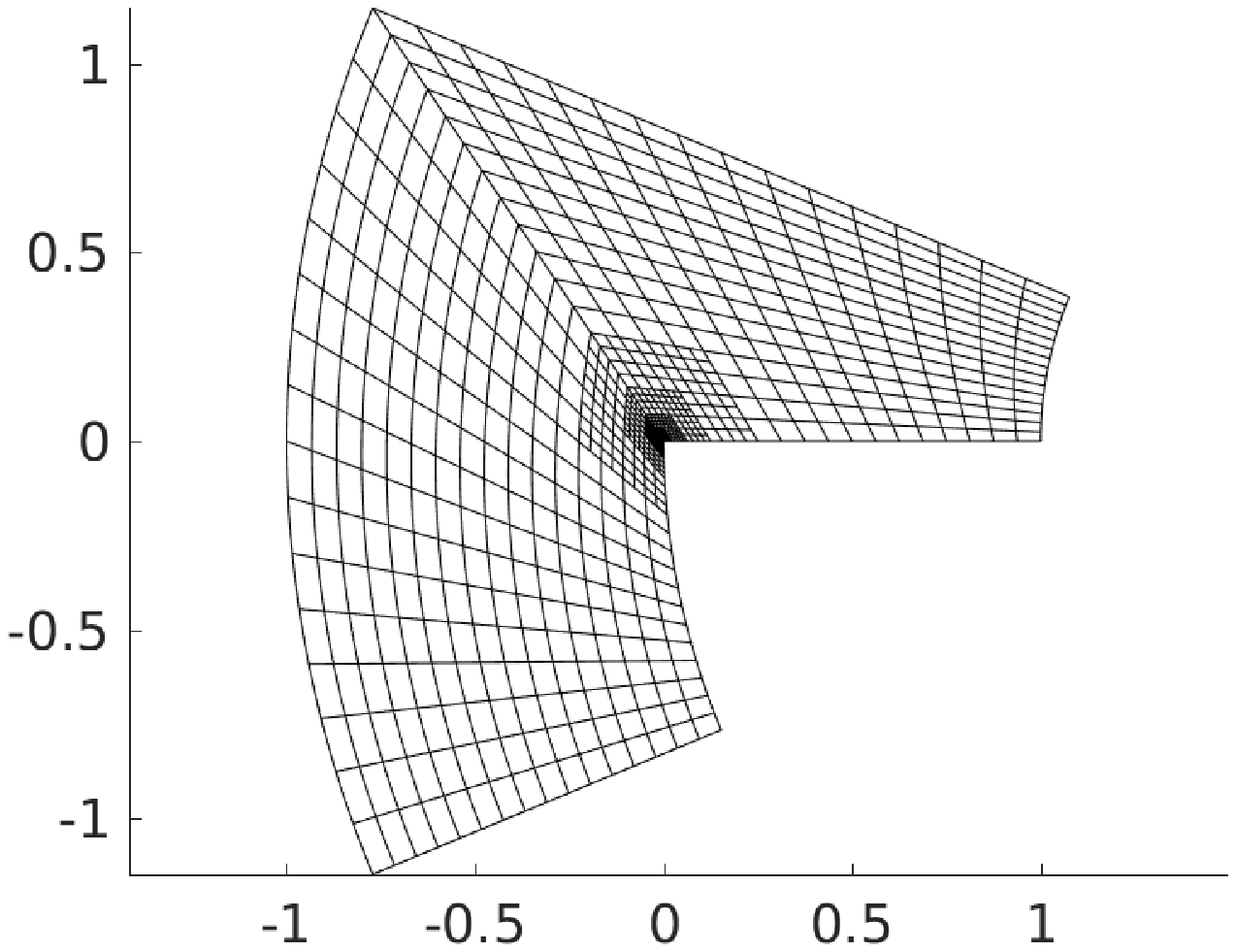}
}
\caption{Computational domain and adaptive meshes for Test 5: curved L-shaped domain.} \label{fig:curvedL_mesh}
\end{figure}

We have run numerical tests discretizing with THB-splines with different values of the degree $p$, from 2 to 4, to compare the results obtained with strictly ${\cal T}$-admissible meshes of class $m=2$ with the ones obtained for non-admissible meshes, the results are depicted in Figure~\ref{fig:test-Lshaped}. The plots confirm the good behavior of the preconditioner also in a situation of adaptive refinement, and the need of admissible meshes to guarantee the boundedness of the condition number. The main difference with respect to previous tests is a more localized refinement, that leads to a smaller number of degrees of freedom, and therefore to smaller condition numbers in the non-preconditioned case.

\begin{figure}[htb]
\input{figures-results-curvedL.tex}
\end{figure}

%% file: figures-test1-alldecomp.tex
\subfigure[Results for $p=1$.]{
\tikzsetnextfilename{test1_d2_BPX_4decomps_p1}
  \begin{tikzpicture}
  \begin{semilogyaxis}[width=.495\textwidth,height=6.2cm,xlabel={Number of levels},ylabel={Condition number}, y label style={at={(axis description cs:0.03,0.6)}, anchor=north},xmin=2,xmax=10.5,ymin=1e0,ymax=2e5, grid = major, legend columns = 1, anchor = north, legend pos = outer north east, legend style = {at={(0,0.98)}}]
  \addplot[color=black, mark=x, solid, thick, mark options={solid}, mark size=3pt] table [x={level}, y={NoPrec}, col sep=comma] {data-alldecomps/THB_Decomposition_All_dofs_DIM2_DEG1.csv};
  \addplot[color=red, mark=square, solid, thick, mark size=3pt] table [x={level}, y={Gauss-Seidel}, col sep=comma] {data-alldecomps/THB_Decomposition_All_dofs_DIM2_DEG1.csv};
  \addplot[color=blue, mark=o, solid, thick, mark options={solid}, mark size=3pt] table [x={level}, y={Gauss-Seidel}, col sep=comma] {data-alldecomps/THB_Decomposition_Support_dofs_DIM2_DEG1.csv};
  \addplot[color=green, mark=triangle, solid, thick, mark size=3pt] table [x={level}, y={Gauss-Seidel}, col sep=comma] {data-alldecomps/THB_Decomposition_Mod_dofs_DIM2_DEG1.csv};
  \addplot[color=cyan, mark=diamond, solid, thick, mark options={solid}, mark size=3pt] table [x={level}, y={Gauss-Seidel}, col sep=comma] {data-alldecomps/THB_Decomposition_New_dofs_DIM2_DEG1.csv};
\footnotesize \legend{\tiny{No-prec.}, \tiny{BPX-all}, \tiny{BPX-supp.}, \tiny{BPX-mod.}, \tiny{BPX-new}}
  \end{semilogyaxis}
\end{tikzpicture}
}
\subfigure[Results for $p=2$.]{
\tikzsetnextfilename{test1_d2_BPX_4decomps_p2}
  \begin{tikzpicture}
  \begin{semilogyaxis}[width=.495\textwidth,height=6.2cm,xlabel={Number of levels},ylabel={Condition number}, y label style={at={(axis description cs:0.03,0.6)}, anchor=north},xmin=2,xmax=10.5,ymin=1e0,ymax=2e5, grid = major, legend columns = 1, anchor = north, legend pos = outer north east, legend style = {at={(0,0.98)}}]
  \addplot[color=black, mark=x, solid, thick, mark options={solid}, mark size=3pt] table [x={level}, y={NoPrec}, col sep=comma] {data-alldecomps/THB_Decomposition_All_dofs_DIM2_DEG2.csv};
  \addplot[color=red, mark=square, solid, thick, mark size=3pt] table [x={level}, y={Gauss-Seidel}, col sep=comma] {data-alldecomps/THB_Decomposition_All_dofs_DIM2_DEG2.csv};
  \addplot[color=blue, mark=o, solid, thick, mark options={solid}, mark size=3pt] table [x={level}, y={Gauss-Seidel}, col sep=comma] {data-alldecomps/THB_Decomposition_Support_dofs_DIM2_DEG2.csv};
  \addplot[color=green, mark=triangle, solid, thick, mark size=3pt] table [x={level}, y={Gauss-Seidel}, col sep=comma] {data-alldecomps/THB_Decomposition_Mod_dofs_DIM2_DEG2.csv};
  \addplot[color=cyan, mark=diamond, solid, thick, mark options={solid}, mark size=3pt] table [x={level}, y={Gauss-Seidel}, col sep=comma] {data-alldecomps/THB_Decomposition_New_dofs_DIM2_DEG2.csv};
\footnotesize \legend{\tiny{No-prec.}, \tiny{BPX-all}, \tiny{BPX-supp.}, \tiny{BPX-mod.}, \tiny{BPX-new}}
  \end{semilogyaxis}
\end{tikzpicture}
}
\subfigure[Results for $p=3$.]{
\tikzsetnextfilename{test1_d2_BPX_4decomps_p3}
  \begin{tikzpicture}
  \begin{semilogyaxis}[width=.495\textwidth,height=6.2cm,xlabel={Number of levels},ylabel={Condition number}, y label style={at={(axis description cs:0.03,0.6)}, anchor=north},xmin=2,xmax=10.5,ymin=1e0,ymax=2e5, grid = major, legend columns = 1, anchor = north, legend pos = outer north east, legend style = {at={(0,0.98)}}]
  \addplot[color=black, mark=x, solid, thick, mark options={solid}, mark size=3pt] table [x={level}, y={NoPrec}, col sep=comma] {data-alldecomps/THB_Decomposition_All_dofs_DIM2_DEG3.csv};
  \addplot[color=red, mark=square, solid, thick, mark size=3pt] table [x={level}, y={Gauss-Seidel}, col sep=comma] {data-alldecomps/THB_Decomposition_All_dofs_DIM2_DEG3.csv};
  \addplot[color=blue, mark=o, solid, thick, mark options={solid}, mark size=3pt] table [x={level}, y={Gauss-Seidel}, col sep=comma] {data-alldecomps/THB_Decomposition_Support_dofs_DIM2_DEG3.csv};
  \addplot[color=green, mark=triangle, solid, thick, mark size=3pt] table [x={level}, y={Gauss-Seidel}, col sep=comma] {data-alldecomps/THB_Decomposition_Mod_dofs_DIM2_DEG3.csv};
  \addplot[color=cyan, mark=diamond, solid, thick, mark options={solid}, mark size=3pt] table [x={level}, y={Gauss-Seidel}, col sep=comma] {data-alldecomps/THB_Decomposition_New_dofs_DIM2_DEG3.csv};
\footnotesize \legend{\tiny{No-prec.}, \tiny{BPX-all}, \tiny{BPX-supp.}, \tiny{BPX-mod.}, \tiny{BPX-new}}
  \end{semilogyaxis}
\end{tikzpicture}
}
\subfigure[Results for $p=4$.]{
\tikzsetnextfilename{test1_d2_BPX_4decomps_p4}
  \begin{tikzpicture}
  \begin{semilogyaxis}[width=.495\textwidth,height=6.2cm,xlabel={Number of levels},ylabel={Condition number}, y label style={at={(axis description cs:0.03,0.6)}, anchor=north},xmin=2,xmax=10.5,ymin=1e0,ymax=2e5, grid = major, legend columns = 1, anchor = north, legend pos = outer north east, legend style = {at={(0,0.98)}}]
  \addplot[color=black, mark=x, solid, thick, mark options={solid}, mark size=3pt] table [x={level}, y={NoPrec}, col sep=comma] {data-alldecomps/THB_Decomposition_All_dofs_DIM2_DEG4.csv};
  \addplot[color=red, mark=square, solid, thick, mark size=3pt] table [x={level}, y={Gauss-Seidel}, col sep=comma] {data-alldecomps/THB_Decomposition_All_dofs_DIM2_DEG4.csv};
  \addplot[color=blue, mark=o, solid, thick, mark options={solid}, mark size=3pt] table [x={level}, y={Gauss-Seidel}, col sep=comma] {data-alldecomps/THB_Decomposition_Support_dofs_DIM2_DEG4.csv};
  \addplot[color=green, mark=triangle, solid, thick, mark size=3pt] table [x={level}, y={Gauss-Seidel}, col sep=comma] {data-alldecomps/THB_Decomposition_Mod_dofs_DIM2_DEG4.csv};
  \addplot[color=cyan, mark=diamond, solid, thick, mark options={solid}, mark size=3pt] table [x={level}, y={Gauss-Seidel}, col sep=comma] {data-alldecomps/THB_Decomposition_New_dofs_DIM2_DEG4.csv};
\footnotesize \legend{\tiny{No-prec.}, \tiny{BPX-all}, \tiny{BPX-supp.}, \tiny{BPX-mod.}, \tiny{BPX-new}}
  \end{semilogyaxis}
\end{tikzpicture}
}

\caption{Condition numbers for Test 1: THB-splines on strictly ${\cal T}$-admissible meshes for $d=2$.} \label{fig:test1-alldecomps}

%% file: figures-test1-decomp.tex
\subfigure[$d=1$, results with and without preconditioning.]{
\tikzsetnextfilename{test1_d1_NP_vs_BPX}
  \begin{tikzpicture}
  \begin{semilogyaxis}[width=.495\textwidth,height=6.2cm,xlabel={Number of levels},ylabel={Condition number}, y label style={at={(axis description cs:0.03,0.6)}, anchor=north}, xmin=2,xmax=11.5,ymin=1e0,ymax=1e9, grid = major, ytick distance = 1e2, legend columns = 2, anchor = north, legend pos = outer north east, legend style = {at={(0,0.98)}}]
  \addplot[color=blue, mark=o, dashed, thick, mark options={solid}, mark size=3pt] table [x={level}, y={NoPrec}, col sep=comma] {data-parametric/THB_Decomposition_Support_dofs_DIM1_DEG1.csv};
  \addplot[color=blue, mark=o, solid, thick, mark size=3pt] table [x={level}, y={Gauss-Seidel}, col sep=comma] {data-parametric/THB_Decomposition_Support_dofs_DIM1_DEG1.csv};
  \addplot[color=red, mark=square, dashed, thick, mark options={solid}, mark size=3pt] table [x={level}, y={NoPrec}, col sep=comma] {data-parametric/THB_Decomposition_Support_dofs_DIM1_DEG2.csv};
  \addplot[color=red, mark=square, solid, thick, mark size=3pt] table [x={level}, y={Gauss-Seidel}, col sep=comma] {data-parametric/THB_Decomposition_Support_dofs_DIM1_DEG2.csv};
  \addplot[color=green, mark=triangle, dashed, thick, mark options={solid}, mark size=3pt] table [x={level}, y={NoPrec}, col sep=comma] {data-parametric/THB_Decomposition_Support_dofs_DIM1_DEG3.csv};
  \addplot[color=green, mark=triangle, solid, thick, mark size=3pt] table [x={level}, y={Gauss-Seidel}, col sep=comma] {data-parametric/THB_Decomposition_Support_dofs_DIM1_DEG3.csv};
  \addplot[color=cyan, mark=x, dashed, thick, mark options={solid}, mark size=3pt] table [x={level}, y={NoPrec}, col sep=comma] {data-parametric/THB_Decomposition_Support_dofs_DIM1_DEG4.csv};
  \addplot[color=cyan, mark=x, solid, thick, mark size=3pt] table [x={level}, y={Gauss-Seidel}, col sep=comma] {data-parametric/THB_Decomposition_Support_dofs_DIM1_DEG4.csv};
\footnotesize \legend{\tiny{$p=1$, No-prec.}, \tiny{$p=1$, BPX-supp.}, \tiny{$p=2$, No-prec.}, \tiny{$p=2$, BPX-supp.}, \tiny{$p=3$, No-prec.}, \tiny{$p=3$, BPX-supp.}, \tiny{$p=4$, No-prec.}, \tiny{$p=4$, BPX-supp.}}
  \end{semilogyaxis}
\end{tikzpicture}
}
\subfigure[$d=1$, results for subspaces ${\cal V}^\ell_{{\cal T}\text{-supp}}$ and ${\cal V}^\ell_{\text{all}}$.]{
\tikzsetnextfilename{test1_d1_All_vs_Support}
  \begin{tikzpicture}
  \begin{semilogyaxis}[width=.495\textwidth,height=6.2cm,xlabel={Number of levels},ylabel={Condition number}, y label style={at={(axis description cs:0.03,0.6)}, anchor=north}, xmin=2,xmax=11.5,ymin=1e0,ymax=2e2, grid = major, legend columns = 2, anchor = north, legend pos = outer north east, legend style = {at={(0,0.98)}}]
  \addplot[color=blue, mark=o, dashed, thick, mark options={solid}, mark size=3pt] table [x={level}, y={Gauss-Seidel}, col sep=comma] {data-parametric/THB_Decomposition_All_dofs_DIM1_DEG1.csv};
  \addplot[color=blue, mark=o, solid, thick, mark size=3pt] table [x={level}, y={Gauss-Seidel}, col sep=comma] {data-parametric/THB_Decomposition_Support_dofs_DIM1_DEG1.csv};
  \addplot[color=red, mark=square, dashed, thick, mark options={solid}, mark size=3pt] table [x={level}, y={Gauss-Seidel}, col sep=comma] {data-parametric/THB_Decomposition_All_dofs_DIM1_DEG2.csv};
  \addplot[color=red, mark=square, solid, thick, mark size=3pt] table [x={level}, y={Gauss-Seidel}, col sep=comma] {data-parametric/THB_Decomposition_Support_dofs_DIM1_DEG2.csv};
  \addplot[color=green, mark=triangle, dashed, thick, mark options={solid}, mark size=3pt] table [x={level}, y={Gauss-Seidel}, col sep=comma] {data-parametric/THB_Decomposition_All_dofs_DIM1_DEG3.csv};
  \addplot[color=green, mark=triangle, solid, thick, mark size=3pt] table [x={level}, y={Gauss-Seidel}, col sep=comma] {data-parametric/THB_Decomposition_Support_dofs_DIM1_DEG3.csv};
  \addplot[color=cyan, mark=x, dashed, thick, mark options={solid}, mark size=3pt] table [x={level}, y={Gauss-Seidel}, col sep=comma] {data-parametric/THB_Decomposition_All_dofs_DIM1_DEG4.csv};
  \addplot[color=cyan, mark=x, solid, thick, mark size=3pt] table [x={level}, y={Gauss-Seidel}, col sep=comma] {data-parametric/THB_Decomposition_Support_dofs_DIM1_DEG4.csv};
\footnotesize \legend{\tiny{$p=1$, BPX-all}, \tiny{$p=1$, BPX-supp.}, \tiny{$p=2$, BPX-all}, \tiny{$p=2$, BPX-supp.}, \tiny{$p=3$, BPX-all}, \tiny{$p=3$, BPX-supp.}, \tiny{$p=4$, BPX-all}, \tiny{$p=4$, BPX-supp.}}
  \end{semilogyaxis}
\end{tikzpicture}
}

\subfigure[$d=2$, results with and without preconditioning.]{
\tikzsetnextfilename{test1_d2_NP_vs_BPX}
  \begin{tikzpicture}
  \begin{semilogyaxis}[width=.495\textwidth,height=6.2cm,xlabel={Number of levels},ylabel={Condition number}, y label style={at={(axis description cs:0.03,0.6)}, anchor=north},xmin=2,xmax=10.5,ymin=1e0,ymax=2e7, grid = major, legend columns = 2, anchor = north, legend pos = outer north east, legend style = {at={(0,0.98)}}]
  \addplot[color=blue, mark=o, dashed, thick, mark options={solid}, mark size=3pt] table [x={level}, y={NoPrec}, col sep=comma] {data-parametric/THB_Decomposition_All_dofs_DIM2_DEG1.csv};
  \addplot[color=blue, mark=o, solid, thick, mark size=3pt] table [x={level}, y={Gauss-Seidel}, col sep=comma] {data-parametric/THB_Decomposition_Support_dofs_DIM2_DEG1.csv};
  \addplot[color=red, mark=square, dashed, thick, mark options={solid}, mark size=3pt] table [x={level}, y={NoPrec}, col sep=comma] {data-parametric/THB_Decomposition_All_dofs_DIM2_DEG2.csv};
  \addplot[color=red, mark=square, solid, thick, mark size=3pt] table [x={level}, y={Gauss-Seidel}, col sep=comma] {data-parametric/THB_Decomposition_Support_dofs_DIM2_DEG2.csv};
  \addplot[color=green, mark=triangle, dashed, thick, mark options={solid}, mark size=3pt] table [x={level}, y={NoPrec}, col sep=comma] {data-parametric/THB_Decomposition_All_dofs_DIM2_DEG3.csv};
  \addplot[color=green, mark=triangle, solid, thick, mark size=3pt] table [x={level}, y={Gauss-Seidel}, col sep=comma] {data-parametric/THB_Decomposition_Support_dofs_DIM2_DEG3.csv};
  \addplot[color=cyan, mark=x, dashed, thick, mark options={solid}, mark size=3pt] table [x={level}, y={NoPrec}, col sep=comma] {data-parametric/THB_Decomposition_All_dofs_DIM2_DEG4.csv};
  \addplot[color=cyan, mark=x, solid, thick, mark size=3pt] table [x={level}, y={Gauss-Seidel}, col sep=comma] {data-parametric/THB_Decomposition_Support_dofs_DIM2_DEG4.csv};
\footnotesize \legend{\tiny{$p=1$, No-prec.}, \tiny{$p=1$, BPX-supp.}, \tiny{$p=2$, No-prec.}, \tiny{$p=2$, BPX-supp.}, \tiny{$p=3$, No-prec.}, \tiny{$p=3$, BPX-supp.}, \tiny{$p=4$, No-prec.}, \tiny{$p=4$, BPX-supp.}}
  \end{semilogyaxis}
\end{tikzpicture}
}
\subfigure[$d=2$, results for subspaces ${\cal V}^\ell_{{\cal T}\text{-supp}}$ and ${\cal V}^\ell_{\text{all}}$.]{
\tikzsetnextfilename{test1_d2_All_vs_Support}
  \begin{tikzpicture}
  \begin{semilogyaxis}[width=.495\textwidth,height=6.2cm,xlabel={Number of levels},ylabel={Condition number}, y label style={at={(axis description cs:0.03,0.6)}, anchor=north},xmin=2,xmax=10.5,ymin=1e0,ymax=5e3, grid = major, legend columns = 2, anchor = north, legend pos = outer north east, legend style = {at={(0,0.98)}}]
  \addplot[color=blue, mark=o, dashed, thick, mark options={solid}, mark size=3pt] table [x={level}, y={Gauss-Seidel}, col sep=comma] {data-parametric/THB_Decomposition_All_dofs_DIM2_DEG1.csv};
  \addplot[color=blue, mark=o, solid, thick, mark size=3pt] table [x={level}, y={Gauss-Seidel}, col sep=comma] {data-parametric/THB_Decomposition_Support_dofs_DIM2_DEG1.csv};
  \addplot[color=red, mark=square, dashed, thick, mark options={solid}, mark size=3pt] table [x={level}, y={Gauss-Seidel}, col sep=comma] {data-parametric/THB_Decomposition_All_dofs_DIM2_DEG2.csv};
  \addplot[color=red, mark=square, solid, thick, mark size=3pt] table [x={level}, y={Gauss-Seidel}, col sep=comma] {data-parametric/THB_Decomposition_Support_dofs_DIM2_DEG2.csv};
  \addplot[color=green, mark=triangle, dashed, thick, mark options={solid}, mark size=3pt] table [x={level}, y={Gauss-Seidel}, col sep=comma] {data-parametric/THB_Decomposition_All_dofs_DIM2_DEG3.csv};
  \addplot[color=green, mark=triangle, solid, thick, mark size=3pt] table [x={level}, y={Gauss-Seidel}, col sep=comma] {data-parametric/THB_Decomposition_Support_dofs_DIM2_DEG3.csv};
  \addplot[color=cyan, mark=x, dashed, thick, mark options={solid}, mark size=3pt] table [x={level}, y={Gauss-Seidel}, col sep=comma] {data-parametric/THB_Decomposition_All_dofs_DIM2_DEG4.csv};
  \addplot[color=cyan, mark=x, solid, thick, mark size=3pt] table [x={level}, y={Gauss-Seidel}, col sep=comma] {data-parametric/THB_Decomposition_Support_dofs_DIM2_DEG4.csv};
\footnotesize \legend{\tiny{$p=1$, BPX-all}, \tiny{$p=1$, BPX-supp.}, \tiny{$p=2$, BPX-all}, \tiny{$p=2$, BPX-supp.}, \tiny{$p=3$, BPX-all}, \tiny{$p=3$, BPX-supp.}, \tiny{$p=4$, BPX-all}, \tiny{$p=4$, BPX-supp.}}
  \end{semilogyaxis}
\end{tikzpicture}
}

\subfigure[$d=3$, results with and without preconditioning.]{
\tikzsetnextfilename{test1_d3_NP_vs_BPX}
  \begin{tikzpicture}
  \begin{semilogyaxis}[width=.495\textwidth,height=6.2cm,xlabel={Number of levels},ylabel={Condition number}, y label style={at={(axis description cs:0.03,0.6)}, anchor=north},xmin=2,xmax=8.5,ymin=1e0,ymax=5e9, grid = major, ytick distance = 1e2, legend columns = 2, anchor = north, legend pos = outer north east, legend style = {at={(0,0.98)}}]
  \addplot[color=blue, mark=o, dashed, thick, mark options={solid}, mark size=3pt] table [x={level}, y={NoPrec}, col sep=comma] {data-parametric/THB_Decomposition_Support_dofs_DIM3_DEG1.csv};
  \addplot[color=blue, mark=o, solid, thick, mark size=3pt] table [x={level}, y={Gauss-Seidel}, col sep=comma] {data-parametric/THB_Decomposition_Support_dofs_DIM3_DEG1.csv};
  \addplot[color=red, mark=square, dashed, thick, mark options={solid}, mark size=3pt] table [x={level}, y={NoPrec}, col sep=comma] {data-parametric/THB_Decomposition_Support_dofs_DIM3_DEG2.csv};
  \addplot[color=red, mark=square, solid, thick, mark size=3pt] table [x={level}, y={Gauss-Seidel}, col sep=comma] {data-parametric/THB_Decomposition_Support_dofs_DIM3_DEG2.csv};
  \addplot[color=green, mark=triangle, dashed, thick, mark options={solid}, mark size=3pt] table [x={level}, y={NoPrec}, col sep=comma] {data-parametric/THB_Decomposition_Support_dofs_DIM3_DEG3.csv};
  \addplot[color=green, mark=triangle, solid, thick, mark size=3pt] table [x={level}, y={Gauss-Seidel}, col sep=comma] {data-parametric/THB_Decomposition_Support_dofs_DIM3_DEG3.csv};
  \addplot[color=cyan, mark=x, dashed, thick, mark options={solid}, mark size=3pt] table [x={level}, y={NoPrec}, col sep=comma] {data-parametric/THB_Decomposition_Support_dofs_DIM3_DEG4.csv};
  \addplot[color=cyan, mark=x, solid, thick, mark size=3pt] table [x={level}, y={Gauss-Seidel}, col sep=comma] {data-parametric/THB_Decomposition_Support_dofs_DIM3_DEG4.csv};
\footnotesize \legend{\tiny{$p=1$, No-prec.}, \tiny{$p=1$, BPX-supp.}, \tiny{$p=2$, No-prec.}, \tiny{$p=2$, BPX-supp.}, \tiny{$p=3$, No-prec.}, \tiny{$p=3$, BPX-supp.}, \tiny{$p=4$, No-prec.}, \tiny{$p=4$, BPX-supp.}}
  \end{semilogyaxis}
\end{tikzpicture}
}
\subfigure[$d=3$, results for subspaces ${\cal V}^\ell_{{\cal T}\text{-supp}}$ and ${\cal V}^\ell_{\text{all}}$.]{
\tikzsetnextfilename{test1_d3_All_vs_Support}
  \begin{tikzpicture}
  \begin{semilogyaxis}[width=.495\textwidth,height=6.2cm,xlabel={Number of levels},ylabel={Condition number}, y label style={at={(axis description cs:0.03,0.6)}, anchor=north},xmin=2,xmax=8.5,ymin=1e0,ymax=5e5, grid = major, legend columns = 2, anchor = north, legend pos = outer north east, legend style = {at={(0,0.98)}}]
  \addplot[color=blue, mark=o, dashed, thick, mark options={solid}, mark size=3pt] table [x={level}, y={Gauss-Seidel}, col sep=comma] {data-parametric/THB_Decomposition_All_dofs_DIM3_DEG1.csv};
  \addplot[color=blue, mark=o, solid, thick, mark size=3pt] table [x={level}, y={Gauss-Seidel}, col sep=comma] {data-parametric/THB_Decomposition_Support_dofs_DIM3_DEG1.csv};
  \addplot[color=red, mark=square, dashed, thick, mark options={solid}, mark size=3pt] table [x={level}, y={Gauss-Seidel}, col sep=comma] {data-parametric/THB_Decomposition_All_dofs_DIM3_DEG2.csv};
  \addplot[color=red, mark=square, solid, thick, mark size=3pt] table [x={level}, y={Gauss-Seidel}, col sep=comma] {data-parametric/THB_Decomposition_Support_dofs_DIM3_DEG2.csv};
  \addplot[color=green, mark=triangle, dashed, thick, mark options={solid}, mark size=3pt] table [x={level}, y={Gauss-Seidel}, col sep=comma] {data-parametric/THB_Decomposition_All_dofs_DIM3_DEG3.csv};
  \addplot[color=green, mark=triangle, solid, thick, mark size=3pt] table [x={level}, y={Gauss-Seidel}, col sep=comma] {data-parametric/THB_Decomposition_Support_dofs_DIM3_DEG3.csv};
  \addplot[color=cyan, mark=x, dashed, thick, mark options={solid}, mark size=3pt] table [x={level}, y={Gauss-Seidel}, col sep=comma] {data-parametric/THB_Decomposition_All_dofs_DIM3_DEG4.csv};
  \addplot[color=cyan, mark=x, solid, thick, mark size=3pt] table [x={level}, y={Gauss-Seidel}, col sep=comma] {data-parametric/THB_Decomposition_Support_dofs_DIM3_DEG4.csv};
\footnotesize \legend{\tiny{$p=1$, BPX-all}, \tiny{$p=1$, BPX-supp.}, \tiny{$p=2$, BPX-all}, \tiny{$p=2$, BPX-supp.}, \tiny{$p=3$, BPX-all}, \tiny{$p=3$, BPX-supp.}, \tiny{$p=4$, BPX-all}, \tiny{$p=4$, BPX-supp.}}
  \end{semilogyaxis}
\end{tikzpicture}
}
\caption{Condition numbers for Test 1: THB-splines on strictly ${\cal T}$-admissible meshes.} \label{fig:test-square}

%% file: figures-test2-admissibility.tex
\subfigure[Results with and without preconditioning.]{\label{fig:test-non-admissiblea}
\tikzsetnextfilename{test2_nonadmissible_NP_vs_BPX}
  \begin{tikzpicture}
  \begin{semilogyaxis}[width=.495\textwidth,height=6.2cm,xlabel={Number of levels},ylabel={Condition number}, y label style={at={(axis description cs:0.03,0.6)}, anchor=north},xmin=2,xmax=8.5,ymin=1e0,ymax=1e9, grid = major, legend columns = 2, anchor = north, legend pos = outer north east, legend style = {at={(0,0.98)}}]
  \addplot[color=blue, mark=o, dashed, thick, mark options={solid}, mark size=3pt] table [x={level}, y={NoPrec}, col sep=comma] {data-non-admissible/THB_Decomposition_Support_dofs_DEG1.csv};
  \addplot[color=blue, mark=o, solid, thick, mark size=3pt] table [x={level}, y={Gauss-Seidel}, col sep=comma] {data-non-admissible/THB_Decomposition_Support_dofs_DEG1.csv};
  \addplot[color=red, mark=square, dashed, thick, mark options={solid}, mark size=3pt] table [x={level}, y={NoPrec}, col sep=comma] {data-non-admissible/THB_Decomposition_Support_dofs_DEG2.csv};
  \addplot[color=red, mark=square, solid, thick, mark size=3pt] table [x={level}, y={Gauss-Seidel}, col sep=comma] {data-non-admissible/THB_Decomposition_Support_dofs_DEG2.csv};
  \addplot[color=green, mark=triangle, dashed, thick, mark options={solid}, mark size=3pt] table [x={level}, y={NoPrec}, col sep=comma] {data-non-admissible/THB_Decomposition_Support_dofs_DEG3.csv};
  \addplot[color=green, mark=triangle, solid, thick, mark size=3pt] table [x={level}, y={Gauss-Seidel}, col sep=comma] {data-non-admissible/THB_Decomposition_Support_dofs_DEG3.csv};
  \addplot[color=cyan, mark=x, dashed, thick, mark options={solid}, mark size=3pt] table [x={level}, y={NoPrec}, col sep=comma] {data-non-admissible/THB_Decomposition_Support_dofs_DEG4.csv};
  \addplot[color=cyan, mark=x, solid, thick, mark size=3pt] table [x={level}, y={Gauss-Seidel}, col sep=comma] {data-non-admissible/THB_Decomposition_Support_dofs_DEG4.csv};
\footnotesize \legend{\tiny{$p=1$, No-prec.}, \tiny{$p=1$, BPX-supp.}, \tiny{$p=2$, No-prec.}, \tiny{$p=2$, BPX-supp.}, \tiny{$p=3$, No-prec.}, \tiny{$p=3$, BPX-supp.}, \tiny{$p=4$, No-prec.}, \tiny{$p=4$, BPX-supp.}}
  \end{semilogyaxis}
\end{tikzpicture}
}
\subfigure[Results for subspaces ${\cal V}^\ell_{{\cal T}\text{-supp}}$ and ${\cal V}^\ell_{\text{all}}$.]{\label{fig:test-non-admissibleb}
\tikzsetnextfilename{test2_nonadmissible_All_vs_Support}
  \begin{tikzpicture}
  \begin{semilogyaxis}[width=.495\textwidth,height=6.2cm,xlabel={Number of levels},ylabel={Condition number}, y label style={at={(axis description cs:0.03,0.6)}, anchor=north},xmin=2,xmax=8.5,ymin=1e0,ymax=2e3, grid = major, legend columns = 2, anchor = north, legend pos = outer north east, legend style = {at={(0,0.98)}}]
  \addplot[color=blue, mark=o, dashed, thick, mark options={solid}, mark size=3pt] table [x={level}, y={Gauss-Seidel}, col sep=comma] {data-non-admissible/THB_Decomposition_All_dofs_DEG1.csv};
  \addplot[color=blue, mark=o, solid, thick, mark size=3pt] table [x={level}, y={Gauss-Seidel}, col sep=comma] {data-non-admissible/THB_Decomposition_Support_dofs_DEG1.csv};
  \addplot[color=red, mark=square, dashed, thick, mark options={solid}, mark size=3pt] table [x={level}, y={Gauss-Seidel}, col sep=comma] {data-non-admissible/THB_Decomposition_All_dofs_DEG2.csv};
  \addplot[color=red, mark=square, solid, thick, mark size=3pt] table [x={level}, y={Gauss-Seidel}, col sep=comma] {data-non-admissible/THB_Decomposition_Support_dofs_DEG2.csv};
  \addplot[color=green, mark=triangle, dashed, thick, mark options={solid}, mark size=3pt] table [x={level}, y={Gauss-Seidel}, col sep=comma] {data-non-admissible/THB_Decomposition_All_dofs_DEG3.csv};
  \addplot[color=green, mark=triangle, solid, thick, mark size=3pt] table [x={level}, y={Gauss-Seidel}, col sep=comma] {data-non-admissible/THB_Decomposition_Support_dofs_DEG3.csv};
  \addplot[color=cyan, mark=x, dashed, thick, mark options={solid}, mark size=3pt] table [x={level}, y={Gauss-Seidel}, col sep=comma] {data-non-admissible/THB_Decomposition_All_dofs_DEG4.csv};
  \addplot[color=cyan, mark=x, solid, thick, mark size=3pt] table [x={level}, y={Gauss-Seidel}, col sep=comma] {data-non-admissible/THB_Decomposition_Support_dofs_DEG4.csv};
\footnotesize \legend{\tiny{$p=1$, BPX-all}, \tiny{$p=1$, BPX-supp.}, \tiny{$p=2$, BPX-all}, \tiny{$p=2$, BPX-supp.}, \tiny{$p=3$, BPX-all}, \tiny{$p=3$, BPX-supp.}, \tiny{$p=4$, BPX-all}, \tiny{$p=4$, BPX-supp.}}
  \end{semilogyaxis}
\end{tikzpicture}
}
\caption{Condition numbers for Test 2: THB-splines on non-admissible meshes.} \label{fig:test-non-admissible}

%% file: figures-test3-admclass.tex
\centering
\tikzsetnextfilename{test2_p2_admissible_NP_vs_BPX}
\subfigure[THB-splines on strictly ${\cal T}$-admissible meshes, $p=2$.]{
  \begin{tikzpicture}
  \begin{semilogyaxis}[width=.495\textwidth,height=6.2cm,xlabel={Number of levels},ylabel={Condition number}, y label style={at={(axis description cs:0.03,0.6)}, anchor=north}, xmin=2,xmax=8.5,ymin=1e0,ymax=1e9, grid = major, ytick distance = 1e2, legend columns = 2, anchor = north, legend pos = outer north east, legend style = {at={(0,0.98)}}]
  \addplot[color=black, mark=o, dashed, thick, mark options={solid}, mark size=3pt] table [x={level}, y={NoPrec-THB-Supp}, col sep=comma] {data-HB-THB-square/Results_9x9_HB-THB_Deg2_Adm0_TypeT.csv};
  \addplot[color=black, mark=o, solid, thick, mark options={solid}, mark size=3pt] table [x={level}, y={Gauss-Seidel-THB-Supp}, col sep=comma] {data-HB-THB-square/Results_9x9_HB-THB_Deg2_Adm0_TypeT.csv};
  \addplot[color=blue, mark=square, dashed, thick, mark options={solid}, mark size=3pt] table [x={level}, y={NoPrec-THB-Supp}, col sep=comma] {data-HB-THB-square/Results_9x9_HB-THB_Deg2_Adm4_TypeT.csv};
  \addplot[color=blue, mark=square, solid, thick, mark options={solid}, mark size=3pt] table [x={level}, y={Gauss-Seidel-THB-Supp}, col sep=comma] {data-HB-THB-square/Results_9x9_HB-THB_Deg2_Adm4_TypeT.csv};
  \addplot[color=red, mark=triangle, dashed, thick, mark options={solid}, mark size=3pt] table [x={level}, y={NoPrec-THB-Supp}, col sep=comma] {data-HB-THB-square/Results_9x9_HB-THB_Deg2_Adm3_TypeT.csv};
  \addplot[color=red, mark=triangle, solid, thick, mark options={solid}, mark size=3pt] table [x={level}, y={Gauss-Seidel-THB-Supp}, col sep=comma] {data-HB-THB-square/Results_9x9_HB-THB_Deg2_Adm3_TypeT.csv};
  \addplot[color=green, mark=x, dashed, thick, mark options={solid}, mark size=3pt] table [x={level}, y={NoPrec-THB-Supp}, col sep=comma] {data-HB-THB-square/Results_9x9_HB-THB_Deg2_Adm2_TypeT.csv};
  \addplot[color=green, mark=x, solid, thick, mark options={solid}, mark size=3pt] table [x={level}, y={Gauss-Seidel-THB-Supp}, col sep=comma] {data-HB-THB-square/Results_9x9_HB-THB_Deg2_Adm2_TypeT.csv};
\footnotesize \legend{\tiny{$m=\infty$, No-prec.}, \tiny{$m=\infty$, BPX-supp.}, \tiny{$m=4$, No-prec.}, \tiny{$m=4$, BPX-supp.}, \tiny{$m=3$, No-prec.}, \tiny{$m=3$, BPX-supp.}, \tiny{$m=2$, No-prec.}, \tiny{$m=2$, BPX-supp.}}
  \end{semilogyaxis}
  \end{tikzpicture}
}
\tikzsetnextfilename{test2_p3_admissible_NP_vs_BPX}
\subfigure[THB-splines on strictly ${\cal T}$-admissible meshes, $p=3$.]{
  \begin{tikzpicture}
  \begin{semilogyaxis}[width=.495\textwidth,height=6.2cm,xlabel={Number of levels},ylabel={Condition number}, y label style={at={(axis description cs:0.03,0.6)}, anchor=north}, xmin=2,xmax=8.5,ymin=1e0,ymax=1e9, grid = major, ytick distance = 1e2, legend columns = 2, anchor = north, legend pos = outer north east, legend style = {at={(0,0.98)}}]
  \addplot[color=black, mark=o, dashed, thick, mark options={solid}, mark size=3pt] table [x={level}, y={NoPrec-THB-Supp}, col sep=comma] {data-HB-THB-square/Results_9x9_HB-THB_Deg3_Adm0_TypeT.csv};
  \addplot[color=black, mark=o, solid, thick, mark options={solid}, mark size=3pt] table [x={level}, y={Gauss-Seidel-THB-Supp}, col sep=comma] {data-HB-THB-square/Results_9x9_HB-THB_Deg3_Adm0_TypeT.csv};
  \addplot[color=blue, mark=square, dashed, thick, mark options={solid}, mark size=3pt] table [x={level}, y={NoPrec-THB-Supp}, col sep=comma] {data-HB-THB-square/Results_9x9_HB-THB_Deg3_Adm4_TypeT.csv};
  \addplot[color=blue, mark=square, solid, thick, mark options={solid}, mark size=3pt] table [x={level}, y={Gauss-Seidel-THB-Supp}, col sep=comma] {data-HB-THB-square/Results_9x9_HB-THB_Deg3_Adm4_TypeT.csv};
  \addplot[color=red, mark=triangle, dashed, thick, mark options={solid}, mark size=3pt] table [x={level}, y={NoPrec-THB-Supp}, col sep=comma] {data-HB-THB-square/Results_9x9_HB-THB_Deg3_Adm3_TypeT.csv};
  \addplot[color=red, mark=triangle, solid, thick, mark options={solid}, mark size=3pt] table [x={level}, y={Gauss-Seidel-THB-Supp}, col sep=comma] {data-HB-THB-square/Results_9x9_HB-THB_Deg3_Adm3_TypeT.csv};
  \addplot[color=green, mark=x, dashed, thick, mark options={solid}, mark size=3pt] table [x={level}, y={NoPrec-THB-Supp}, col sep=comma] {data-HB-THB-square/Results_9x9_HB-THB_Deg3_Adm2_TypeT.csv};
  \addplot[color=green, mark=x, solid, thick, mark options={solid}, mark size=3pt] table [x={level}, y={Gauss-Seidel-THB-Supp}, col sep=comma] {data-HB-THB-square/Results_9x9_HB-THB_Deg3_Adm2_TypeT.csv};
\footnotesize \legend{\tiny{$m=\infty$, No-prec.}, \tiny{$m=\infty$, BPX-supp.}, \tiny{$m=4$, No-prec.}, \tiny{$m=4$, BPX-supp.}, \tiny{$m=3$, No-prec.}, \tiny{$m=3$, BPX-supp.}, \tiny{$m=2$, No-prec.}, \tiny{$m=2$, BPX-supp.}}
  \end{semilogyaxis}
  \end{tikzpicture}
}
\caption{Condition numbers for Test 3: dependence on the admissibility class $m$.} \label{fig:test2-p2-p3}

%% file: figures-test4-HB-THB.tex
\centering
\tikzsetnextfilename{test3_HB-THB-Hadmissible}
\subfigure[Comparison of HB-splines and THB-splines on strictly ${\cal H}$-admissible meshes, $m=3$.]{
  \begin{tikzpicture}
  \begin{semilogyaxis}[width=.495\textwidth,height=6.2cm,xlabel={Number of levels},ylabel={Condition number}, y label style={at={(axis description cs:0.03,0.6)}, anchor=north}, xmin=2,xmax=8.5,ymin=1e0,ymax=1e6, grid = major, ytick distance = 1e2, legend columns = 2, anchor = north, legend pos = outer north east, legend style = {at={(0,0.98)}}]
  \addplot[color=black, mark=o, dashed, thick, mark options={solid}, mark size=3pt] table [x={level}, y={Gauss-Seidel-HB-Supp}, col sep=comma] {data-HB-THB-square/Results_9x9_HB-THB_Deg1_Adm3_TypeH.csv};
  \addplot[color=black, mark=o, solid, thick, mark options={solid}, mark size=3pt] table [x={level}, y={Gauss-Seidel-THB-Supp}, col sep=comma] {data-HB-THB-square/Results_9x9_HB-THB_Deg1_Adm3_TypeH.csv};
  \addplot[color=blue, mark=square, dashed, thick, mark options={solid}, mark size=3pt] table [x={level}, y={Gauss-Seidel-HB-Supp}, col sep=comma] {data-HB-THB-square/Results_9x9_HB-THB_Deg2_Adm3_TypeH.csv};
  \addplot[color=blue, mark=square, solid, thick, mark options={solid}, mark size=3pt] table [x={level}, y={Gauss-Seidel-THB-Supp}, col sep=comma] {data-HB-THB-square/Results_9x9_HB-THB_Deg2_Adm3_TypeH.csv};
  \addplot[color=red, mark=triangle, dashed, thick, mark options={solid}, mark size=3pt] table [x={level}, y={Gauss-Seidel-HB-Supp}, col sep=comma] {data-HB-THB-square/Results_9x9_HB-THB_Deg3_Adm3_TypeH.csv};
  \addplot[color=red, mark=triangle, solid, thick, mark options={solid}, mark size=3pt] table [x={level}, y={Gauss-Seidel-THB-Supp}, col sep=comma] {data-HB-THB-square/Results_9x9_HB-THB_Deg3_Adm3_TypeH.csv};
  \addplot[color=green, mark=x, dashed, thick, mark options={solid}, mark size=3pt] table [x={level}, y={Gauss-Seidel-HB-Supp}, col sep=comma] {data-HB-THB-square/Results_9x9_HB-THB_Deg4_Adm3_TypeH.csv};
  \addplot[color=green, mark=x, solid, thick, mark options={solid}, mark size=3pt] table [x={level}, y={Gauss-Seidel-THB-Supp}, col sep=comma] {data-HB-THB-square/Results_9x9_HB-THB_Deg4_Adm3_TypeH.csv};
\footnotesize \legend{\tiny{$p=1$, HB-splines}, \tiny{$p=1$, THB-splines}, \tiny{$p=2$, HB-splines}, \tiny{$p=2$, THB-splines}, \tiny{$p=3$, HB-splines}, \tiny{$p=3$, THB-splines}, \tiny{$p=4$, HB-splines}, \tiny{$p=4$, THB-splines}}
  \end{semilogyaxis}
  \end{tikzpicture}
\label{fig:test3a}
}
\tikzsetnextfilename{test3_HB_T-H-admissible}
\subfigure[Results for HB-splines on strictly ${\cal H}$-admissible and strictly ${\cal T}$-admissible meshes, $p=3$.]{
  \begin{tikzpicture}
  \begin{semilogyaxis}[width=.495\textwidth,height=6.2cm,xlabel={Number of levels},ylabel={Condition number}, y label style={at={(axis description cs:0.03,0.6)}, anchor=north}, xmin=2,xmax=8.5,ymin=5e0,ymax=3e3, grid = major, ytick distance = 1e2, legend columns = 2, anchor = north, legend pos = outer north east, legend style = {at={(0,0.98)}}]
  \addplot[color=green, mark=x, dashed, thick, mark options={solid}, mark size=3pt] table [x={level}, y={Gauss-Seidel-HB-Supp}, col sep=comma] {data-HB-THB-square/Results_9x9_HB-THB_Deg3_Adm2_TypeH.csv};
  \addplot[color=green, mark=x, solid, thick, mark options={solid}, mark size=3pt] table [x={level}, y={Gauss-Seidel-HB-Supp}, col sep=comma] {data-HB-THB-square/Results_9x9_HB-THB_Deg3_Adm2_TypeT.csv};
  \addplot[color=red, mark=triangle, dashed, thick, mark options={solid}, mark size=3pt] table [x={level}, y={Gauss-Seidel-HB-Supp}, col sep=comma] {data-HB-THB-square/Results_9x9_HB-THB_Deg3_Adm3_TypeH.csv};
  \addplot[color=red, mark=triangle, solid, thick, mark options={solid}, mark size=3pt] table [x={level}, y={Gauss-Seidel-HB-Supp}, col sep=comma] {data-HB-THB-square/Results_9x9_HB-THB_Deg3_Adm3_TypeT.csv};
  \addplot[color=blue, mark=square, dashed, thick, mark options={solid}, mark size=3pt] table [x={level}, y={Gauss-Seidel-HB-Supp}, col sep=comma] {data-HB-THB-square/Results_9x9_HB-THB_Deg3_Adm4_TypeH.csv};
  \addplot[color=blue, mark=square, solid, thick, mark options={solid}, mark size=3pt] table [x={level}, y={Gauss-Seidel-HB-Supp}, col sep=comma] {data-HB-THB-square/Results_9x9_HB-THB_Deg3_Adm4_TypeT.csv};
  \addplot[color=black, mark=o, dashed, thick, mark options={solid}, mark size=3pt] table [x={level}, y={Gauss-Seidel-HB-Supp}, col sep=comma] {data-HB-THB-square/Results_9x9_HB-THB_Deg3_Adm0_TypeH.csv};
\footnotesize \legend{\tiny{$m=2$, ${\cal H}$-adm.}, \tiny{$m=2$, ${\cal T}$-adm.}, \tiny{$m=3$, ${\cal H}$-adm.}, \tiny{$m=3$, ${\cal T}$-adm.}, \tiny{$m=4$, ${\cal H}$-adm.}, \tiny{$m=4$, ${\cal T}$-adm.}, \tiny{$m=\infty$}}
  \end{semilogyaxis}
  \end{tikzpicture}
\label{fig:test3b}
}
\caption{Condition numbers for Test 4: HB-splines and THB-splines.} \label{fig:test3}

%% file: figures-results-curvedL.tex
\centering
\subfigure[Strictly ${\cal T}$-admissible meshes, $m=2$.]{
\tikzsetnextfilename{test_curvedL_m2_MP_vs_BPX}
  \begin{tikzpicture}
  \begin{semilogyaxis}[width=.495\textwidth,height=6.2cm,xlabel={Number of levels},ylabel={Condition number}, y label style={at={(axis description cs:0.03,0.6)}, anchor=north}, xmin=2,xmax=11.5,ymin=1e0,ymax=1e6, grid = major, ytick distance = 1e2, legend columns = 2, anchor = north, legend pos = outer north east, legend style = {at={(0,0.98)}}]
  \addplot[color=blue, mark=o, dashed, thick, mark options={solid}, mark size=3pt] table [x={level}, y={NoPrec-THB-Supp}, col sep=comma] {data-curvedL/Results_curvedL_Deg2_Adm2_TypeT.csv};
  \addplot[color=blue, mark=o, solid, thick, mark options={solid}, mark size=3pt] table [x={level}, y={Gauss-Seidel-THB-Supp}, col sep=comma] {data-curvedL/Results_curvedL_Deg2_Adm2_TypeT.csv};
  \addplot[color=red, mark=square, dashed, thick, mark options={solid}, mark size=3pt] table [x={level}, y={NoPrec-THB-Supp}, col sep=comma] {data-curvedL/Results_curvedL_Deg3_Adm2_TypeT.csv};
  \addplot[color=red, mark=square, solid, thick, mark options={solid}, mark size=3pt] table [x={level}, y={Gauss-Seidel-THB-Supp}, col sep=comma] {data-curvedL/Results_curvedL_Deg3_Adm2_TypeT.csv};
  \addplot[color=green, mark=triangle, dashed, thick, mark options={solid}, mark size=3pt] table [x={level}, y={NoPrec-THB-Supp}, col sep=comma] {data-curvedL/Results_curvedL_Deg4_Adm2_TypeT.csv};
  \addplot[color=green, mark=triangle, solid, thick, mark options={solid}, mark size=3pt] table [x={level}, y={Gauss-Seidel-THB-Supp}, col sep=comma] {data-curvedL/Results_curvedL_Deg4_Adm2_TypeT.csv};
\footnotesize \legend{\tiny{$p=2$, No-prec.}, \tiny{$p=2$, BPX-supp}, \tiny{$p=3$, No-prec.}, \tiny{$p=3$, BPX-supp}, \tiny{$p=4$, No-prec.}, \tiny{$p=4$, BPX-supp}}
  \end{semilogyaxis}
  \end{tikzpicture}
}
\subfigure[Non-admissible meshes.]{
\tikzsetnextfilename{test_curvedL_m0_MP_vs_BPX}
  \begin{tikzpicture}
  \begin{semilogyaxis}[width=.495\textwidth,height=6.2cm,xlabel={Number of levels},ylabel={Condition number}, y label style={at={(axis description cs:0.03,0.6)}, anchor=north}, xmin=2,xmax=11.5,ymin=1e0,ymax=1e6, grid = major, ytick distance = 1e2, legend columns = 2, anchor = north, legend pos = outer north east, legend style = {at={(0,0.98)}}]
  \addplot[color=blue, mark=o, dashed, thick, mark options={solid}, mark size=3pt] table [x={level}, y={NoPrec-THB-Supp}, col sep=comma] {data-curvedL/Results_curvedL_Deg2_Adm0_TypeT.csv};
  \addplot[color=blue, mark=o, solid, thick, mark options={solid}, mark size=3pt] table [x={level}, y={Gauss-Seidel-THB-Supp}, col sep=comma] {data-curvedL/Results_curvedL_Deg2_Adm0_TypeT.csv};
  \addplot[color=red, mark=square, dashed, thick, mark options={solid}, mark size=3pt] table [x={level}, y={NoPrec-THB-Supp}, col sep=comma] {data-curvedL/Results_curvedL_Deg3_Adm0_TypeT.csv};
  \addplot[color=red, mark=square, solid, thick, mark options={solid}, mark size=3pt] table [x={level}, y={Gauss-Seidel-THB-Supp}, col sep=comma] {data-curvedL/Results_curvedL_Deg3_Adm0_TypeT.csv};
  \addplot[color=green, mark=triangle, dashed, thick, mark options={solid}, mark size=3pt] table [x={level}, y={NoPrec-THB-Supp}, col sep=comma] {data-curvedL/Results_curvedL_Deg4_Adm0_TypeT.csv};
  \addplot[color=green, mark=triangle, solid, thick, mark options={solid}, mark size=3pt] table [x={level}, y={Gauss-Seidel-THB-Supp}, col sep=comma] {data-curvedL/Results_curvedL_Deg4_Adm0_TypeT.csv};
\footnotesize \legend{\tiny{$p=2$, No-prec.}, \tiny{$p=2$, BPX-supp}, \tiny{$p=3$, No-prec.}, \tiny{$p=3$, BPX-supp}, \tiny{$p=4$, No-prec.}, \tiny{$p=4$, BPX-supp}}
  \end{semilogyaxis}
  \end{tikzpicture}
}

\caption{Condition numbers for Test 5: THB-splines with adaptive refinement.} \label{fig:test-Lshaped}

%% file: ms.bbl
\def\cprime{$'$}
\begin{thebibliography}{10}

\bibitem{Bazilevs_Beirao_Cottrell_Hughes_Sangalli}
{\sc Y.~Bazilevs, L.~Beir{\~a}o~da Veiga, J.~A. Cottrell, T.~J.~R. Hughes, and
  G.~Sangalli}, {\em Isogeometric analysis: approximation, stability and error
  estimates for {$h$}-refined meshes}, Math. Models Methods Appl. Sci., 16
  (2006), pp.~1031--1090.

\bibitem{IGA-acta}
{\sc L.~Beir{\~a}o~da Veiga, A.~Buffa, G.~Sangalli, and R.~V{\'a}zquez}, {\em
  Mathematical analysis of variational isogeometric methods}, Acta Numer., 23
  (2014), pp.~157--287.

\bibitem{bracco2018b}
{\sc C.~Bracco, C.~Giannelli, and R.~V\'azquez}, {\em {Refinement algorithms
  for adaptive isogeometric methods with hierarchical splines}}, Axioms, 7(3)
  (2018), p.~43.

\bibitem{BPX1}
{\sc J.~H. Bramble, J.~E. Pasciak, and J.~Xu}, {\em Parallel multilevel
  preconditioners}, Math. Comp., 55 (1990), pp.~1--22.

\bibitem{multigrid_tutorial}
{\sc W.~L. Briggs, V.~E. Henson, and S.~F. McCormick}, {\em A multigrid
  tutorial}, Society for Industrial and Applied Mathematics (SIAM),
  Philadelphia, PA, second~ed., 2000.

\bibitem{buffa2016b}
{\sc A.~Buffa, E.~M. Garau, C.~Giannelli, and G.~Sangalli}, {\em On
  quasi-interpolation operators in spline spaces}, in Building Bridges:
  Connections and Challenges in Modern Approaches to Numerical Partial
  Differential Equations, G.~R. Barrenechea et~al., eds., vol.~114, Lecture
  Notes in Computational Science and Engineering, 2016, pp.~73--91.

\bibitem{BC2016}
{\sc A.~Buffa and C.~Giannelli}, {\em Adaptive isogeometric methods with
  hierarchical splines: Error estimator and convergence}, Math. Models Methods
  Appl. Sci., 26 (2016), pp.~1--25.

\bibitem{BC2017}
\leavevmode\vrule height 2pt depth -1.6pt width 23pt, {\em Adaptive
  isogeometric methods with hierarchical splines: Optimality and convergence
  rates}, Math. Models Methods Appl. Sci., 27 (2017), pp.~2781--2802.

\bibitem{BC2019}
\leavevmode\vrule height 2pt depth -1.6pt width 23pt, {\em Remarks on
  {P}oincar\'e and interpolation estimates for truncated hierarchical
  {B}-splines}, Submitted,  (2019).

\bibitem{BHKS13}
{\sc A.~Buffa, H.~Harbrecht, A.~Kunoth, and G.~Sangalli}, {\em
  {BPX}-preconditioning for isogeometric analysis}, Comput. Methods Appl. Mech.
  Engrg., 265 (2013), pp.~63 -- 70.

\bibitem{CGRV19}
{\sc M.~Carraturo, C.~Giannelli, A.~Reali, and R.~V\'{a}zquez}, {\em Suitably
  graded {THB}-spline refinement and coarsening: towards an adaptive
  isogeometric analysis of additive manufacturing processes}, Comput. Methods
  Appl. Mech. Engrg., 348 (2019), pp.~660--679.

\bibitem{CNX}
{\sc L.~Chen, R.~H. Nochetto, and J.~Xu}, {\em Optimal multilevel methods for
  graded bisection grids}, Numer. Math., 120 (2012), pp.~1--34.

\bibitem{CV19}
{\sc D.~Cho and R.~V\'azquez}, {\em {BPX preconditioners for isogeometric
  analysis using analysis-suitable T-splines}}, IMA J. Numer. Anal.,  (2019).
\newblock DOI:10.1093/imanum/dry032.

\bibitem{DeBoor}
{\sc C.~de~Boor}, {\em A practical guide to splines}, vol.~27 of Applied
  Mathematical Sciences, Springer-Verlag, New York, revised~ed., 2001.

\bibitem{dePrenter2019}
{\sc F.~de~Prenter, C.~V. Verhoosel, E.~H. van Brummelen, J.~A. Evans,
  C.~Messe, J.~Benzaken, and K.~Maute}, {\em Multigrid solvers for immersed
  finite element methods and immersed isogeometric analysis}, Comput. Mech.,
  (2019).
\newblock DOI:10.1007/s00466-019-01796-y.

\bibitem{Manni_MG}
{\sc M.~Donatelli, C.~Garoni, C.~Manni, S.~Serra-Capizzano, and H.~Speleers},
  {\em Symbol-{B}ased {M}ultigrid {M}ethods for {G}alerkin {B}-{S}pline
  {I}sogeometric {A}nalysis}, SIAM J. Numer. Anal., 55 (2017), pp.~31--62.

\bibitem{FUHRER2019}
{\sc T.~F\"uhrer, G.~Gantner, D.~Praetorius, and S.~Schimanko}, {\em {Optimal
  additive Schwarz preconditioning for adaptive 2D IGA boundary element
  methods}}, Comput. Methods Appl. Mech. Engrg., 351 (2019), pp.~571--598.

\bibitem{Gahalaut_MG}
{\sc K.~Gahalaut, J.~Kraus, and S.~Tomar}, {\em Multigrid methods for
  isogeometric discretization}, Comput. Methods Appl. Mech. Engrg., 253 (2013),
  pp.~413 -- 425.

\bibitem{Gahalaut_AMG}
{\sc K.~Gahalaut, S.~Tomar, and J.~Kraus}, {\em Algebraic multilevel
  preconditioning in isogeometric analysis: Construction and numerical
  studies}, Comput. Methods Appl. Mech. Engrg., 266 (2013), pp.~40 -- 56.

\bibitem{GHP17}
{\sc G.~Gantner, D.~Haberlik, and D.~Praetorius}, {\em Adaptive
  {I}{G}{A}{F}{E}{M} with optimal convergence rates: {H}ierarchical
  {B}-splines}, Math. Models Methods Appl. Sci., 27 (2017), pp.~2631--2674.

\bibitem{garau2018}
{\sc E.~Garau and R.~V\'{a}zquez}, {\em Algorithms for the implementation of
  adaptive isogeometric methods using hierarchical {B}-splines}, Appl. Numer.
  Math., 123 (2018), pp.~58--87.

\bibitem{Giannelli2016337}
{\sc C.~Giannelli, B.~J{\"u}ttler, S.~K. Kleiss, A.~Mantzaflaris, B.~Simeon,
  and J.~\v{S}peh}, {\em {THB}-splines: An effective mathematical technology
  for adaptive refinement in geometric design and isogeometric analysis},
  Comput. Methods Appl. Mech. Engrg., 299 (2016), pp.~337 -- 365.

\bibitem{Giannelli2012485}
{\sc C.~Giannelli, B.~J{\"u}ttler, and H.~Speleers}, {\em {THB}-splines: The
  truncated basis for hierarchical splines}, Comput. Aided Geom. Design, 29
  (2012), pp.~485 -- 498.

\bibitem{GJS14}
{\sc C.~Giannelli, B.~J{\"u}ttler, and H.~Speleers}, {\em Strongly stable bases
  for adaptively refined multilevel spline spaces}, Adv. Comput. Math., 40
  (2014), pp.~459--490.

\bibitem{hennig2018}
{\sc P.~Hennig, M.~Ambati, L.~{De Lorenzis}, and M.~K{\"a}stner}, {\em
  Projection and transfer operators in adaptive isogeometric analysis with
  hierarchical b-splines}, Comput. Methods Appl. Mech. Engrg., 334 (2018),
  pp.~313 -- 336.

\bibitem{Hofreither2016b}
{\sc C.~Hofreither, B.~J\"{u}ttler, G.~Kiss, and W.~Zulehner}, {\em Multigrid
  methods for isogeometric analysis with {THB}-splines}, Comput. Methods Appl.
  Mech. Engrg., 308 (2016), pp.~96--112.

\bibitem{Hofreither_SINUM2017}
{\sc C.~Hofreither and S.~Takacs}, {\em Robust multigrid for isogeometric
  analysis based on stable splittings of spline spaces}, SIAM J. Numer. Anal.,
  55 (2017), pp.~2004--2024.

\bibitem{Hofreither2016}
{\sc C.~Hofreither, S.~Takacs, and W.~Zulehner}, {\em A robust multigrid method
  for isogeometric analysis in two dimensions using boundary correction},
  Comput. Methods Appl. Mech. Engrg., 316 (2017), pp.~22--42.

\bibitem{Hughes_Cottrell_Bazilevs}
{\sc T.~J.~R. Hughes, J.~A. Cottrell, and Y.~Bazilevs}, {\em Isogeometric
  analysis: {CAD}, finite elements, {NURBS}, exact geometry and mesh
  refinement}, Comput. Methods Appl. Mech. Engrg., 194 (2005), pp.~4135--4195.

\bibitem{Kraft}
{\sc R.~Kraft}, {\em Adaptive and linearly independent multilevel
  {$B$}-splines}, in Surface fitting and multiresolution methods
  ({C}hamonix--{M}ont-{B}lanc, 1996), Vanderbilt Univ. Press, Nashville, TN,
  1997, pp.~209--218.

\bibitem{Kuru2013}
{\sc G.~Kuru, C.~Verhoosel, K.~van~der Zee, and E.~van Brummelen}, {\em
  Goal-adaptive isogeometric analysis with hierarchical splines}, Comput.
  Methods Appl. Mech. and Engrg., 270 (2014), pp.~270--292.

\bibitem{RRG19}
{\sc A.~P{\'e}~de~la Riva, C.~Rodrigo, and F.~J. Gaspar}, {\em A {R}obust
  {M}ultigrid {S}olver for {I}sogeometric {A}nalysis {B}ased on
  {M}ultiplicative {S}chwarz {S}moothers}, SIAM J. Sci. Comput., 41 (2019),
  pp.~S321--S345.

\bibitem{Saad_book}
{\sc Y.~Saad}, {\em Iterative methods for sparse linear systems}, Society for
  Industrial and Applied Mathematics, Philadelphia, PA, second~ed., 2003.

\bibitem{Schumi}
{\sc L.~L. Schumaker}, {\em Spline functions: basic theory}, Cambridge
  Mathematical Library, Cambridge University Press, Cambridge, third~ed., 2007.

\bibitem{MS15}
{\sc H.~Speleers and C.~Manni}, {\em Effortless quasi-interpolation in
  hierarchical spaces}, Numer. Math.,  (2015), pp.~1--30.

\bibitem{Panayot}
{\sc P.~S. Vassilevski}, {\em Multilevel block factorization preconditioners},
  Springer, New York, 2008.
\newblock Matrix-based analysis and algorithms for solving finite element
  equations.

\bibitem{GEOPDES-NEW}
{\sc R.~V{\'a}zquez}, {\em A new design for the implementation of isogeometric
  analysis in {O}ctave and {M}atlab: {G}eo{PDE}s 3.0}, Comput. Math. Appl., 72
  (2016), pp.~523 -- 554.

\bibitem{Vuong_giannelli_juttler_simeon}
{\sc A.-V. Vuong, C.~Giannelli, B.~J\"uttler, and B.~Simeon}, {\em A
  hierarchical approach to adaptive local refinement in isogeometric analysis},
  Comput. Methods Appl. Mech. Engrg., 200 (2011), pp.~3554--3567.

\bibitem{WuChen06}
{\sc H.~Wu and Z.~Chen}, {\em Uniform convergence of multigrid {V}-cycle on
  adaptively refined finite element meshes for second order elliptic problems},
  Sci. China Ser. A, 49 (2006), pp.~1405--1429.

\bibitem{JXu_SIAM_Review}
{\sc J.~Xu}, {\em Iterative methods by space decomposition and subspace
  correction}, SIAM Rev., 34 (1992), pp.~581--613.

\bibitem{XCN}
{\sc J.~Xu, L.~Chen, and R.~H. Nochetto}, {\em Optimal multilevel methods for
  {$H({\rm grad})$}, {$H({\rm curl})$}, and {$H({\rm div})$} systems on graded
  and unstructured grids}, in Multiscale, nonlinear and adaptive approximation,
  Springer, Berlin, 2009, pp.~599--659.

\bibitem{Xu_Zikatanov_2002}
{\sc J.~Xu and L.~Zikatanov}, {\em The method of alternating projections and
  the method of subspace corrections in {H}ilbert space}, J. Amer. Math. Soc.,
  15 (2002), pp.~573--597 (electronic).

\bibitem{XCH10}
{\sc X.~Xu, H.~Chen, and R.~H.~W. Hoppe}, {\em Optimality of local multilevel
  methods on adaptively refined meshes for elliptic boundary value problems},
  J. Numer. Math., 18 (2010), pp.~59--90.

\end{thebibliography}
